\documentclass[11pt,a4paper]{amsart}
\usepackage{amsfonts}
\usepackage{amsthm}
\usepackage{amsmath}
\usepackage{amscd}
\usepackage[latin2]{inputenc}
\usepackage{t1enc}
\usepackage[mathscr]{eucal}
\usepackage{indentfirst}
\usepackage{graphicx}
\usepackage{graphics}
\usepackage{pict2e}
\usepackage{epic}
\numberwithin{equation}{section}
\usepackage[margin=2.9cm]{geometry}
\usepackage{epstopdf}

\RequirePackage[OT1]{fontenc}
\RequirePackage{amsthm,amsmath}
\RequirePackage[numbers]{natbib}
\RequirePackage[colorlinks,citecolor=blue,urlcolor=blue]{hyperref}

\usepackage{mathrsfs,algorithmicx,algorithm,bbm,bm,amssymb,xcolor}
\usepackage{ulem} %nyb-110118
\usepackage{longtable} %nyb-12162018
\allowdisplaybreaks %nyb-01022019
\usepackage{framed}
\usepackage{soul}
\usepackage{enumitem}
\usepackage{tikz}
\usetikzlibrary{shapes,snakes, bayesnet}
\makeatletter
\renewcommand*\env@matrix[1][*\c@MaxMatrixCols c]{%
	\hskip -\arraycolsep
	\let\@ifnextchar\new@ifnextchar
	\array{#1}}
\makeatother

\numberwithin{equation}{section}
\theoremstyle{plain}
\newtheorem{theorem}{Theorem}[section]
\newtheorem{definition}{Definition}[section]

\newtheorem{corollary}{Corollary}[section]
\newtheorem{lemma}{Lemma}[section]
\newtheorem{assumption}{Assumption}[section]
\newtheorem{remark}{Remark}[section]
\newtheorem{proposition}{Proposition}[section]

\newcommand{\abs}[1]{\left\vert#1\right\vert}

\def\ind{\mathbbm{1}}

\newcommand \bbP{\mathbb{P}}
\newcommand \bbE{\mathbb{E}}

\begin{document}
\setstcolor{red}

\title{Bayesian Graph Selection Consistency Under Model Misspecification}
\author[Yabo Niu, Debdeep Pati, Bani K. Mallick]{Yabo Niu, Debdeep Pati, Bani K. Mallick}
\address{Department of Statistics, Texas A\&M University} 
\email{ybniu@stat.tamu.edu, debdeep@stat.tamu.edu, bmallick@stat.tamu.edu}
\subjclass[2010]{Primary: 62F15. Secondary: 60K35}
\keywords{Bayesian, consistency, decomposable, graph selection, hyper inverse Wishart, marginal likelihood}

\begin{abstract} Gaussian graphical models are a popular tool to learn the dependence structure in the form of a graph among variables of interest.    Bayesian methods have gained in popularity in the last two decades due to their ability to simultaneously learn the covariance and the graph and characterize uncertainty in the selection. For scalability of the Markov chain Monte Carlo algorithms, decomposability is commonly imposed on the graph space. A wide variety of graphical conjugate priors are proposed jointly on the covariance matrix and the graph with improved algorithms to search along the space of decomposable graphs, rendering the methods extremely popular in the context of multivariate dependence modeling.
%%
%\st{An open problem in Bayesian decomposable structure learning is whether the posterior distribution is able to select the true graph asymptotically when the dimension of the variables increases with the sample size.}
%%
{\it An open problem} in Bayesian decomposable structure learning is whether the posterior distribution is able to select a meaningful decomposable graph that it is  ``close'' in an appropriate sense to the true non-decomposable graph, when the dimension of the variables increases with the sample size. 
%The selection consistency under model misspecification is characterized by the class of minimal triangulations of the true graph.}
%
In this article, we explore specific conditions on the true precision matrix and the graph which results in an affirmative answer to this question using a commonly used hyper-inverse Wishart prior on the covariance matrix and a suitable complexity prior on the graph space, both in the well-specified and misspecified settings. In absence of structural sparsity assumptions, our strong selection consistency holds in a high dimensional setting where $p = O(n^{\alpha})$ for  $\alpha < 1/3$. 
We show when the true graph is non-decomposable, the posterior distribution on the graph concentrates on a set of graphs that are {\it minimal triangulations} of the true graph.
\end{abstract}

\maketitle

\section{Introduction}

Graphical models provide a framework for describing statistical dependencies in (possibly large) collections of random variables \citep{lauritzen1996graphical}. In this article, we revisit the well known problem of inference on the underlying graph from observed data from a Bayesian point of view.   Research on Bayesian inference for natural exponential families and associated conjugate priors (DY priors) is pioneered by  \cite{diaconis1979conjugate} and has profound impact on the development of Bayesian Gaussian graphical models.  Consider independent and identically distributed vectors $Y_1, Y_2, \ldots, Y_n$ drawn from  $p$-variate normal distribution with mean vector $0$ and a sparse inverse covariance matrix $\Omega$.   The sparsity pattern in $\Omega$ can be encoded in terms of a graph $G$ on the set of variables as follows. If the variables $i$ and $j$ do not share an edge in $G$, then $\Omega_{ij}=0$. Hence, an undirected (or concentration) graphical model corresponding to $G$ restricts the inverse covariance matrix $\Omega$ to a linear subspace of the cone of positive definite matrices. 

A probabilistic framework for learning the dependence structure and the graph $G$ requires specification of a prior distribution for $(\Omega, G)$.  Conditional on $G$, a hyper-inverse Wishart distribution \citep{dawid1993hyper} on $\Sigma = \Omega^{-1}$  and the corresponding induced class of distributions on $\Omega$ \citep{roverato2000cholesky} are attractive choices of DY priors. %%\cite{dawid1993hyper} introduced a class of prior distributions for $\Sigma = \Omega^{-1}$ called hyper-inverse Wishart distributions. The induced class of prior distributions for $\Omega$  is known as the class of G-Wishart distributions (see \cite{roverato2000cholesky}). This class of prior distributions is quite useful and popular, and has several desirable properties, including the fact that it corresponds to the Diaconis- Ylvisaker class of conjugate priors for the concentration graph model corresponding to the graph $G$. 
A rich family of conjugate priors that subsumes the DY class is developed by \cite{letac2007wishart}. 
%%Both the Hyper-inverse Wishart priors and the ``Letac-Massam" priors have attractive properties which enable Bayesian inference.  
Bayesian procedures corresponding to these Letac-Massam priors have been derived in a decision theoretic framework in the recent work of \cite{rajaratnam2008flexible}.  The key component of Bayesian structure learning is achieved through  specification of a prior distribution on the space of graphs. There is a need for a flexible but tractable family of such priors, capable of representing a variety of prior beliefs about the conditional independence structure. In the interests of tractability and scalability, there has been a strong focus on the case where the true graph may be assumed to be decomposable.   On the other hand, relatively few papers have considered non decomposable graphs in a Bayesian set-up; refer to HIW distributions for non-decomposable graphs \cite{roverato2002hyper,atay2005monte,dellaportas2003bayesian,moghaddam2009accelerating,wang2010simulation,khare2018bayesian}. 

%Just as this underlying graph localizes the pattern of dependence among variables, it is appealing that the prior on the graph itself should exhibit dependence locally. Informally, the presence or absence of two edges should be independent when they are sufficiently separated by other edges in the graph. The first class of graph priors demonstrating such a structural Markov property was presented in a 2012 Cambridge University PhD thesis by Simon Byrne, and
%later published in \cite{byrne2015structural}.

In this paper, we focus on the HIW distribution for decomposable graphs as this construction enjoys many advantages, such as computational efficiency due to its conjugate formulation and exact calculation of marginal likelihoods \citep{scott2008feature}. 
%Furthermore, this prior can impose sparsity on the graph \citep{giudici1996learning} which is perfect for high dimensional graphs.
The use of HIW prior within a Bayesian framework for Gaussian graphical models has been well studied for the past decade, see \cite{giudici1996learning,giudici1999decomposable,carvalho2007simulation,carvalho2009objective}.  Although deemed as a restrictive model choice in the space of graphs, as long as the model for the data allows arbitrarily small interactions, the resulting model assuming decomposability is quite flexible. 
%Indeed, the HIW distributions for non-decomposable graphs exist and many studies have been done in that scenario (Roverato (2002), Atay-Kayis & Massam (2005), Wang et al. (2010)). 
Stochastic search algorithms are empirically demonstrated to have good practical performance in these models.  For detailed description and comparison of various Bayesian computation methods in this scenario, see \cite{jones2005experiments,donnet2012empirical}.  

There has been a growing literature on model selection consistency in Gaussian graphical models from a frequentist point of view \citep{rask,meinshausen2006high,yuan2007model,drton2007multiple}. Beyond the literature on Gaussian graphical models, there has been a incredible amount of frequentist work in the context of estimating high-dimensional covariance matrix estimation with rates of convergence of various regularized covariance estimators derived in \cite{bickel2008regularized,lam2009sparsistency,el2008operator,cai2011adaptive} among others. There is a relatively smaller literature on asymptotic properties of Bayesian procedures for covariance or precision matrices in graphical models; refer to \cite{banerjee2014posterior,banerjee2015bayesian}. However, the literature on graph selection consistency in a Bayesian paradigm is surprisingly sparse.  In the context of decomposable graphs, the only article we were aware of is \cite{fitch2014performance} who considered the behavior of Bayesian procedures that perform model selection for decomposable Gaussian graphical models.  However, the analysis is restricted to the fixed dimensional regime and involves the behavior of the marginal likelihood ratios between graphs differing by an edge.   For general graph selection consistency within a Bayesian framework, refer to the very recent article \cite{cao2016posterior} in the context of Gaussian directed acyclic graph (DAG) models. 
%
%\st{Despite the popularity of using HIW prior for decomposable graphs and its associated posterior computation techniques over the past 20 years, there is a clear dearth of theoretical results for such methods.}
%
The question of validity of using decomposable graphical models using the HIW prior when the true graph is in fact non-decomposable is unanswered till date despite its popularity and development of associated posterior computation techniques over the past 20 years.  

In this article, focusing on the hyper-inverse g-Wishart (g-HIW) distribution on the covariance matrix and a complexity prior on the graph, we derive sufficient conditions for strong selection consistency when $p = O(n^{\alpha})$ with $\alpha < 1/3$. The key conditions relate to precise upper and lower bounds on the partial correlation and a suitably complexity prior on the space of graphs.  We emphasize here that we do {\it not} need conditions to be verified on all subgraphs -  all assumptions are easy to understand and relatively straightforward to verify.   Regarding our findings, we discover that g-HIW prior  places heavy penalty on missing true edges (false negatives), but comparatively smaller penalty on adding false edges (false positives). Henceforth in high-dimensional regime a carefully chosen complexity prior on the graph space is needed for penalizing false positives and achieving strong consistency. 
%
%\st{To the best of our knowledge, this is the first paper to show the strong selection consistency under HIW prior for high-dimensional decomposable graphs. In addition, we develop selection consistency results when the true graph is non-decomposable.} 
%

In the well-specified case, the hierarchical model used here is a subset of \cite{cao2016posterior} since hyper-inverse Wishart prior is a special case of DAG-Wishart prior proposed in \cite{ben2011high} under {\it perfect} DAGs. However, the assumptions in this paper are distinctly different from those stated in \cite{cao2016posterior}. In particular, our assumptions are on the magnitude of the elements of partial correlation matrix rather than on the eigen values of covariance matrix as in \cite{cao2016posterior}.  Also, the main focus of this article is to study the behavior of graph selection consistency under model misspecification, which cannot be addressed within a DAG framework. To the best of our knowledge, this is the first paper to show the strong selection consistency under HIW prior for high-dimensional graphs under model misspecification. 
In particular, we show that the posterior concentrates on decomposable graphs which are in some sense closest to the true non-decomposable graph. Interestingly, the pairwise Bayes factors between such graphs are stochastically bounded.  Our result under model-misspecification is inspired by \cite{fitch2014performance}, but extends to the case when $p$ is growing with $n$ and provides a rigorous proof the convergence of the posterior distribution to the class of decomposable graphs which are closest to the true one. We also present a detailed simulation study both for the well-specified and misspecified case, which provides empirical justification for some of our technical results.  

{\it En-route}, we develop precise bounds for Bayes factor in favor of an alternative graph with respect to the true graph.  The main proof technique is a combination of a) {\em localization:} which involves breaking down the Bayes factor between any two graphs into local moves, i.e. addition and deletion of one edge  using decomposable graph chain rule and b) {\em correlation association:} which converts the Bayes factor between two graphs differing by an edge into a suitable function of sample partial correlations. By developing sharp concentration and tail bounds for sample partial correlation, we obtain bounds for ratios of local marginal likelihoods which are then combined to yield strong selection consistency results.

The remaining part of the paper is organized as follows.  In \S \ref{sec:prelim}, we introduce the necessary background and notations. \S \ref{sec:model} introduces the model with the HIW prior. \S \ref{sec:theory} describes the main results on pairwise {\it posterior ratio consistency} and  consistent graph selection when the true graph is decomposable.  \S \ref{sec:theory2} states the main results on consistent graph selection 
 under model misspecification  and results on equivalence of minimal triangulations.
In each of Sections \ref{sec:theory} and \ref{sec:theory2}, the results are presented progressively as follows: First we provide a non-asymptotic sharp upper bound for pairwise Bayes factor. Next, we state the main theorem for posterior ratio consistency when $p$ diverges with $n$ with $p$ of the order $n^{\alpha}$ for $\alpha < 1/2$. Finally we state the main theorem on strong graph selection consistency which further requires $\alpha < 1/3$. Numerical experiments are presented in \S \ref{sec:sim} followed by a discussion in \S \ref{sec:disc}.

\section{Preliminaries}\label{sec:prelim}
In this section, we define a collection of notations required to describe the model and the prior. \S \ref{ssec:cor} introduces sample and population correlations and partial correlations, \S \ref{ssec:graphs} sets up the notations for undirected graphs and briefly introduces the definitions and properties  associated with decomposable graphs. \S \ref{ssec:matrix} contained matrix abbreviations and notations used throughout the paper. 

\subsection{Correlation and partial correlation}\label{ssec:cor}
Let $\bm{X}_p=(X_1,X_2,\ldots,X_p)^T$ denote a random vector which follows a $p$-dimensional Gaussian distribution and $\mathrm{x}^{(1)}, \mathrm{x}^{(2)}, \ldots, \mathrm{x}^{(n)}$ denote $n$ independent and identically distributed (i.i.d) samples  observations from $\bm{X}_p$.  Clearly, the $n \times p$ matrix formed by augmenting the $n$-dimensional column vectors $x_i$,  denoted $(\mathrm{x}_1, \mathrm{x}_2, \ldots,\mathrm{x}_p)$ is the same as  $( \mathrm{x}^{(1)}, \mathrm{x}^{(2)}, \ldots, \mathrm{x}^{(n)})^T$ and $\bar{\mathrm{x}}_i = n^{-1}\ind_n^T \mathrm{x}_i$, $i=1,2,\ldots,p$. Here  $\ind_n$ is an $n$-dimensional vector with all ones. Let $I_n$ denote an $n \times n$ identity matrix. % Anderson (1984) p60

\begin{definition} \label{defpcorr} {\normalfont{(Population correlation coefficient).}} % Anderson (1984) p21
	The population correlation coefficient between $X_i$ and $X_j$,  $1\leq i, j \leq p$, is defined as 
	\small
	\begin{equation*}
		\rho_{ij} = \frac{\sigma_{ij}}{\sqrt{\sigma_{ii}}\sqrt{\sigma_{jj}}},
	\end{equation*}
	\normalsize
	where \small$\sigma_{ii}=\bbE(X_i-\bbE X_i)^2$ \normalsize and \small$\sigma_{ij}=\bbE\{(X_i-\bbE X_i)(X_j-\bbE X_j)\}$\normalsize.
\end{definition}

\begin{definition} \label{defscorr} {\normalfont{(Sample/Pearson correlation coefficient).}} % Anderson (1984) p65 & Pearson (1896)
	The sample correlation coefficient between $X_i$ and $X_j$,  $1\leq i, j \leq p$, is defined as 
	\small
	\begin{equation*}
		\hat{\rho}_{ij} = \frac{\hat{\sigma}_{ij}}{\sqrt{\hat{\sigma}_{ii}}\sqrt{\hat{\sigma}_{jj}}},
	\end{equation*}
	\normalsize
	where \small$\hat{\sigma}_{ii}=(\mathrm{x}_i-\bar{\mathrm{x}}_i\ind_n)^T(\mathrm{x}_i-\bar{\mathrm{x}}_i\ind_n)/n$ \normalsize and \small$\hat{\sigma}_{ij}=(\mathrm{x}_i-\bar{\mathrm{x}}_i\ind_n)^T(\mathrm{x}_j-\bar{\mathrm{x}}_j\ind_n)/n$\normalsize.
\end{definition}

\begin{definition} \label{defppcorr} {\normalfont{(Population partial correlation coefficient).}} %Anderson (1984) p37
	Let $S = \{ i_1, i_2, \ldots, i_{\abs{S}} \}$, where $1\leq i_1, i_2, \ldots, i_{\abs{S}}\leq p$ and $\abs{S}$ is the cardinality of set $S$. Define $X_S = (\mathrm{X}_{i_1}, \mathrm{X}_{i_2}, \ldots, \mathrm{X}_{i_{\abs{S}}})^T$. The population partial correlation coefficient between $X_i$ and $X_j$, where $i,j\not\in S$ and $1\leq i,j \leq p$, holding $X_S$ fixed is defined as 
	\small
	\begin{equation*}
		\rho_{ij\mid S} = \frac{\sigma_{ij\mid S}}{\sqrt{\sigma_{ii\mid S}}\sqrt{\sigma_{jj\mid S}}},
	\end{equation*}
	\normalsize
	where \small$\sigma_{ii\mid S}=\sigma_{ii}-\sigma_{Si}^T\sigma^{-1}_{SS}\sigma_{Si}$\normalsize, \small$\sigma_{ij\mid S}=\sigma_{ij}-\sigma_{Si}^T\sigma^{-1}_{SS}\sigma_{Sj}$\normalsize. And \small$\sigma_{Si}=\bbE\{(X_S-\bbE X_S)(X_i-\bbE X_i)\}$\normalsize, \small$\sigma_{SS}=\bbE\{(X_S-\bbE X_S)^T(X_S-\bbE X_S)\}$\normalsize.
\end{definition}

\begin{definition} \label{defspcorr} {\normalfont{(Sample partial correlation coefficient).}} %Anderson (1984) p127
	Define $\mathrm{x}_S = (\mathrm{x}_{i_1}, \mathrm{x}_{i_2}, \ldots, \mathrm{x}_{i_{\abs{S}}})$. The sample partial correlation coefficient between $X_i$ and $X_j$, where $i,j\not\in S$ and $1\leq i,j \leq p$, holding $X_S$ fixed is defined as 
	\small
	\begin{equation*}
		\hat{\rho}_{ij\mid S} = \frac{\hat{\sigma}_{ij\mid S}}{\sqrt{\hat{\sigma}_{ii\mid S}}\sqrt{\hat{\sigma}_{jj\mid S}}},
	\end{equation*}
	\normalsize
	where \small$\hat{\sigma}_{ii\mid S}=\hat{\sigma}_{ii}-\hat{\sigma}_{Si}^T\hat{\sigma}_{SS}^{-1}\hat{\sigma}_{Si}$\normalsize, \small$\hat{\sigma}_{ij\mid S}=\hat{\sigma}_{ij}-\hat{\sigma}_{Si}^T\hat{\sigma}_{SS}^{-1}\hat{\sigma}_{Sj}$\normalsize. And \small$\hat{\sigma}_{Si}=(\mathrm{x}_S-\bar{\mathrm{x}}_S)^T(\mathrm{x}_i-\bar{\mathrm{x}}_i)/n$\normalsize, \small$\hat{\sigma}_{SS}=\big\{(\mathrm{x}_S-\bar{\mathrm{x}}_S)^T(\mathrm{x}_S-\bar{\mathrm{x}}_S)/n\big\}^{-1}$\normalsize, \small$\bar{\mathrm{x}}_S =(\bar{\mathrm{x}}_{i_1}\ind_n, \ldots, \bar{\mathrm{x}}_{i_{\abs{S}}}\ind_n)$.
\end{definition}
% Since the projection matrix $H_S$ does not change when you permutate the column of $\mathrm{x}_S$, the entry order in $S$ does not affects the value of sample partial correlation $\hat{rho}_{ij\mid S}$. 

\subsection{Undirected decomposable graphs} \label{ssec:graphs}
%Or Decomposable graphs and Gaussian graphical models
Denote an undirected graph by $G = (V, E)$ with a vertex set $V = \{1, 2, \ldots, p\}$ and an edge set $E=\{(r,s): e_{rs}=1, 1\leq r<s\leq p\}$ with $e_{rs} =1$ if the edge $(r, s)$ is present in $G$ and $0$ otherwise. 

For purpose of a self-contained exposition, we first review some basic terminologies of graph theory. A {\it path} of length $k$ in $G$ from vertex $u$ to $v$ is a sequence of $k-1$ distinct vertices of the form $u=v_0,v_1,\ldots,v_{k-1},v_k=v$ such that $(v_{i-1},v_i)\in E$ for all $i=1,2,\ldots,k$. The path is a {\it $k$-cycle} if the end points are the same, $u=v$. If there is a path from $u$ to $v$, then we say $u$ and $v$ are {\it connected}. A subset $S\subseteq V$ is said to be an {\it $uv$-separator} if all paths from $u$ to $v$ intersect $S$. The subset $S$ is said to {\it separate} $A$ from $B$ if it is an $uv$-separator for every $u\in A$, $v\in B$. A {\it chord} of a cycle is a pair of vertices that are not consecutive on the cycle, but are adjacent in $G$. A graph is {\it complete} if all vertices are joined by an edge. A {\it clique} is a complete subgraph that is maximal, maximally complete subgraph. See \cite{lauritzen1996graphical} for more graph related terminologies.

We shall focus on decomposable graphs in this paper. A graph is decomposable \cite{lauritzen1996graphical} if and only if its every cycle of length greater than or equal to four possesses a chord. A decomposable graph $G$ can be represented by a perfect ordering of its cliques and separators. Refer to \cite{lauritzen1996graphical} for formal definitions of a clique and a separator, and other equivalent representations. An ordering of cliques $C_i \in \mathcal{C}$ and separators $S_i \in \mathcal{S}$, where $\mathcal{C}=\{C_i\}_{i=1}^k$ and $\mathcal{S}=\{S_i\}_{i=2}^k$, $(C_1, S_2, C_2, S_3, \ldots C_k)$, is said to be perfect if for every $i = 2, 3,\ldots, k$ the running intersection property [\cite{lauritzen1996graphical}, page 15] is fulfilled, meaning that there exists a $j < i$ such that $S_i = C_i \cap H_{i-1}\subset C_j$ where $H_{i-1} = \cup_{j=1}^{i-1} C_j$.  A junction tree for the decomposable graph $G$ is a tree representation of the cliques. (For a non-decomposable graph, the junction tree consists of its prime components that are not necessarily cliques, i.e. complete). A tree with a set of vertices equal to the set of cliques of $G$ is said to be a junction tree if, for any two cliques $C_i$ and $C_j$ and any clique $C$ on the unique path between $C_i$ and $C_j$,  we have $C_i \cap C_j \subset C$. A set of vertices shared by two adjacent nodes of the junction tree is complete and defines the separator of the two subgraphs induced by the nodes. Denote by $\mathcal{D}_k$ the space of all decomposable graphs on $k$ notes.  Figures \ref{example1} and \ref{example2} briefly  illustrate a decomposable and a non-decomposable graph, both defined on $6$ nodes. 

\begin{figure}[h]
	\centering
	\begin{tikzpicture}[thick]
	\node[draw, circle] (n1) {$1$};
	\node[draw, circle, right=0.8 of n1] (n2) {$2$};
	\node[draw, circle, above right=0.6 of n2] (n3) {$3$};
	\node[draw, circle, below=1.1 of n3] (n4) {$4$};
	\node[draw, circle, right=1.1 of n4] (n5) {$5$};
	\node[draw, circle, above=1.1 of n5] (n6) {$6$};
	\edge[-] {n1} {n2};
	\edge[-] {n2} {n3};
	\edge[-] {n2} {n4};
	\edge[-] {n3} {n4};
	\edge[-] {n3} {n5};
	\edge[-] {n3} {n6};
	\edge[-] {n4} {n5};
	\edge[-] {n4} {n6};
	\edge[-] {n5} {n6};
	\node[left=0.15 of n1] {$G_6$};
	\node[draw, ellipse, right=0.8 of n6, label={above:$C_1$}] (C1) {$1$, $2$};
	\node[draw, regular polygon sides=4, right=0.4 of C1, label={above:$S_2$}] (S2) {$2$};
	\node[draw, ellipse, right=0.4 of S2, label={above:$C_2$}] (C2) {$2$, $3$, $4$};
	\node[draw, regular polygon sides=4, below=0.4 of C2, label={right:$S_3$}] (S3) {$3$, $4$};
	\node[draw, ellipse, below=0.4 of S3, label={right:$C_3$}] (C3) {$3$, $4$, $5$, $6$};
	\edge[-] {C1} {S2};
	\edge[-] {S2} {C2};
	\edge[-] {C2} {S3};
	\edge[-] {S3} {C3};
	\end{tikzpicture}
	\caption{$G_6$ is a $6$-node decomposable graph and its junction tree decomposition {\normalfont(}right{\normalfont)} has $3$ {\bf cliques} and $2$ separators, i.e. $C_1=\{1,2\}$, $S_2=\{2\}$, $C_2=\{2,3,4\}$, $S_3=\{3,4\}$, $C_3=\{3,4,5,6\}$.} \label{example1}
\end{figure}
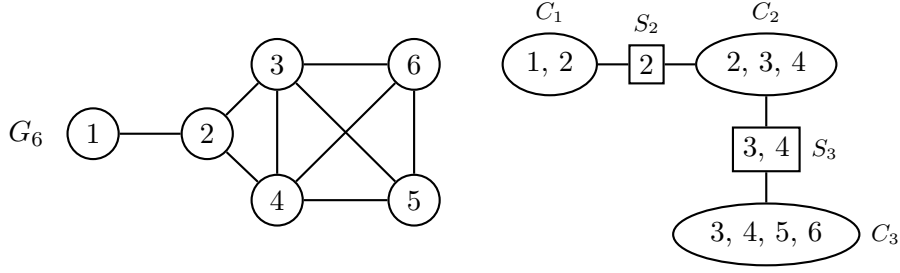

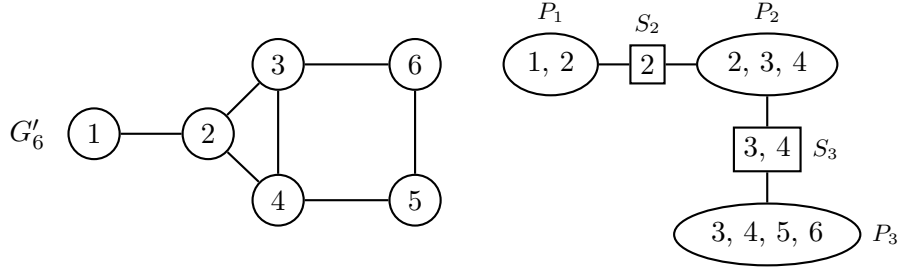
\begin{figure}[h]
	\centering
	\begin{tikzpicture}[thick]
	\node[draw, circle] (n1) {$1$};
	\node[draw, circle, right=0.8 of n1] (n2) {$2$};
	\node[draw, circle, above right=0.6 of n2] (n3) {$3$};
	\node[draw, circle, below=1.1 of n3] (n4) {$4$};
	\node[draw, circle, right=1.1 of n4] (n5) {$5$};
	\node[draw, circle, above=1.1 of n5] (n6) {$6$};
	\edge[-] {n1} {n2};
	\edge[-] {n2} {n3};
	\edge[-] {n2} {n4};
	\edge[-] {n3} {n4};
	\edge[-] {n3} {n6};
	\edge[-] {n4} {n5};
	\edge[-] {n5} {n6};
	\node[left=0.15 of n1] {$G'_6$};
	\node[draw, ellipse, right=0.8 of n6, label={above:$P_1$}] (C1) {$1$, $2$};
	\node[draw, regular polygon sides=4, right=0.4 of C1, label={above:$S_2$}] (S2) {$2$};
	\node[draw, ellipse, right=0.4 of S2, label={above:$P_2$}] (C2) {$2$, $3$, $4$};
	\node[draw, regular polygon sides=4, below=0.4 of C2, label={right:$S_3$}] (S3) {$3$, $4$};
	\node[draw, ellipse, below=0.4 of S3, label={right:$P_3$}] (C3) {$3$, $4$, $5$, $6$};
	\edge[-] {C1} {S2};
	\edge[-] {S2} {C2};
	\edge[-] {C2} {S3};
	\edge[-] {S3} {C3};
	\end{tikzpicture}
	\caption{$G'_6$ is a $6$-node non-decomposable graph because its cycle of four, $3-4-5-6$, does not have a cord. Its junction tree decomposition {\normalfont(}right{\normalfont)} has $3$ {\bf prime} components and $2$ separators, i.e. $P_1=\{1,2\}$, $S_2=\{2\}$, $P_2=\{2,3,4\}$, $S_3=\{3,4\}$, $P_3=\{3,4,5,6\}$. Out of all prime components only $P_1$ and $P_2$ are cliques.} \label{example2}
\end{figure}

\subsection{Matrix notations}\label{ssec:matrix}
For an $n \times p$ matrix $Y$, $\mathrm{Y}_C$ is defined as the submatrix of $\mathrm{Y}$ consisting of columns with indices in the clique $C$. Let $(\mathrm{y}_1, \mathrm{y}_2, \ldots, \mathrm{y}_p)=(Y_1,Y_2,\ldots,Y_n)^T$, where $\mathrm{y}_i$ is the $i$th column of $\mathrm{Y}_{n\times p}$. If $C = \{ i_1, i_2, \ldots, i_{\abs{C}} \}$, where $1\leq i_1 < i_2 < \ldots < i_{\abs{C}} \leq p$, then $\mathrm{Y}_C = (\mathrm{y}_{i_1}, \mathrm{y}_{i_2}, \ldots, \mathrm{y}_{i_{\abs{C}}})$. For any square matrix $A = {(a_{ij})}_{p \times p}$, define $A_C = {(a_{ij})}_{\abs{C}\times\abs{C}}$ where $i, j \in C$, and the order of entries carries into the new submatrix $A_C$. Therefore, $\mathrm{Y}_C^T\mathrm{Y}_C = (\mathrm{Y}^T\mathrm{Y})_C$. 

$\mbox{MN}_{m \times n}(M,\Sigma_r,\Sigma_c)$ is an $m\times n$ matrix normal distribution with mean matrix $M$, $\Sigma_r$ and $\Sigma_c$ as covariance matrices between rows and columns, respectively. 

\subsection{Miscellaneous}\label{ssec:misc}
Let $\bbP$ be the probability corresponding to the true data generating distribution. Denote $\mathcal{G}_k$ and $\mathcal{D}_k$ as the $k$-dimensional graph space and $k$-dimensional decomposable graph space. Let $\mathcal{M}_t$ be the minimal triangulation space of $G_t$ when $G_t$ is non-decomposable. $a \asymp b$ denotes $C_1a \leq b \leq C_2 a$ for constants $C_1, C_2$. $a \precsim b$ denotes $ a \leq C_3 b$ for a constant $C_3$. For set relations, $A\subset B$ means $A$ is a subset of $B$; $A\subsetneq B$ means $A\subset B$ and $A\neq B$; $A\not\subset B$ means $A$ is not a subset of $B$. $\abs{\,\cdot\,}$ determined by context can be absolute value, cardinality of sets or determinant of matrices. $\pi(\cdot)$ and $\pi(\cdot\mid\mathrm{Y})$ are the prior distribution and posterior distribution of graphs, respectively.  Refer also to Table \ref{Table:default} for a detailed list of notations used in the theorem statements and the proofs.

\section{Bayesian hierarchical model for graph selection}\label{sec:model}
Suppose we observe independent and identically distributed $p$-dimensional Gaussian random variables $Y_i, i=1, \ldots, n$.  To describe the common distribution of $Y_i$, define a  $p \times p$ covariance matrix $\Sigma_G$ that depends on an undirected decomposable graph as defined in \S \ref{ssec:graphs}. Assume  $Y_i \mid \Sigma_G, G \sim \mbox{N}_p(0, \Sigma_G)$. In matrix notations, 
\begin{equation}\label{eq:dgm}
	\mathrm{Y}_{n\times p} \mid  \Sigma_G, G \sim \mbox{MN}_{n \times p}(\bm{0}_{n\times p},  I_n, \Sigma_G),
\end{equation}
where $\mathrm{Y}_{n\times p} = (Y_1,Y_2,\ldots, Y_n)^T$ and $\bm{0}_{n\times p}$ is an $n\times p$ matrix with all zeros. The prior used here for covariance matrix $\Sigma_G$ given a decomposable graph $G$ is the hyper-inverse Wishart prior, described below.

\subsection{The Hyper-inverse Wishart distribution}
Denote by $\mbox{HIW}_G(b, D)$ \citep{dawid1993hyper, carvalho2009objective} a distribution on the cone of $p \times p$ positive definite matrices with degrees of freedom $b>2$ \citep{jones2005experiments} and a fixed $p \times p$ positive definite matrix $D$ such that the joint density factorizes on the junction tree of the given decomposable graph $G$ as 
\begin{equation}\label{eq:HIW}
	p(\Sigma_G \mid b, D) = \frac{\prod_{C \in \mathcal{C}} p(\Sigma_C \mid b, D_C)}
	{\prod_{S \in \mathcal{S}} p(\Sigma_S \mid b, D_S)},
\end{equation}
where for each $C \in \mathcal{C}$, $\Sigma_C \sim \mbox{IW}_{\abs{C}}(b, D_C)$ with density 
\begin{equation*}
	p(\Sigma_C \mid b, D_C) \propto \abs{\Sigma_C}^{-(b +2 \abs{C})/2} \mbox{etr}\Big\{-\frac{1}{2}\Sigma_C^{-1} D_C\Big\},
\end{equation*}
where $\abs{C}$ is the cardinality of the clique $C$ and $\mbox{etr}(\cdot)=\exp\big\{\mbox{tr}(\cdot)\big\}$. $\mbox{IW}_p(b,D)$ is the inverse Wishart distribution with degrees of freedom $b$ and a fixed $p\times p$ positive definite matrix $D$ with normalizing constant 
\small
\begin{equation*}
	\abs{\frac{1}{2}D}^{(b+p-1)/{2}}\Gamma^{-1}_p\Big(\frac{b+p-1}{2}\Big),
\end{equation*}
\normalsize
where $\Gamma_p(\cdot)$ is a multivariate gamma function. Refer to  \cite{carvalho2009objective} for more details about this parametrization of the inverse Wishart distribution.

\subsection{Bayesian inference on graphs}
Since the joint density factorizes over cliques and separators,
\small
\begin{equation}
	f(\mathrm{Y}\mid\Sigma_G) = {(2\pi)}^{-\frac{np}{2}} \frac
	{\prod_{C\in\mathcal{C}} {\abs{\Sigma_C}}^{-\frac{n}{2}} \mbox{etr}\Big( -\frac{1}{2}\Sigma_C^{-1}\mathrm{Y}_C^T\mathrm{Y}_C \Big)}
	{\prod_{S\in\mathcal{S}} {\abs{\Sigma_S}}^{-\frac{n}{2}} \mbox{etr}\Big( -\frac{1}{2}\Sigma_S^{-1}\mathrm{Y}_S^T\mathrm{Y}_S \Big)} \label{eq:likelihood}
\end{equation}
\normalsize
in the same way as in \eqref{eq:HIW}, and
\small
\begin{align*}
	f(\Sigma_G\mid G) &= \frac{\prod_{C\in\mathcal{C}} p(\Sigma_C\mid b, D_C)}{\prod_{S\in\mathcal{S}} p(\Sigma_S\mid b, D_S)} \\
	&= \frac{\prod_{C\in\mathcal{C}} \abs{\frac{1}{2}D_C}^{\frac{b+\abs{C}-1}{2}}\Gamma^{-1}_{\abs{C}}\big(\frac{b+\abs{C}-1}{2}\big)\abs{\Sigma_C}^{-\frac{b+2\abs{C}}{2}} \mbox{etr}\big( -\frac{1}{2} \Sigma_C^{-1}D_C \big)}
	{\prod_{S\in\mathcal{S}} \abs{\frac{1}{2}D_S}^{\frac{b+\abs{S}-1}{2}}\Gamma^{-1}_{\abs{S}}\big(\frac{b+\abs{S}-1}{2}\big)\abs{\Sigma_S}^{-\frac{b+2\abs{S}}{2}} \mbox{etr}\big( -\frac{1}{2} \Sigma_S^{-1}D_S \big )},
\end{align*}
\normalsize
it is straightforward to obtain the marginal likelihood of the decomposable graph $G$,
\small 
\begin{equation*}
	f(\mathrm{Y}\mid G) = 
	{(2\pi)}^{-\frac{np}{2}} \frac{h(G,b,D)}{h(G,b+n,D+\mathrm{Y}^T\mathrm{Y})} 
	= {(2\pi)}^{-\frac{np}{2}} \frac{\prod_{C\in\mathcal{C}} w(C) }{\prod_{S\in\mathcal{S}} w(S)},
\end{equation*}
\normalsize
where 
\small
\begin{equation*}
	h(G,b,D) = \frac{\prod_{C\in\mathcal{C}}\abs{\frac{1}{2}D_C}^{\frac{b+\abs{C}-1}{2}}\Gamma^{-1}_{\abs{C}}\big(\frac{b+\abs{C}-1}{2}\big)} {\prod_{S\in\mathcal{S}}\abs{\frac{1}{2}D_S}^{\frac{b+\abs{S}-1}{2}}\Gamma^{-1}_{\abs{S}}\big(\frac{b+\abs{S}-1}{2}\big)}, \,
	w(C) = \frac{\abs{D_C}^{\frac{b+\abs{C}-1}{2}} {\abs{D_C+\mathrm{Y}_C^T\mathrm{Y}_C}}^{-\frac{b+n+\abs{C}-1}{2}}} {2^{-\frac{n\abs{C}}{2}}\Gamma_{\abs{C}}\big(\frac{b+\abs{C}-1}{2}\big)\Gamma_{\abs{C}}^{-1}\big(\frac{b+n+\abs{C}-1}{2}\big)}.
\end{equation*}
\normalsize
Throughout the remainder of the paper, we shall be working with the hyper-inverse Wishart $g$-prior \cite{carvalho2009objective}, denoted as
\begin{equation}\label{hiwg}
	\Sigma_G \mid  G \sim \mathrm{HIW}_G(b, g\mathrm{Y}^T\mathrm{Y}),
\end{equation}
where $g$ is some suitably small fraction in $(0, 1)$ and $b>0$ is a fixed constant. Following the recommendation in \cite{carvalho2009objective}, we choose $g = 1/n$ through the remainder of the paper. Intuitively,  this choice of $g$ avoids  overwhelming the likelihood asymptotically as well as arbitrarily diffusing the prior.
%{\color{blue}\bf (need to give a reason, what if using $g=O(n^{-\phi})$, where $\phi>0$. This can be another assumption.)} 
In that case,
\small
\begin{equation*}
	w(C) = \frac{(n+1)^{-\frac{\abs{C}(b+n+\abs{C}-1)}{2}}\abs{\mathrm{Y}^T_C\mathrm{Y}_C}^{-\frac{n}{2}}}
	{(2n)^{-\frac{n\abs{C}}{2}}\Gamma_{\abs{C}}\big(\frac{b+\abs{C}-1}{2}\big)\Gamma_{\abs{C}}^{-1}\big(\frac{b+n+\abs{C}-1}{2}\big)}.
\end{equation*}
\normalsize
 The choice of focusing on the hyper-inverse Wishart $g$-prior in this paper is driven by the following two reasons. First, we can simplify the edge/signal strength assumption in terms of the smallest nonzero entries  in the partial correlation matrix, which serves as a natural interpretation of the edge strength compared to assumptions on the eigenvalues of the correlation matrix. Second, we conjecture that the results stated in \S \ref{sec:theory} and \ref{sec:theory2} continue to hold for any choice of HIW prior. The proof techniques under HIW g-prior serve as representations to the principle ideas in the article and can be easily adapted to other variations of HIW prior.

To complete a fully Bayesian specification, we place a prior distribution $\pi(\cdot)$ on the decomposable graph $G$.  Our theoretical results in \S \ref{sec:theory} and \ref{sec:theory2} are independent of the prior choice on $G$ if we consider a fixed $p$ asymptotics. However, for $p$ increasing with $n$ we need a suitable penalty on the number of edges of the random graph to penalize the false positives. Here is a popular example \citep{jones2005experiments,dobra2004sparse,carvalho2009objective,scott2008feature,cao2016posterior} we consider in the paper. Considering an undirected decomposable graph $G$, we assume the edges are independently drawn from a Bernoulli distribution with a common probability $q$:    
\begin{equation}\label{eq:priorG}
	\pi(G \mid q) \propto \bigg[\prod_{r<s} q^{e_{rs}} (1-q)^{1-e_{rs}}\bigg]\cdot\ind_{\mathcal{D}}(G),
\end{equation}
where $\mathcal{D}$ is the set of all decomposable graphs with $\abs{V} = p$ vertices and $q$ is the prior edge inclusion probability. We control the parameter $q$ to induce sparsity on the number of edges.  \cite{jones2005experiments} recommends using $2/(\abs{V}-1)$ as the hyper-parameter for the Bernoulli distribution. For an undirected graph, it has peak around $\abs{V}$ edges and the mode is smaller for decomposable graphs. We outline specific choices in  \S \ref{sec:theory} and \ref{sec:theory2} below.

\section{Theoretical results in the well-specified case}\label{sec:theory}
In this section, we present our main consistency results. The proofs of the results are deferred to the Appendix.  
Before introducing the assumptions, we need to adapt previous notations to the high-dimensional graph selection problem. Let $\mathrm{Y}=(Y_1, Y_2,\ldots,Y_n)^T$
%Let $p_n$ be the dimension of data which is a function of sample size $n$. Let $\mathrm{Y}^n=(Y^n_1, Y^n_2,\ldots,Y^n_n)^T$ and $Y^n_1, Y^n_2,\ldots,Y^n_n$ are $n$ independent and identically distributed samples from a $p_n$-dimensional Gaussian distribution $\mbox{N}_{p_n}(0,\Sigma_0^n)$, where $\Sigma_0^n$ is the true covariance matrix and denote 
and $\Omega_0=\Sigma_0^{-1}$  the corresponding precision matrix. Without loss of generality, we assume all column means of $\mathrm{Y}$ are zero. Let $G_t=(V, E_t)$ denote the true decomposable graph induced by $\Omega_0$, $\rho_{ij\mid V\backslash\{i,j\}}$ denote the true partial correlation between node $i$ and $j$ given the rest of the nodes $V\backslash\{i,j\}$. Assume $\rho_L$ and $\rho_U$ are the smallest and largest in absolute value of the {\it non-zero} population partial correlations, i.e.
\small
\begin{equation*}
	\rho_L=\min_{\substack{1\leq i<j\leq p\\(i,j)\in E_t}} \abs{\rho_{ij\mid V\backslash\{i,j\}}}, \quad \rho_U=\max_{\substack{1\leq i<j\leq p\\(i,j)\in E_t}} \abs{\rho_{ij\mid V\backslash\{i,j\}}},
\end{equation*}
\normalsize

Let $G_a=(V, E_a)$ be any alternative decomposable graph other than the true graph $G_t$. Denote by $E^{1}_a=E_t\cap E^n_a$ the set of true edges in $G_a$. Notice, when $E_t\subsetneq E_a$, we have $E^{1}_a=E_t$. Denoting by $|\cdot|$ the cardinality of a set,  $\abs{E_t}$ is the number of edges in $G_t$, $\abs{E^{1}_a}$ is the number of true edges in $G_a$. Define  $G_c=(V,E_c)$, where $E_c=\{(i,j):e_{ij}=1, 1\leq i<j\leq p\}$, to be the complete graph such that $\abs{E_c}=p(p-1)/2$. By definition, $G_c$ is a decomposable graph. We use $G_a\neq G_t$ to denote $E_a\neq E_t$; $G^n_a\not\subset G_t$ to denote $E_a\not\subset E_t$; $G_a\subsetneq G_t$ to denote $E_a\subsetneq E_t$. In the following, we state the main assumptions for graph selection consistency.

\begin{assumption}\label{assump-dim}{\normalfont{(Graph size)}}
	\begin{equation*}
		p  \precsim n^\alpha, \text{ where } 0<\alpha<1.
	\end{equation*}
\end{assumption}

\begin{assumption}\label{assump-lower}{\normalfont{(Edge sensitivity and identifiability)}}
	\begin{equation*}
		 \rho_L\asymp n^{-\lambda}, \text{ where } 0\leq\lambda<\frac{1}{2}.
	\end{equation*}
\end{assumption}

\begin{assumption}\label{assump-edge}{\normalfont{(Number of maximum edges in $G^n_t$)}}
	\begin{equation*}
		|E_t| \precsim n^\sigma, \text{ where } 0\leq\sigma\leq 2\alpha.
	\end{equation*}
\end{assumption}

\begin{assumption}\label{assump-prior}{\normalfont{(Prior edge inclusion probability)}}
	\begin{equation*}
		q \asymp e^{-C_qn^\gamma}, \text{ where } 0<\gamma<1,\,0<C_q<\infty.
	\end{equation*}
\end{assumption}

\begin{assumption}\label{assump-upper}{\normalfont{(Imperfect linear relationship)}}
	\begin{equation*}
		1-\rho_U \asymp n^{-k}, \text{ where } k\geq 0 \text{ and } \rho_U \neq 1.
	\end{equation*}
\end{assumption}
The main results will have additional restrictions on the parameters $(\alpha, \lambda)$,  but it is important to note that we require $\rho_L$ to not decrease to $0$ too quickly in order to ensure that the graph is identifiable. On the other hand, $\rho_U$ can be allowed to be sufficiently close to $1$. 
%
%We would like to point out that the assumptions stated in th
%\cite{cao2016posterior}. In their paper, eigenvalue conditions are assumed instead of partial correlations. Due to the fundamental differences in proving techniques, it is hard to draw precise comparisons between the assumptions of these two papers. But in some cases, we found the partial correlation conditions are somehow weaker than the eigenvalue conditions.

%A detailed discussion on the assumptions is deferred to the end of \S \ref{ssec:sgsc} after the main results. 
%For the remainder of this paper, we drop the $n$ from $p_n$, $\mathrm{Y}^n$, $\Sigma^n_0$, $\Omega^n_0$, $V^n$, $G^n_t$, $E^n_t$, $G^n_a$, $E^n_a$, $E^{n,1}_a$, $G^n_c$, $E^n_c$, $\rho^n_L$, $\rho^n_U$, $q_n$ and use $p$, $\mathrm{Y}$, $\Sigma_0$, $\Omega_0$, $V$, $G_t$, $E_t$, $G_a$, $E_a$, $E^1_a$, $G_c$, $E_c$, $\rho_L$, $\rho_U$, $q$ instead.

\subsection{Pairwise Bayes factor consistency for fixed $p$}\label{ssec:pwfp}
In this section, we assume $p, \rho_U$ and $\rho_L$ are all fixed constants. 
As a first step towards model selection, we investigate the behavior of the pairwise Bayes factor 
\begin{eqnarray}\label{eq:BF}
	\mbox{BF}(G_a; G_t) = \frac{f(\mathrm{Y} \mid G_a)}{f(\mathrm{Y}\mid G_t)},
\end{eqnarray}
where $G_t$ is the decomposable true graph and $G_a$ is any other decomposable graph.  
In this section, we shall investigate sufficient conditions on the likelihood \eqref{eq:likelihood} and the prior on $(\Sigma_G, G)$ given by \eqref{hiwg} and \eqref{eq:priorG}  such that the Bayes factor \eqref{eq:BF} converges to $0$ as $n \to \infty$ for any graph 
$G_a \neq G_t$. 
\begin{theorem} \label{bfupper0} {\normalfont (Upper bound for pairwise Bayes factor).}
	Assume the graph dimension $p$ is a fixed constant and $\rho_U\neq 1$. Given any decomposable graph $G_a\neq G_t$, there exists a set $\Delta_a$, such that on the set $\Delta_a$, if $n>\max\{p+b,4p\}$, we have
	\begin{enumerate}
	\setlength\itemsep{0em}
		\item when $G_t\not\subset G_a$,
			\small
			\begin{equation}
				{\normalfont\mbox{BF}}(G_a;G_t) < \exp \Big\{-\frac{n\rho^2_L}{2} + \delta(n) \Big\}, \label{eq:true_edge}
			\end{equation}
			\normalsize
		\item when $G_t\subsetneq G_a$,
			\small
			\begin{equation}
				{\normalfont\mbox{BF}}(G_a;G_t) < \big(e^{p^2}\big)\cdot n^{-\frac{1}{2}(|E_a|-|E_t|)(1- 2/\tau^*)}, \label{eq:false_edge}
			\end{equation}
			\normalsize
	\end{enumerate}
	and
	\small
	\begin{equation*}
		\bbP(\Delta_a) \geq 1 - \frac{42p^2}{(1-\rho_U)^2} (n-p)^{-\frac{1}{4\tau^*} }\Big\{\frac{1}{\tau^*}\log(n-p)\Big\}^{-\frac{1}{2}}, 
	\end{equation*}
	\normalsize
	where $\tau^*>2$ and $\delta(n) = p^2\log n + \sqrt{n\log n} + 3p^2\log p$ satisfying $\delta(n)/n\rightarrow 0$, as $n\rightarrow\infty$.
\end{theorem}
\noindent The next corollary is the direct result from Theorem \ref{bfupper0}.
\begin{corollary} {\normalfont(Finite graph pairwise Bayes factor consistency).}
	Let $G_a$ be any decomposable graph and $G_a\neq G_t$. The graph dimension $p$ is a fixed constant. If $\rho_U\neq 1$, then ${\normalfont\mbox{BF}}(G_a;G_t)\stackrel{\bbP}{\rightarrow} 0$, as $n\rightarrow \infty$.
\end{corollary}
When $p$ is fixed, the likelihood  is strong enough to consistently recover the graph. One key aspect of the proof is that Bayes factor in favor of adding a true edge versus the lack of it is exponentially small, while the Bayes factor in favor for adding a false edge decreases to zero only at a polynomial rate.   

We emphasize here that exponential rate for deletion (of true edges) is only true when the corresponding population partial correlation or correlation is non-zero. From the {\it global Markov property}, we know if two nodes are adjacent then any partial correlation between them is non-zero but their correlation can be zero. The polynomial rate for addition (of false edges) is only true when the corresponding population partial correlation or correlation is zero. When two nodes are not adjacent, then only the set that separates them will results in a zero partial correlation. We choose the path of $G_t\rightarrow G_c\rightarrow G_a$ which ensures us the exponential decay when missing true edges and polynomial decay when adding false edges. 
%These results have been formed into two lemmas before the main theorems. They are the central part of the consistency proofs. Any other path will not likely result in the same behavior. By that we mean, Bayes factors of deletion cases may go to zero in a polynomial rate and same goes for addition cases, it can be an exponential rate. 
%Notice in both deletion and addition, after selecting the edge to be deleted or added, i.e. after selecting the two nodes $x$ and $y$ related to it, the separator $S$ in the partial correlation $x$ and $y$ conditioned on is fixed. So in the expression below we do not write $S$ explicitly. Let $q_s^i = |S_i^{\{x_i,y_i\}}|$.

\subsection{Posterior ratio consistency for growing $p$}\label{ssec:prc}
Next we examine the convergence of posterior ratio,
\begin{eqnarray}\label{eq:PR}
	\mbox{PR}(G_a; G_t) = \frac{f(\mathrm{Y} \mid G_a)\pi(G_a)}{f(\mathrm{Y}\mid G_t)\pi(G_t)},
\end{eqnarray}
when the dimension of graphs grows with sample size.

\begin{theorem} \label{thpr} {\normalfont (High-dimensional graph posterior ratio consistency).}
	Let $G_a$ be any decomposable graph and $G_a\neq G_t$ and Assumptions \ref{assump-dim}-\ref{assump-upper} are satisfied with 
\begin{align*}
0 < \alpha < \frac{1}{2}, \quad 0\leq\lambda < \min\Big\{\alpha, \frac{1}{2}-\alpha \Big\}.
\end{align*}
By choosing $\gamma$ in the interval $(\max\{0,1-4\alpha\},  1-\sigma-2\lambda)$ we have 
	 ${\normalfont\mbox{PR}}(G_a;G_t)\stackrel{\bbP}{\rightarrow} 0$, as $n\rightarrow\infty$.
\end{theorem}
When the graph size grows with $n$, the partial correlation is no longer a constant. The HIW prior does not naturally favor parsimonious graphs, so a penalty on the number of edges in the graph in needed by restricting $\gamma$ in the above interval.  Note also that we do not need any further restriction on $\sigma$ in Assumption \ref{assump-edge}  meaning that the true graph is allowed to be the complete graph for the posterior ratio consistency to hold. 

\subsection{Strong graph selection consistency}\label{ssec:sgsc}
In this section, we examine the behavior of 
\begin{eqnarray*}
\pi(G\mid \mathrm{Y}) = \frac{f(\mathrm{Y}\mid G)\pi(G)}{\sum_{G' \in \mathcal{D}}f(\mathrm{Y}\mid G')\pi(G')}
\end{eqnarray*}
as $n, p \to \infty$. 
\begin{theorem} \label{thstrong} {\normalfont (Strong graph selection consistency).}
	Let $G_a$ be any decomposable graph and $G_a\neq G_t$ and Assumptions \ref{assump-dim}-\ref{assump-upper} are satisfied with 
\begin{align*}
& 0 < \alpha < \frac{1}{3}, \quad 0\leq\lambda < \min\Big\{\alpha, \frac{1-3\alpha}{2} \Big\}. 
\end{align*}
By choosing $\gamma$ in the interval $(\max\big\{\alpha,1-4\alpha\big\},  1-\sigma-2\lambda)$,  we have
	\begin{equation*}
		\pi(G_t\mid \mathrm{Y})\stackrel{\bbP}{\rightarrow} 1, \text{ as } n\rightarrow \infty.
	\end{equation*}
\end{theorem}
Strong selection consistency demands all posterior ratio to be converging simultaneously at a sufficiently fast rate so that the sum is convergent. Since the number of alternative graphs is of the order $2^{p^2}$, to make the sum convergent, we require further assumptions on the model complexity and an accompanying stronger penalty $\pi$.  We achieve this by shrinking the dimension of graph space $(\alpha < 1/3)$ and inducing a slightly stronger sparsity (by selecting larger $\gamma$) on the prior over the graph space.   

In the proofs of Theorem \ref{bfupper0}-\ref{thstrong}, by using the decomposable graph chain rule, we traverse to any decomposable graph from the true graph and thus break down the Bayes factor into local moves, i.e. addition and deletion of a single edge. The local moves then can be associated with sample partial correlations and sample correlations, which are the natural criterion of edge selection by definition. This enables us to transform the problem into a more understandable manner.

In practice, one might be interested in a consistent point estimate rather than the entire posterior distribution. In Bayesian inference for discrete configurations, a posterior mode  provides a natural surrogate for the MLE. In the following, we investigate the consistency of the posterior mode obtained from our hierarchical Bayesian  model  as a simple bi-product of Theorems  \ref{thpr} and  \ref{thstrong}. Define $\hat{G}$ to be the posterior mode in the decomposable graph space, i.e.
\begin{equation*}
	\hat{G} = \mbox{argmax}_{G\in \mathcal{D}} \pi(G\mid \mathrm{Y}).
\end{equation*}
Then the following in true. 
%The posterior mode should provide a good point estimate of the underlying true graph. Thus we have the following corollary. It is a direct result from Theorem \ref{thpr} and \ref{thstrong}. 
\begin{corollary} \label{decomp-mode} {\normalfont(Consistency of posterior mode when $G_t$ is decomposable).}
	Under the assumptions of Theorem \ref{thstrong}, the probability which the posterior mode $\hat{G}$ is equal to the true graph $G_t$ goes to one, i.e.
	\begin{equation*}
		\bbP\big(\hat{G}=G_t)\rightarrow 1, \quad \text{ as } n\rightarrow \infty.
	\end{equation*}

\end{corollary}

\section{Theoretical results under model misspecification}\label{sec:theory2}
In this section, we investigate the effect of model misspecification when the underlying true graph $G_t$ is non-decomposable. We begin with some definitions on  triangulation and minimal triangulations of a graph. A triangulation of graph $G=(V,E)$ is a decomposable graph $G^\Delta=(V,E\cup F)$. The edges in $F$ are called {\it fill-in} edges. A triangulation $G^{\Delta}=(V,E\cup F)$ of $G=(V,E)$ is minimal if $(V,E\cup F')$ is non-decomposable for every $F'\subsetneq F$ \cite{heggernes2006minimal}. A triangulation is minimal if and only if the removal of any single fill-in edge from it results in a non-decomposable graph \cite{rose1976algorithmic,heggernes2006minimal}. This property captures the important aspect of minimal triangulations. For a summary of minimal triangulations of graphs, see \cite{heggernes2006minimal} for more details. Next, we state two theorems graph selection consistency under a true non-decomposable graph. 
\begin{theorem} \label{finitetri} {\normalfont(Convergence and equivalence of minimal triangulations for finite graphs).}
	Assume the true graph $G_t$ is non-decomposable. When the graph dimension $p$ is a fixed constant ($\rho_U, \rho_U$ are fixed constants), we have the following: 
	\begin{enumerate}
		\setlength\itemsep{0.6em}
		\item Let $G_m$ be any minimal triangulation of $G_t$ and $G_a$ be any decomposable graph that is not a minimal triangulation of $G_t$. If $\rho_U\neq 1$, then ${\normalfont\mbox{BF}}(G_a;G_m)\stackrel{\bbP}{\rightarrow} 0$, as $n\rightarrow\infty$.
		\item Let $G_{m_1}$ and $G_{m_2}$ be any two different minimal triangulations of $G_t$ {\normalfont(}with the same number of fill-in edges{\normalfont)}. Then the Bayes factor between them are stochastically bounded, i.e. for any $0<\epsilon<1$, there exist two positive finite constants $A_1(\epsilon)<1$ and $A_2(\epsilon)>1$, such that
		\small
		\begin{equation*}
			\bbP\big\{A_1<{\normalfont\mbox{BF}}(G_{m_1};G_{m_2})<A_2\big\}>1-\epsilon, \quad\text{ for } n>p+\max\Big\{3,b,6\log\big(10p^2/\epsilon\big)\Big\}.
		\end{equation*}
		\normalsize
		\item If $\rho_U\neq 1$, we have $\sum_{G_m\in\mathcal{M}_t} \pi(G_m\mid \mathrm{Y})\stackrel{\bbP}{\rightarrow} 1$, as $n\rightarrow\infty$, where $\mathcal{M}_t$ is the minimal triangulation space of $G_t$.
	\end{enumerate}
\end{theorem}

\begin{theorem} \label{hightri} {\normalfont(Convergence and equivalence of minimal triangulations for high-dimensional graphs).}
	Assume the true graph $G_t$ is not decomposable. When the graph dimension $p$ grows with $n$, we have the following results.
	\begin{enumerate}
		\setlength\itemsep{0.6em}
		\item Let $G_m$ be any minimal triangulation of $G_t$ and $G_a$ be any decomposable graph that is not a minimal triangulation of $G_t$. Assume
		\small
		\begin{align*}
			 0 < \alpha < \frac{1}{2}, \quad 0\leq\lambda < \min\Big\{\alpha, \frac{1}{2}-\alpha \Big\}, \quad 0 < \sigma <  \min\Big\{2(\alpha-\lambda), 2(\frac{1}{2}-\alpha-\lambda)\Big\}. 
		\end{align*}
		\normalsize
Choose $\gamma$ in the interval $(\max\big\{2\alpha,1-2\alpha\big\} , 1-\sigma-2\lambda)$. Then under Assumptions \ref{assump-dim}-\ref{assump-upper}, we have ${\normalfont\mbox{PR}}(G_a;G_m)\stackrel{\bbP}{\rightarrow} 0$, as $n\rightarrow\infty$.
		\item Let $G_{m_1}$ and $G_{m_2}$ be any two different minimal triangulations of $G_t$. If the number of fill-in edges is finite, then the Bayes factor between them are stochastically bounded.
		\item If \small
		\begin{align*}
			 0 < \alpha < \frac{1}{3}, \quad 0\leq\lambda < \min\Big\{\alpha, \frac{1-3\alpha}{2} \Big\}, 0\leq \sigma < \min\Big\{2(\alpha-\lambda), 2\Big(\frac{1-3\alpha}{2}-\lambda\Big)\Big\}. 
		\end{align*}
		\normalsize
And we choose $\gamma$ in the interval $(\max\big\{3\alpha,1-2\alpha\big\}, 1-\sigma-2\lambda)$, then under Assumptions \ref{assump-dim}-\ref{assump-upper}, we have $\sum_{G_m\in\mathcal{M}_t} \pi(G_m\mid \mathrm{Y})\stackrel{\bbP}{\rightarrow} 1$, as $n\rightarrow\infty$, where $\mathcal{M}_t$ is the minimal triangulation space of $G_t$.
	\end{enumerate}
\end{theorem}
Based on the theorems presented above, the equivalence among minimal triangulations is true when the number of fill-in edges is finite.  Adding infinitely many fill-in edges prompts the minimal triangulations to drift further away from the true graph. In that case, there are too many possibilities among the minimal triangulations such that they can be vastly different for each other.  %The final minimal triangulation selected depends on the data and the true underlying model. I
It is worth mentioning that any decomposable subgraph of the true graph is not a good posterior estimate of the true graph. This is simply due to the fact that such a graph is associated with at least one edge deletion step following by reciprocal of addition steps from a minimal triangulation. Since deletion of any true edge results in an exponential decay of the Bayes factor in favor of the deletion and the reciprocal of additions will be in favor of additions (the minimal triangulations) or neutral depending on whether the corresponding population partial correlation is zero. Thus, pairwise speaking, the posterior mode is among minimal triangulation class. 
%In the next section, this is illustrated by a small scale simulation study.

Analogous to Corollary \ref{decomp-mode}, when the true graph $G_t$ is not decomposable, we state the behavior of posterior mode in the following corollary under model misspecification.

\begin{corollary} \label{nondecomp-mode} {\normalfont(Consistency of posterior mode when $G_t$ is non-decomposable).}
	Under the assumptions of Theorem \ref{hightri}, the posterior mode $\hat{G}$ is in the minimal triangulation space $\mathcal{M}_t$ of the true graph $G_t$ with probability converging to one, i.e.
	\small
	\begin{equation*}
		\bbP\big(\hat{G}\in\mathcal{M}_t)\rightarrow 1, \quad \text{ as } n\rightarrow \infty.
	\end{equation*}
	\normalsize
\end{corollary}

\section{Simulations}\label{sec:sim}
We conduct two sets of simulations for the demonstrate the convergence of Bayes factors in the well-specified case (Theorem \ref{bfupper0}) and  in the misspecified case (Theorem \ref{finitetri})  for fixed $p.$  
\subsection{Simulation 1: Demonstration of pairwise Bayes factor convergence rate}\label{ssec:well}
In this section, we conduct a simulation study in $\mathcal{D}_3$ to demonstrate the convergence rate of pairwise Bayes factors. Let $\mathcal{G}_k$ be the $k$-dimensional graph space. Since there is no non-decomposable graph with 3 nodes, $\mathcal{D}_3$ is the same as $\mathcal{G}_3$. All 8 graphs in $\mathcal{D}_3$ are enumerated in Figure \ref{3nodegraph}.

\begin{figure}[H]
	\centering
	\begin{tikzpicture}[thick]
	\node[draw, shape=circle, fill=black, inner sep=2pt, label={left: $1$}] (n11) {};
	\node[draw, shape=circle, fill=black, inner sep=2pt, label={left: $2$}, below left=0.8 of n11] (n12) {};
	\node[draw, shape=circle, fill=black, inner sep=2pt, label={right: $3$}, below right=0.8 of n11] (n13) {};
	\node[below=0.9 of n11] {$G_0$};

	\node[draw, shape=circle, fill=black, inner sep=2pt, label={left: $1$}, right=3.2 of n11] (n21) {};
	\node[draw, shape=circle, fill=black, inner sep=2pt, label={left: $2$}, below left=0.8 of n21] (n22) {};
	\node[draw, shape=circle, fill=black, inner sep=2pt, label={right: $3$}, below right=0.8 of n21] (n23) {};
	\edge[-] {n21} {n22};
	\node[below=0.9 of n21] {$G_{12}$};

	\node[draw, shape=circle, fill=black, inner sep=2pt, label={left: $1$}, right=3.2 of n21] (n31) {};
	\node[draw, shape=circle, fill=black, inner sep=2pt, label={left: $2$}, below left=0.8 of n31] (n32) {};
	\node[draw, shape=circle, fill=black, inner sep=2pt, label={right: $3$}, below right=0.8 of n31] (n33) {};
	\edge[-] {n31} {n33};
	\node[below=0.9 of n31] {$G_{13}$};

	\node[draw, shape=circle, fill=black, inner sep=2pt, label={left: $1$}, right=3.2 of n31] (n41) {};
	\node[draw, shape=circle, fill=black, inner sep=2pt, label={left: $2$}, below left=0.8 of n41] (n42) {};
	\node[draw, shape=circle, fill=black, inner sep=2pt, label={right: $3$}, below right=0.8 of n41] (n43) {};
	\edge[-] {n42} {n43};
	\node[below=0.9 of n41] {$G_{23}$};

	\node[draw, shape=circle, fill=black, inner sep=2pt, label={left: $1$}, below=2.2 of n11] (n51) {};
	\node[draw, shape=circle, fill=black, inner sep=2pt, label={left: $2$}, below left=0.8 of n51] (n52) {};
	\node[draw, shape=circle, fill=black, inner sep=2pt, label={right: $3$}, below right=0.8 of n51] (n53) {};
	\edge[-] {n51} {n52};
	\edge[-] {n51} {n53};
	\node[below=0.9 of n51] {$G_{-23}$};

	\node[draw, shape=circle, fill=black, inner sep=2pt, label={left: $1$}, right=3.2 of n51] (n61) {};
	\node[draw, shape=circle, fill=black, inner sep=2pt, label={left: $2$}, below left=0.8 of n61] (n62) {};
	\node[draw, shape=circle, fill=black, inner sep=2pt, label={right: $3$}, below right=0.8 of n61] (n63) {};
	\edge[-] {n61} {n62};
	\edge[-] {n62} {n63};
	\node[below=0.9 of n61] {$G_t$};

	\node[draw, shape=circle, fill=black, inner sep=2pt, label={left: $1$}, right=3.2 of n61] (n71) {};
	\node[draw, shape=circle, fill=black, inner sep=2pt, label={left: $2$}, below left=0.8 of n71] (n72) {};
	\node[draw, shape=circle, fill=black, inner sep=2pt, label={right: $3$}, below right=0.8 of n71] (n73) {};
	\edge[-] {n71} {n73};
	\edge[-] {n72} {n73};
	\node[below=0.9 of n71] {$G_{-12}$};

	\node[draw, shape=circle, fill=black, inner sep=2pt, label={left: $1$}, right=3.2 of n71] (n81) {};
	\node[draw, shape=circle, fill=black, inner sep=2pt, label={left: $2$}, below left=0.8 of n81] (n82) {};
	\node[draw, shape=circle, fill=black, inner sep=2pt, label={right: $3$}, below right=0.8 of n81] (n83) {};
	\edge[-] {n81} {n82};
	\edge[-] {n81} {n83};
	\edge[-] {n82} {n83};
	\node[below=0.9 of n81] {$G_c$};
	\end{tikzpicture}
	\caption{Enumerating all 3-node decomposable graphs in $\mathcal{D}_3$ with $G_t$ as the true graph, $G_0$ as the null graph and $G_c$ as the complete graph.} \label{3nodegraph}
\end{figure}
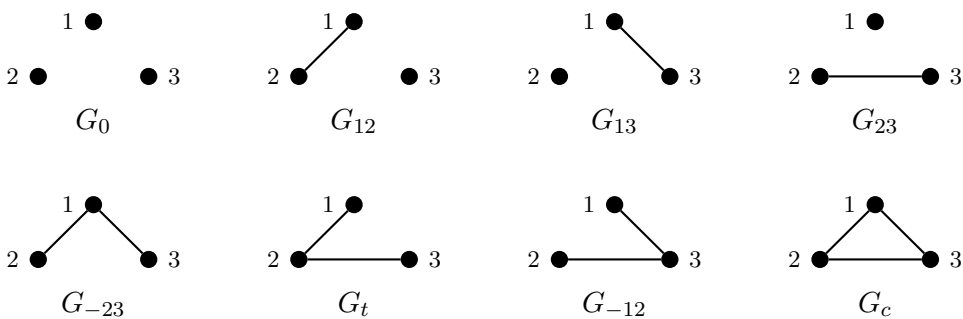

The underlying covariance matrix $\Sigma_3$ and its precision matrix $\Omega_3$ are shown below along with the correlation matrix $R_3$ and the partial correlation matrix $\overline{R}_3$. Samples are drawn independent and identically from $N_3(\bm{0},\Sigma_3)$. The range of the sample size simulated is from 100 to 10,000 with an increment of 100. The Bayes factor for each sample size is averaged over 1000 simulation replicates. The degree of freedom $b$ in the HIW g-prior is chosen to be 3. The first six pairwise Bayes factors in logarithmic scale  is shown  in Figure \ref{sim3node} (a) and the logarithm of $\mbox{BF}(G_c;G_t)$ is shown separately in Figure \ref{sim3node} (b) due to its slower convergence rate. To better understand the simulation results, asymptotic leading terms of pairwise Bayes factors in logarithmic scale and the empirically estimated slopes for $n$ or $\log n$  are listed in the second and third columns of Table \ref{leadingterms}. To calculate the leading terms in the logarithm of Bayes factors, the sample partial correlations or sample correlations are replaced with their population counterparts that do not depend on $n$. The leading terms are obtained by following the route we have used in the proof, i.e. $G_t\rightarrow G_c\rightarrow G_a$. The slopes of logarithms of the first six Bayes factors in Figure \ref{sim3node} (a) are calculated in Table \ref{leadingterms} based on linear regression fit on $n$. The last slope in Table \ref{leadingterms} is calculated based on linear regression on  $\log n$; refer to Figure \ref{sim3node} (b). Table \ref{leadingterms} shows that the theoretical asymptotic leading terms match well with the empirical values.  
\small
\begin{equation*}
	\Sigma_3 =
	\begin{bmatrix}
		 \phantom{-}0.7119 &           -0.4237 &  \phantom{-}0.1695 \\
		           -0.4237 & \phantom{-}0.8475 &            -0.3390 \\
		 \phantom{-}0.1695 &           -0.3390 &  \phantom{-}0.6356
	\end{bmatrix}, \quad
	\Omega_3 = 
	\begin{bmatrix}
		\phantom{0.}2 & \phantom{0.}1 & \phantom{0.}0 \\
		\phantom{0.}1 & \phantom{0.}2 &           0.8 \\
		\phantom{0.}0 &           0.8 & \phantom{0.}2
	\end{bmatrix}.
\end{equation*}
\normalsize
\small
\begin{equation*}
	R_3 =
	\begin{bmatrix}
		 \phantom{-}1.0000 &           -0.5456 &  \phantom{-}0.2520 \\
		           -0.5456 & \phantom{-}1.0000 &            -0.4619 \\
		 \phantom{-}0.2520 &           -0.4619 &  \phantom{-}1.0000
	\end{bmatrix}, \quad
	\overline{R}_3 = 
	\begin{bmatrix}
		\phantom{0.}1 &           0.5 & \phantom{0.}0 \\
		          0.5 & \phantom{0.}1 &           0.4 \\
		\phantom{0.}0 &           0.4 & \phantom{0.}1
	\end{bmatrix}.
\end{equation*}
\normalsize
\bgroup
\def\arraystretch{1.2}
\begin{table}[H]
	\caption{Asymptotic leading terms and simulation slopes of Bayes factors in logarithmic scale} \label{leadingterms}
	\centering
	\scalebox{0.9}{
	\begin{tabular}{cl|l|l}
		& \textbf{Bayes factors}   & \textbf{asymptotic leading term} & \textbf{simulation slope} \\
		\hline
		& $\mbox{BF}(G_0;G_t)$     & $\big\{\log\big(1-\rho^2_{12}\big)+\log\big(1-\rho^2_{23\phantom{_{\mid 1}}}\big)\big\}\cdot n/2 = \bm{-0.2967\cdot n}$ 
		                                                                                                           & $\bm{-0.2963}$ \\ [0.5ex]
		& $\mbox{BF}(G_{13};G_t)$  & $\big\{\log\big(1-\rho^2_{12}\big)+\log\big(1-\rho^2_{23\mid 1}\big)\big\}\cdot n/2=\bm{-0.2639\cdot n}$ 
		                                                                                                           & $\bm{-0.2637}$ \\ [0.5ex]
		& $\mbox{BF}(G_{23};G_t)$  & $\log\big(1-\rho^2_{12\phantom{_{\mid 3}}}\big)\cdot n/2=\bm{-0.1767\cdot n}$ & $\bm{-0.1765}$ \\ [0.5ex]
		& $\mbox{BF}(G_{-12};G_t)$ & $\log\big(1-\rho^2_{12\mid 3}\big)\cdot n/2=\bm{-0.1438\cdot n}$              & $\bm{-0.1439}$ \\ [0.5ex]
		& $\mbox{BF}(G_{12};G_t)$  & $\log\big(1-\rho^2_{23\phantom{_{\mid 1}}}\big)\cdot n/2=\bm{-0.1120\cdot n}$ & $\bm{-0.1198}$ \\ [0.5ex]
		& $\mbox{BF}(G_{-23};G_t)$ & $\log\big(1-\rho^2_{23\mid 1}\big)\cdot n/2=\bm{-0.0872\cdot n}$              & $\bm{-0.0873}$ \\ [0.5ex]
		& $\mbox{BF}(G_c;G_t)$     & $\bm{-0.5\cdot \log n}$                                                       & $\bm{-0.5106}$
	\end{tabular}
	}
\end{table}
\egroup
From the simulation results, we can see missing at least one true edge of $G_t$ in $G_a$ will result in the Bayes factor converging to zero exponentially. This is perfectly illustrated by all six Bayes factors in Figure \ref{sim3node} (a). On the other hand, adding false edges in $G_a$ results in a Bayes factor going to zero at a polynomial rate which is much slower than missing a true edge, see Figure \ref{sim3node} (b). These discoveries are consistent with Table \ref{leadingterms} and our proofs. 

Next we compare the different types of rates in the convergence of the first six Bayes factors. The convergence rate associated with missing two edges of $G_t$ is faster than missing only one edge, i.e. $\mbox{BF}(G_0;G_t)$ vs. $\mbox{BF}(G_{23};G_t)$ and $\mbox{BF}(G_0;G_t)$ vs. $\mbox{BF}(G_{12};G_t)$. The convergence rate is faster when the missing edge of $G_t$ corresponds to a larger partial correlation (or correlation) in absolute value, i.e. $\mbox{BF}(G_{-12};G_t)$ vs. $\mbox{BF}(G_{-23};G_t)$ and $\mbox{BF}(G_{23};G_t)$ vs. $\mbox{BF}(G_{12};G_t)$. One interesting fact is although $G_0$ and $G_{13}$ are both missing two edges of $G_t$, with $G_{13}$ having an additional false edge of $G_t$  compared to $G_0$, the convergence rate of the Bayes factor for $G_{13}$ is slower than that for $G_0$. The reason is clear from Table \ref{leadingterms}. As the absolute value of correlation between node 2 and 3 ($|\rho_{23}|=0.4619$) is larger than the absolute value of partial correlation between them given node 1 ($|\rho_{23\mid 1}|=0.4$), the leading term of $\mbox{BF}(G_0;G_t)$ is smaller than that of $\mbox{BF}(G_{13};G_t)$. The effect due to false edges (polynomial rate) is overwhelmed by the leading term (exponential rate). It is evident that HIW prior places higher penalties on false negative edges compared to false positive edges. Hence in the  high-dimensional case, a prior on graph space is needed for penalizing false positive edges. Similar conclusions can be made comparing $\mbox{BF}(G_{23};G_t)$ and $\mbox{BF}(G_{-12};G_t)$, also from comparing $\mbox{BF}(G_{12};G_t)$ and $\mbox{BF}(G_{-23};G_t)$.

\begin{figure}[H]
	\centering
	\includegraphics[width=1\textwidth]{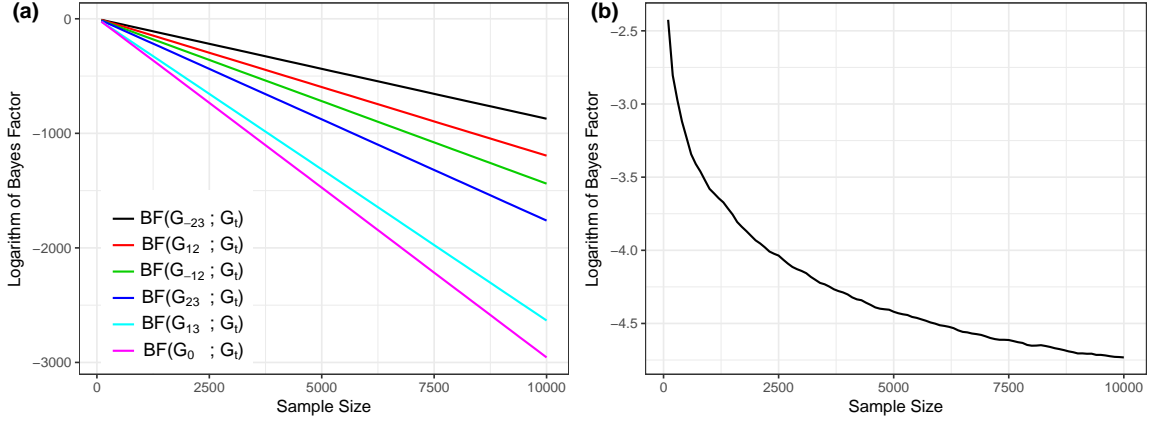}
	\caption{Simulation results of pairwise Bayes factors of $\mathcal{D}_3$ in logarithmic scale. (a) Six Bayes factors where $G_t\not\subset G_a$ (at least missing one edge in $G_t$). (b) When $G_t\subsetneq G_a=G_c$ (only addition).} \label{sim3node}
\end{figure}

\subsection{Simulation 2: Examination of model misspecification}\label{ssec:mis}
In this section, we illustrate the stochastic equivalence between minimal triangulations when the true graph is non-decomposable. The smallest non-decomposable graph is a cycle of length 4 without a chord. So we focus our simulation in $\mathcal{D}_4$. Since the number of decomposable graph increases exponentially with the dimension of graphs, we only select 5 alternative graphs in $\mathcal{D}_4$ other than the minimal triangulations, see Figure \ref{4nodegraph}. The true covariance matrix $\Sigma_4$ and its precision matrix $\Omega_4$ are listed below along with the correlation matrix $R_4$ and the partial correlation matrix $\overline{R}_4$. All simulation settings are the same as in the simulation of $\mathcal{D}_3$. 
\small
\begin{equation*}
	\Sigma_4 =
	\begin{bmatrix}
		 \phantom{-}1.8364 &           -1.0909 &  \phantom{-}0.8909 &           -1.3636 \\
		           -1.0909 & \phantom{-}1.0606 &            -0.7273 & \phantom{-}0.9091 \\
		 \phantom{-}0.8909 &           -0.7273 &  \phantom{-}0.9273 &           -0.9091 \\
		           -1.3636 & \phantom{-}0.9091 &            -0.9091 & \phantom{-}1.6364
	\end{bmatrix}, \quad
	\Omega_4 =
	\begin{bmatrix}
		 \phantom{0.}2 &           1.2 &  \phantom{0.}0 & \phantom{0.}1 \\
		           1.2 & \phantom{0.}3 &            1.2 & \phantom{0.}0 \\
		 \phantom{0.}0 &           1.2 &  \phantom{0.}3 & \phantom{0.}1 \\
		 \phantom{0.}1 & \phantom{0.}0 &  \phantom{0.}1 & \phantom{0.}2
	\end{bmatrix}.
\end{equation*}
\normalsize
\small
\begin{equation*}
	R_4 =
	\begin{bmatrix}
		 \phantom{-}1.0000 &           -0.7817 &  \phantom{-}0.6827 &           -0.7866 \\
		           -0.7817 & \phantom{-}1.0000 &            -0.7334 & \phantom{-}0.6901 \\
		 \phantom{-}0.6827 &           -0.7334 &  \phantom{-}1.0000 &           -0.7380 \\
		           -0.7866 & \phantom{-}0.6901 &            -0.7380 & \phantom{-}1.0000
	\end{bmatrix}, \quad
	\overline{R}_4 =
	\begin{bmatrix}
		\phantom{0.0}1 &           0.49 & \phantom{0.0}0 &           0.50 \\
		          0.49 & \phantom{0.0}1 &           0.40 & \phantom{0.0}0 \\
		 \phantom{0.}0 &           0.40 & \phantom{0.0}1 &           0.41 \\
		          0.50 & \phantom{0.0}0 &           0.41 & \phantom{0.0}1
	\end{bmatrix}.
\end{equation*}
\normalsize

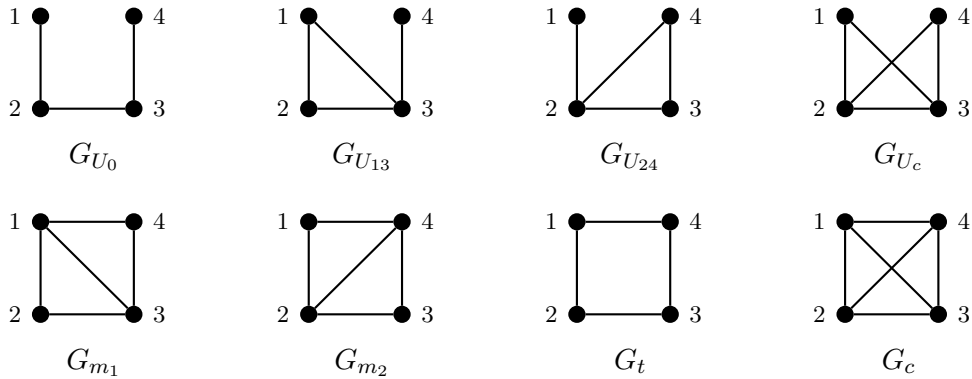
\begin{figure}[H]
	\centering
	\begin{tikzpicture}[thick]
	\node[draw, shape=circle, fill=black, inner sep=2pt, label={left: $1$}] (n11) {};
	\node[draw, shape=circle, fill=black, inner sep=2pt, label={left: $2$}, below=1 of n11] (n12) {};
	\node[draw, shape=circle, fill=black, inner sep=2pt, label={right: $3$}, right=1 of n12] (n13) {};
	\node[draw, shape=circle, fill=black, inner sep=2pt, label={right: $4$}, right=1 of n11] (n14) {};
	\edge[-] {n11} {n12};
	\edge[-] {n12} {n13};
	\edge[-] {n13} {n14};
	\node[right=0.45 of n11] (n10) {};
	\node[below=1.4 of n10] {$G_{U_0}$};

	\node[draw, shape=circle, fill=black, inner sep=2pt, label={left: $1$}, right=3.3 of n11] (n21) {};
	\node[draw, shape=circle, fill=black, inner sep=2pt, label={left: $2$}, below=1 of n21] (n22) {};
	\node[draw, shape=circle, fill=black, inner sep=2pt, label={right: $3$}, right=1 of n22] (n23) {};
	\node[draw, shape=circle, fill=black, inner sep=2pt, label={right: $4$}, right=1 of n21] (n24) {};
	\edge[-] {n21} {n22};
	\edge[-] {n22} {n23};
	\edge[-] {n23} {n24};
	\edge[-] {n21} {n23};
	\node[right=0.45 of n21] (n20) {};
	\node[below=1.4 of n20] {$G_{U_{13}}$};

	\node[draw, shape=circle, fill=black, inner sep=2pt, label={left: $1$}, right=3.3 of n21] (n31) {};
	\node[draw, shape=circle, fill=black, inner sep=2pt, label={left: $2$}, below=1 of n31] (n32) {};
	\node[draw, shape=circle, fill=black, inner sep=2pt, label={right: $3$}, right=1 of n32] (n33) {};
	\node[draw, shape=circle, fill=black, inner sep=2pt, label={right: $4$}, right=1 of n31] (n34) {};
	\edge[-] {n31} {n32};
	\edge[-] {n32} {n33};
	\edge[-] {n33} {n34};
	\edge[-] {n32} {n34};
	\node[right=0.45 of n31] (n30) {};
	\node[below=1.4 of n30] {$G_{U_{24}}$};

	\node[draw, shape=circle, fill=black, inner sep=2pt, label={left: $1$}, right=3.3 of n31] (n41) {};
	\node[draw, shape=circle, fill=black, inner sep=2pt, label={left: $2$}, below=1 of n41] (n42) {};
	\node[draw, shape=circle, fill=black, inner sep=2pt, label={right: $3$}, right=1 of n42] (n43) {};
	\node[draw, shape=circle, fill=black, inner sep=2pt, label={right: $4$}, right=1 of n41] (n44) {};
	\edge[-] {n41} {n42};
	\edge[-] {n42} {n43};
	\edge[-] {n43} {n44};
	\edge[-] {n41} {n43};
	\edge[-] {n42} {n44};
	\node[right=0.45 of n41] (n40) {};
	\node[below=1.4 of n40] {$G_{U_c}$};

	\node[draw, shape=circle, fill=black, inner sep=2pt, label={left: $1$}, below=2.5 of n11] (n51) {};
	\node[draw, shape=circle, fill=black, inner sep=2pt, label={left: $2$}, below=1 of n51] (n52) {};
	\node[draw, shape=circle, fill=black, inner sep=2pt, label={right: $3$}, right=1 of n52] (n53) {};
	\node[draw, shape=circle, fill=black, inner sep=2pt, label={right: $4$}, right=1 of n51] (n54) {};
	\edge[-] {n51} {n52};
	\edge[-] {n52} {n53};
	\edge[-] {n53} {n54};
	\edge[-] {n51} {n53};
	\edge[-] {n51} {n54};
	\node[right=0.45 of n51] (n50) {};
	\node[below=1.4 of n50] {$G_{m_1}$};

	\node[draw, shape=circle, fill=black, inner sep=2pt, label={left: $1$}, right=3.3 of n51] (n61) {};
	\node[draw, shape=circle, fill=black, inner sep=2pt, label={left: $2$}, below=1 of n61] (n62) {};
	\node[draw, shape=circle, fill=black, inner sep=2pt, label={right: $3$}, right=1 of n62] (n63) {};
	\node[draw, shape=circle, fill=black, inner sep=2pt, label={right: $4$}, right=1 of n61] (n64) {};
	\edge[-] {n61} {n62};
	\edge[-] {n62} {n63};
	\edge[-] {n63} {n64};
	\edge[-] {n62} {n64};
	\edge[-] {n61} {n64};
	\node[right=0.45 of n61] (n60) {};
	\node[below=1.4 of n60] {$G_{m_2}$};

	\node[draw, shape=circle, fill=black, inner sep=2pt, label={left: $1$}, right=3.3 of n61] (n71) {};
	\node[draw, shape=circle, fill=black, inner sep=2pt, label={left: $2$}, below=1 of n71] (n72) {};
	\node[draw, shape=circle, fill=black, inner sep=2pt, label={right: $3$}, right=1 of n72] (n73) {};
	\node[draw, shape=circle, fill=black, inner sep=2pt, label={right: $4$}, right=1 of n71] (n74) {};
	\edge[-] {n71} {n72};
	\edge[-] {n72} {n73};
	\edge[-] {n73} {n74};
	\edge[-] {n71} {n74};
	\node[right=0.45 of n71] (n70) {};
	\node[below=1.4 of n70] {$G_t$};

	\node[draw, shape=circle, fill=black, inner sep=2pt, label={left: $1$}, right=3.3 of n71] (n81) {};
	\node[draw, shape=circle, fill=black, inner sep=2pt, label={left: $2$}, below=1 of n81] (n82) {};
	\node[draw, shape=circle, fill=black, inner sep=2pt, label={right: $3$}, right=1 of n82] (n83) {};
	\node[draw, shape=circle, fill=black, inner sep=2pt, label={right: $4$}, right=1 of n81] (n84) {};
	\edge[-] {n81} {n82};
	\edge[-] {n82} {n83};
	\edge[-] {n83} {n84};
	\edge[-] {n81} {n84};
	\edge[-] {n81} {n83};
	\edge[-] {n82} {n84};
	\node[right=0.45 of n81] (n80) {};
	\node[below=1.4 of n80] {$G_c$};
	\end{tikzpicture}
	\caption{Some selected graphs in $\mathcal{G}_4$, including $G_t$ as the true graph which is non-decomposable. $G_{m_1}$ and $G_{m_2}$ are two minimal triangulations of $G_t$.} \label{4nodegraph}
\end{figure}

Since the true graph $G_t$ is non-decomposable, the two minimal triangulations of $G_t$ act like the pseudo-true graphs. So we plot the first four pairwise Bayes factors where $G_{m_i}\not\subset G_a$, $i=1,2$ for $G_{m_1}$ and $G_{m_2}$ in logarithmic scale together in Figure \ref{sim4node} (a) and (b), respectively. The logarithm of Bayes factor between two minimal triangulations is in Figure \ref{sim4node} (c). Finally, we plot the Bayes factors of one triangulation (i.e. $G_c$, not minimal) of $G_t$ against both minimal triangulations in Figure \ref{sim4node} (d).

From Figure \ref{sim4node} (a) and (b), we can see the behavior of two minimal triangulations is the same as what we observed in the case where Bayes factors against the true decomposable graph, i.e. missing true edges causes exponential decay of pairwise Bayes factors. And in the case of false positive edges, i.e. Figure \ref{sim4node} (c), the rate is what we expected if $G_{m_1}$ and $G_{m_2}$ are the true graph, polynomial rate. Based on the simulation result in Figure \ref{sim4node} (c), we can see the Bayes factor between two minimal triangulations neither converges to zero nor diverges to infinity. And they are stochastically bounded. In this case, it is closely to 1 which means these two minimal triangulations of $G_t$ are almost the same in this case (in terms of posterior probability). It is also demonstrated by Figure \ref{sim4node} (a), (b) and (d) where the curves between $G_{m_1}$ and $G_{m_2}$ are almost identical.

\begin{figure}[H]
	\centering
	\includegraphics[width=1\textwidth]{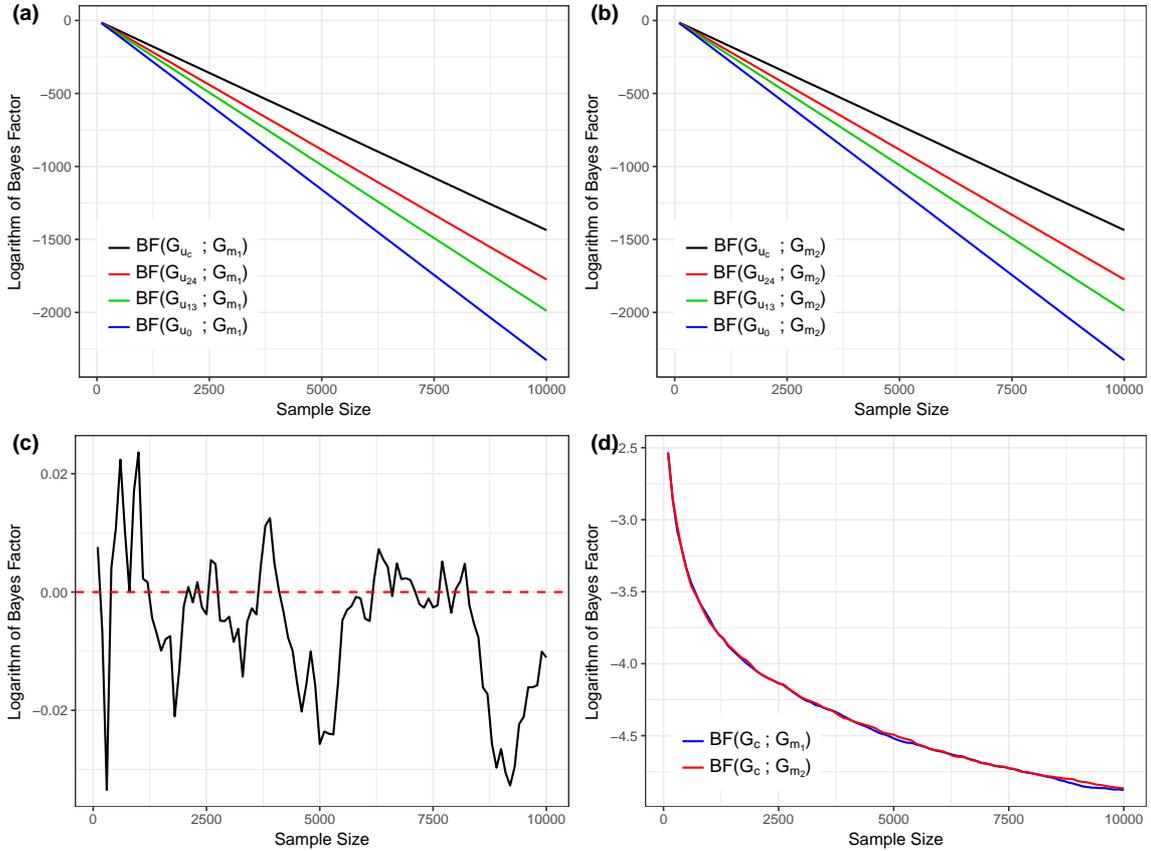}
	\caption{Simulation results of pairwise Bayes factors of $\mathcal{D}_4$ in logarithmic scale. (a) When $G_{m_1}\not\subset G_a$ (missing true edges). (b) When $G_{m_2}\not\subset G_a$ (missing true edges). (c) The Bayes factor between two minimal triangulations of $G_t$, i.e. $\mbox{BF}(G_{m_2};G_{m_1})$. (d) When $G_{m_i}\subsetneq G_a=G_c$, $i=1,2$ (only addition).} \label{sim4node}
\end{figure}

\section{Discussion}\label{sec:disc}
In this paper, we provide a complete theoretical foundation for high-dimensional decomposable graph selection under model misspecification. When the graph dimension is finite, Fitch, Jones and Massam \cite{fitch2014performance} present  pairwise Bayes factor consistency results and stochastic equivalence among minimal triangulations. We  provide more general results of both pairwise consistency and strong selection consistency in high-dimensional scenario. To the best of our knowledge, these are the first complete results on this topic so far. 

In our results, the graph dimension can not be equal to or exceed $n^{1/2}$ and $n^{1/3}$ for pairwise consistency and strong selection consistency, respectively. The limitation of the growth rate of the graph dimension is caused by the convergence rate of sample partial correlations and sample correlations. With the current techniques, without further investigating the relationship among sample partial correlations, these results cannot be improved.   Observe that in i.i.d. case without any sparsity assumptions, it is well-known that the MLE is consistent under ``$p/n$ small'', the Fisher expansion for the MLE is valid under ``$p^2/n$ small'' while the Wilks and asymptotic normality results apply under ``$p^3/n$ small'' \cite{johnstone2010high,spokoiny2013bernstein}. We conjecture that it may not be  possible to relax the growth rate of  $p$ for achieving strong selection consistency using the current formulation of the HIW prior.  This is simply because HIW does not penalize false edges significantly enough so that in high dimension a prior on graph space is needed to achieve both pairwise and strong selection consistency. Also any other sparsity restriction on the elements of the precision matrix is not supported by the HIW prior due to its inability to enforce sufficient shrinkage conditional on the graph. This limits extending the technical results to ultra-high-dimensional case by enforcing additional sparsity assumptions on the elements of the precision matrix. This apparent ``flaw'' lies in the construction of the HIW prior itself and can not be improved by adding any reasonable penalty on the graph space.  

For technical simplicity,  our results are based on HIW $g$-prior only. We conjecture that the consistency results continue to hold for general HIW prior.  Moreover, extensions to non-decomposable graphical models can be done by using $G$-Wishart prior, but major bottlenecks are expected stemming from the lack of a closed form for the normalizing constant for the general HIW prior. Recent work \cite{uhler2018exact} on the development of approximation results for the normalizing constant  may prove to be useful in this regard. 
%T
%he behavior of $G$-Wishart prior for non-decomposable graphs is expected to be similar to the HIW prior for decomposable graphs as they are all built upon sample partial correlations.  

\begin{table}[H]
	\caption{Summary of notations} \label{notations}
	\begin{center}
		\begin{tabular}{cl|l}
			& \textbf{Symbol} & \textbf{definition} \\
			\hline
			& $\bbP$  & probability corresponding to the true data generating distribution \\ [0.5ex]
			& $\mathcal{G}_k$, $\mathcal{D}_k$ & $k$-dimensional graph space, $k$-dimensional decomposable graph space \\ [0.5ex]
			& $\mathcal{M}_t$ & the minimal triangulation space of $G_t$ when $G_t$ is non-decomposable \\ [0.5ex]
			& $a\asymp b$ & $C_1a \leq b \leq C_2 a$ for constants $C_1, C_2$ \\ [0.5ex]
			& $a\precsim b$ & $a \leq C_3 b$ for a constant $C_3$ \\ [0.5ex]
			& $A\subset B$, $A\not\subset B$ & $A$ is a subset of $B$, $A$ is not a subset of $B$ \\ [0.5ex]
			& $A\subsetneq B$ & $A\subset B$ and $A\neq B$ \\ [0.5ex]			
			& $\abs{\cdot}$ & absolute value, cardinality of sets or determinant of matrices by context \\ [0.5ex]
			& $\pi(\cdot)$, $\pi(\cdot\mid\mathrm{Y})$ & prior distribution and posterior distribution of graphs \\ [0.5ex]
			& $\mathrm{Y}$, $Y_i^T$, $\mathrm{y}_i$ & $n\times p$ data matrix, row of $\mathrm{Y}$, column of $\mathrm{Y}$ \\ [0.5ex] 
			& $\rho_{ij}$, $\rho_{ij|S}$ & correlation and partial correlation between $X_i$ and $X_j$ given $X_S$ \\ [0.5ex]
			& $\hat{\rho}_{ij}$, $\hat{\rho}_{ij|S}$ & sample correlation and partial correlation between $X_i$ and $X_j$ given $X_S$ \\ [0.5ex]
			& $\rho_L$, $\rho_U$ & the lower and upper bound for all $\rho_{ij|V\backslash\{i,j\}}$, where $(i,j)\in E_t$ \\ [0.5ex]
			& $C_i$, $\mathcal{C}$, $S_i$, $\mathcal{S}$ & clique, set of cliques, separator, set of separators \\ [0.5ex]
			& $G_t$, $G_a$, $G_c$ & the true graph, any decomposable graph, the complete graph \\ [0.5ex]
			& $G_m$, $G_0$ & the minimal triangulation when $G_t$ is non-decomposable, empty graph \\ [0.5ex]
			& $\hat{G}$ & posterior mode in the decomposable graph space \\ [0.5ex]
			& $E_t$, $E_a$, $E_c$, $E_a^1$ & edge set of $G_t$, $G_a$, $G_c$ and $E_a^1=E_a\cap E_t$ \\ [0.5ex]
			& $p$, $V$ & graph dimension, vertex set, where $V=\{1,2,\ldots,p\}$ \\ [0.5ex]
 			& $x$, $\overline{x}$, $\widetilde{x}$ & nodes in the graph \\ [0.5ex]
 			& $i$, $j$ & determined by context, nodes in the graph or indices of nodes \\ [0.5ex]
 			& $S$, $\overline{S}$, $\widetilde{S}$ & separators in the graph \\ [0.5ex]
 			& $d_S$, $q$ & cardinality of separator $S$, prior edge inclusion probability \\ [0.5ex]
 			& $\Delta'_\epsilon$, $\Delta'_\epsilon(n)$, $\Delta''_\epsilon(n)$ & probability regions of sample partial correlations \\ [0.5ex]
 			& $\Pi_{xy}$ & the set of all sets that separates node $x$ and $y$, where $(x,y)\not\in E_t$ \\ [0.5ex]
 			& $G_{\pm(x,y)\in E_t}$ & a graph with/without true edge $(x,y)$ \\  [0.5ex]
 			& $G_{\pm(x,y)\not\in E_t}$ & a graph with/without false edge $(x,y)$ \\ [0.5ex]
 			& $\overline{G}_i^{\,c\rightarrow a}$, $\widetilde{G}_i^{\,t\rightarrow c}$ & the $i$th graph in the sequence from $G_c$ to $G_a$ and $G_t$ to $G_c$ \\ [0.5ex]
 			% & $\mbox{BF}$, $\mbox{PR}$ & Bayes factor, posterior ratio \\ [0.5ex]
		\end{tabular}
	\end{center}
	\label{Table:default}
\end{table}

\appendix
The Appendix begins with a set of auxiliary results related to the concentration and tail behavior of partial correlations, following by bounds for Bayes factor for local moves required to prove Theorem \ref{bfupper0}.  Then we provide a proof of Theorem  \ref{bfupper0} followed by  the proofs of Theorem \ref{thpr}, Theorem \ref{thstrong}, Corollary \ref{decomp-mode}, the minimal triangulation Theorems \ref{finitetri} and \ref{hightri} and Corollary \ref{nondecomp-mode}.  
\section{Some results on sample correlation and sample partial correlation coefficients}
\begin{theorem} \label{thcorr0} {\normalfont{(When the population correlation is zero \cite{anderson1984introduction}).}} % Anderson (1984) p108
	Assume we have $n$ i.i.d. samples from a multivariate Gaussian distribution. If the population correlation between $X_i$ and $X_j$ is zero, i.e. $\rho_{ij}=0$, the density of the corresponding sample correlation coefficient $\hat{\rho}_{ij}$ as defined in Definition \ref{defscorr} is
	\small
	\begin{equation*}
	f_n(r\mid \rho_{ij}=0) = \frac{\Gamma\big\{\frac{1}{2}(n-1)\big\}}{\Gamma\big\{\frac{1}{2}(n-2)\big\}\sqrt{\pi}} (1-r^2)^{\frac{1}{2}(n-4)}.
	\end{equation*}
	\normalsize
\end{theorem}

\begin{theorem} \label{thcorrnot0} {\normalfont{(When the population correlation is nonzero \cite{hotelling1953new}).}} % Hotelling (1953) p200
	The sample correlation coefficient in a sample of $n$ from a bivariate normal distribution with population correlation coefficient $\rho$ is distributed with density
	\small	
	\begin{equation*}
	f_n(r\mid\rho) = \frac{n-2}{\sqrt{2\pi}} \frac{\Gamma(n-1)}{\Gamma(n-\frac{1}{2})} (1-\rho^2)^{\frac{1}{2}(n-1)} (1-r^2)^{\frac{1}{2}(n-4)}(1-\rho r)^{-n+\frac{3}{2}} F\Big( \frac{1}{2}, \frac{1}{2}; n-\frac{1}{2}; \frac{1+\rho r}{2} \Big),
	\end{equation*}
	\normalsize
	where $n>2$, $-1\leq r \leq 1$ and $F(\cdot, \cdot; \cdot; \cdot)$ is the hypergeometric function. When $\rho=0$, the density becomes the same as in Theorem \ref{thcorr0}.
\end{theorem}

% \begin{proposition}
% 	$1+x < e^x$ and $1-x < e^{-x}$, for all $x>0$.
% \end{proposition}

\begin{proposition} \label{mills} {\normalfont(Mill's ratio).}
	Let $\phi(\cdot)$ and $\Phi(\cdot)$ be the pdf and cdf of the standard normal distribution, respectively and $\widetilde{\Phi}(x) = 1 - \Phi(x)$. Then, we have $\phi(x)\big(\frac{1}{x}-\frac{1}{x^3}\big) \leq \widetilde{\Phi}(x)\leq \frac{\phi(x)}{x}$, for all $x>0$.
\end{proposition}
% https://mikespivey.wordpress.com/2011/10/21/normaltails/
% https://math.stackexchange.com/questions/28751/proof-of-upper-tail-inequality-for-standard-normal-distribution/69417#69417
% https://math.stackexchange.com/questions/74151/proof-of-an-estimate-for-the-tail-of-a-normal-distribution/74156#74156

\begin{proposition} \label{gammabound} {\normalfont(Watson's inequality \cite{watson1959note,mortici2010new}).}
	\small
	\begin{equation*}
		\sqrt{x+\frac{1}{4}} < \frac{\Gamma(x+1)}{\Gamma(x+\frac{1}{2})} \leq \sqrt{x+\frac{1}{\pi}} < \sqrt{x+\frac{1}{2}}, \quad \text{for all } x\geq 0.
	\end{equation*}
	\normalsize
\end{proposition}

\begin{theorem} \label{thrate} {\normalfont{(Tail behavior of sample correlation coefficient).}}
	Let $\hat{\rho}_{ij}$ be the sample correlation coefficient between $X_i$ and $X_j$ with $n$ samples from a $p$-dimensional normal distribution and the corresponding population correlation coefficient is $\rho_{ij}$, where $0\leq\abs{\rho_{ij}}<1$. Then
	\small
	\begin{equation*}
			P\big(\abs{\hat{\rho}_{ij}-\rho_{ij}} > \epsilon \big) < \frac{21}{(1-\abs{\rho_{ij}})^2}\frac{\exp(-n\epsilon^2/4)}{\epsilon\sqrt{n}}, \qquad \text{ for any } 0<\epsilon<1-\abs{\rho_{ij}}, n>2. % \big(-\frac{n\epsilon^2}{4}\big)
	\end{equation*}
	\normalsize
\end{theorem}

\begin{proof} 
	First, let $r=\hat{\rho}_{ij}$ and $\rho=\rho_{ij}$, then by Theorem \ref{thcorrnot0}, $f_n(x\mid\rho)$ is the pdf of $r$. Define
	\small
	\begin{equation*}
		P_n(r_0,\rho) = P(r>r_0)= \int_{r_0}^1 f_n(x\mid\rho)dx, \quad -1\leq r_0\leq 1.
	\end{equation*}
	\normalsize
	\noindent By \cite{hotelling1953new}, we have
	\small
	\begin{align*}
		P_n(r_0, \rho) & = \frac{(n-2)\Gamma(n-1)}{\sqrt{2\pi}\Gamma(n-\frac{1}{2})}\Bigg[ M_0 + \frac{2M_0-M_1}{4(2n-1)} + \frac{9(4M_0-4M_1+M_2)}{32(2n-1)(2n+1)} + \ldots \Bigg] \\
		& = \frac{(n-2)\Gamma(n-1)}{\sqrt{2\pi}\Gamma(n-\frac{1}{2})} (M_0 + R),
	\end{align*}
	\normalsize
	where
	\small
	\begin{align*}
		M_k & = \int_{r_0}^1 (1-\rho^2)^{\frac{1}{2}(n-1)}(1-x^2)^{\frac{1}{2}(n-4)}(1-\rho x)^{-n+k+\frac{3}{2}} dx, \qquad k=0,1,2,\ldots , \\
		R & = \frac{2M_0-M_1}{4(2n-1)} + \frac{9(4M_0-4M_1+M_2)}{32(2n-1)(2n+1)} + \ldots,
	\end{align*}
	\normalsize
	and we know that the first term $M_0$ and the rest of the terms have the following inequality \cite{hotelling1953new},
	\small
	\begin{equation*}
		2(2n-1)\frac{1-\abs{\rho}}{3-\abs{\rho}} \leq \frac{M_0}{R} \leq 4(2n-1)\frac{1-\abs{\rho}}{3-\abs{\rho}}.
	\end{equation*}
	\normalsize
	Let $\delta_{\rho} = \frac{1-\abs{\rho}}{3-\abs{\rho}}$. Since $0\leq \abs{\rho}< 1$, then $0<\delta_{\rho}\leq \frac{1}{3}$. We can bound the residual term $R$ by a fraction of $M_0$,
	\small
	\begin{equation*}
		% 0 \leq \frac{M_0}{4\delta_{\rho}(2n-1)} \leq 
		R \leq \frac{M_0}{2\delta_{\rho}(2n-1)} < \frac{M_0}{6\delta_{\rho}},
	\end{equation*}
	\normalsize
	Therefore,
	\small
	\begin{equation*}
		% \frac{(n-2)\Gamma(n-1)}{\sqrt{2\pi}\Gamma(n-\frac{1}{2})} M_0 \leq 
		P_n(r_0, \rho) < \frac{(n-2)\Gamma(n-1)}{\sqrt{2\pi}\Gamma(n-\frac{1}{2})} \bigg(1+\frac{1}{6\delta_{\rho}}\bigg)M_0.
	\end{equation*}
	\normalsize
	\noindent Next, we further simplify the bound of $P_n(r_0, \rho)$. By Proposition \ref{gammabound}, we have
	\small
	\begin{equation*}
		% \frac{1}{\sqrt{2n}} < \frac{1}{\sqrt{n-1}} < 
		\frac{\Gamma(n-1)}{\Gamma(n-\frac{1}{2})} < \frac{1}{\sqrt{n-\frac{5}{4}}} < \frac{1}{\sqrt{n-2}}.
	\end{equation*}
	\normalsize
	Thus,
	\small
	\begin{equation*}
		% \frac{n-2}{2\sqrt{\pi n}}M_0 < 
		P_n(r_0, \rho) < \sqrt{\frac{n-2}{2\pi}}\bigg(1+\frac{1}{6\delta_{\rho}}\bigg)M_0 < \frac{1}{\sqrt{\pi}}\bigg(1+\frac{1}{6\delta_{\rho}}\bigg)\sqrt{n}M_0.
	\end{equation*}
	\normalsize
	\noindent Let $r_0 = \rho + \epsilon > \rho$, where $0<\epsilon \leq 1-\rho$. Next, we calculate the upper bound of $\sqrt{n}M_0$ for $0\leq\rho<1$ and $-1<\rho<0$ separately. But first, when $-1<\rho<1$ and $\rho<\rho+\epsilon\leq x\leq 1$, then $1-\rho x>0$. Observe that,
	\small
	\begin{equation*}
		\sqrt{n}M_0 = \sqrt{n} \int_{\rho+\epsilon}^{1}  \Big(\frac{1-x^2}{1-\rho x}\Big)^{\frac{1}{2}(n-4)} \Big(\frac{1-\rho^2}{1-\rho x}\Big)^{\frac{1}{2}(n-1)} \frac{1}{1-\rho x} dx. 
	\end{equation*}
	\normalsize
	\noindent $(\mathrm{I})$ When $0 \leq \rho <1$.	Since $\rho < \rho + \epsilon \leq x \leq 1$ and $\rho \geq 0$, we have $(1- \rho^2)^{-1} < (1-\rho x)^{-1} \leq (1-\rho)^{-1}$. Then
	\small
	\begin{equation*}
		\sqrt{n}M_0 \leq \frac{\sqrt{n}}{1-\rho}\int_{\rho+\epsilon}^{1} \Big( 1 - \frac{x^2 - \rho x}{1-\rho x}   \Big)^{\frac{1}{2}(n-4)} \Big(1 + \frac{\rho x - \rho^2}{1-\rho x} \Big)^{\frac{1}{2}(n-1)} dx. 
	\end{equation*}
	\normalsize
	Since $0<\frac{x^2 - \rho x}{1-\rho x}\leq 1$ and $0<\frac{\rho x - \rho^2}{1-\rho x}\leq \rho$, we have
	\small
	\begin{align*}
		\sqrt{n}M_0 & \leq \frac{\sqrt{n}}{1-\rho}\int_{\rho+\epsilon}^{1} \exp\Big( -\frac{n}{2}\frac{x^2 - \rho x}{1-\rho x} + 2\frac{x^2 - \rho x}{1-\rho x} \Big ) \exp\Big( \frac{n}{2}\frac{\rho x - \rho^2}{1-\rho x}-\frac{1}{2}\frac{\rho x - \rho^2}{1-\rho x} \Big) dx \\
		& \leq \frac{e^2 \sqrt{n}}{1-\rho}\int_{\rho+\epsilon}^{1} \exp\Big( -\frac{n}{2}\frac{x^2 - \rho x}{1-\rho x} \Big) \exp\Big( \frac{n}{2}\frac{\rho x - \rho^2}{1-\rho x} \Big) dx \\
		& \leq \frac{e^2 \sqrt{n}}{1-\rho}\int_{\rho+\epsilon}^{1} \exp\bigg\{-\frac{n(x-\rho)^2}{2(1-\rho^2)}\bigg\} dx.
	\end{align*}
	\normalsize
	Thus, by Proposition \ref{mills},
	\small
	\begin{align*}
		\sqrt{n}M_0 & \leq e^2\sqrt{2\pi}\sqrt{\frac{1+\rho}{1-\rho}} \widetilde{\Phi}\bigg(\frac{\epsilon\sqrt{n}}{\sqrt{1-\rho^2}}\bigg) \\
		& \leq e^2\sqrt{2\pi} (1+ \rho) \frac{\phi\Big( \frac{\epsilon\sqrt{n}}{\sqrt{1-\rho^2}}\Big)}{\epsilon\sqrt{n}} 
		\leq \frac{\exp(2+\rho/2)}{1-\rho} \cdot \frac{\exp(-n\epsilon^2/4)}{\epsilon\sqrt{n}}.
	\end{align*}
	\normalsize

 	\noindent $(\mathrm{II})$ When $-1<\rho<0$. Since $\rho < \rho + \epsilon \leq x \leq 1$ and $\rho < 0$, we have $(1-\rho)^{-1}\leq(1-\rho x)^{-1}<(1-\rho^2)^{-1}$. Then
 	\small
	\begin{equation*}
		\sqrt{n}M_0 \leq \frac{\sqrt{n}}{1-\rho^2}\int_{\rho+\epsilon}^{1} \Big( 1 - \frac{x^2 - \rho x}{1-\rho x}   \Big)^{\frac{1}{2}(n-4)} \Big(1 + \frac{\rho x - \rho^2}{1-\rho x} \Big)^{\frac{1}{2}(n-1)} dx := \overline{M}.
	\end{equation*}
	\normalsize
	\noindent $(\mathrm{II.1})$ When $\rho+\epsilon<0$,
	\small
	\begin{align*}
		\overline{M} 
		& = \frac{\sqrt{n}}{1-\rho^2} \Bigg\{\int_{\rho+\epsilon}^{0} + \int_{0}^{1} \Big( 1 + \frac{\rho x - x^2}{1-\rho x}   \Big)^{\frac{1}{2}(n-4)} \Big(1 - \frac{\rho^2 - \rho x}{1-\rho x} \Big)^{\frac{1}{2}(n-1)} dx \Bigg\}\\
		& := A + B.
	\end{align*}
	\normalsize
	Since $0\leq\frac{\rho x - x^2}{1-\rho x}\leq \Big( \frac{1-\sqrt{1-\rho^2}}{\rho} \Big)^2$ and $0<\frac{\rho^2 - \rho x}{1-\rho x}\leq\rho^2$ when $\rho< x\leq 0$,
	\small
	\begin{align*}
		A & \leq \frac{\sqrt{n}}{1-\rho^2} \int_{\rho+\epsilon}^{0} \exp\Big( \frac{n}{2}\frac{\rho x - x^2}{1-\rho x} - 2\frac{\rho x - x^2}{1-\rho x} \Big) \exp\Big( -\frac{n}{2}\frac{\rho^2 - \rho x}{1-\rho x} + \frac{1}{2} \frac{\rho^2-\rho x}{1-\rho x} \Big) dx \\
		& \leq \frac{e^{\frac{\rho^2}{2}}\sqrt{n}}{1-\rho^2} \int_{\rho+\epsilon}^{0} \exp\Big( \frac{n}{2}\frac{\rho x - x^2}{1-\rho x} \Big ) \exp\Big(-\frac{n}{2}\frac{\rho^2 - \rho x}{1-\rho x} \Big) dx \\
		& \leq \frac{e^{2-\frac{\rho}{2}}\sqrt{n}}{1-\rho^2} \int_{\rho+\epsilon}^{0} \exp \bigg\{ -\frac{n(x-\rho)^2}{2(1-\rho)} \bigg\} dx, \text{ since $0<\frac{\rho^2}{2}<-\frac{\rho}{2}$}.
	\end{align*}
	\normalsize
	Since $0\leq\frac{x^2 - \rho x}{1-\rho x}\leq 1$ and $0<\rho^2\leq\frac{\rho^2 - \rho x}{1-\rho x}\leq-\rho$ when $\rho< 0 \leq x\leq 1$,
	\small
	\begin{align*}
		B & \leq \frac{\sqrt{n}}{1-\rho^2} \int_{0}^{1} \exp\Big( -\frac{n}{2}\frac{x^2 - \rho x}{1-\rho x} + 2\frac{x^2 - \rho x}{1-\rho x} \Big ) \exp\Big( -\frac{n}{2}\frac{\rho^2 - \rho x}{1-\rho x} + \frac{1}{2} \frac{\rho^2-\rho x}{1-\rho x} \Big) dx \\
		& \leq \frac{e^{2-\frac{\rho}{2}}\sqrt{n}}{1-\rho^2} \int_{0}^{1} \exp\Big( -\frac{n}{2}\frac{x^2 - \rho x}{1-\rho x} \Big) \exp\Big( -\frac{n}{2}\frac{\rho^2 - \rho x}{1-\rho x} \Big) dx \\
		& \leq \frac{e^{2-\frac{\rho}{2}}\sqrt{n}}{1-\rho^2} \int_{0}^{1} \exp\bigg\{-\frac{n(x-\rho)^2}{2(1-\rho)} \bigg\} dx,
	\end{align*}
	\normalsize
	Hence, when $-1<\rho<0$ and $\rho+\epsilon<0$, by Proposition \ref{mills} we have
	\small
	\begin{align*}
		\sqrt{n}M_0 & \leq e^{2-\frac{\rho}{2}}\sqrt{2\pi}\frac{\sqrt{1-\rho}}{1-\rho^2} \widetilde{\Phi}\bigg(\frac{\epsilon\sqrt{n}}{\sqrt{1-\rho}}\bigg) \\
		& \leq \frac{e^{2-\frac{\rho}{2}}\sqrt{2\pi}}{1+\rho} \frac{\phi\big( \frac{\sqrt{n}\epsilon}{\sqrt{1-\rho}} \big)}{\sqrt{n}\epsilon} 
		\leq \frac{\exp(2-\rho/2)}{1+\rho} \cdot \frac{\exp(-n\epsilon^2/4)}{\epsilon\sqrt{n}}.
	\end{align*}
	\normalsize
	\noindent $(\mathrm{II.2})$ When $\rho+\epsilon\geq 0$, similar to $B$, we still have
	\small
	\begin{equation*}
		\sqrt{n}M_0 \leq \frac{e^{2-\frac{\rho}{2}}\sqrt{n}}{1-\rho^2} \int_{\rho+\epsilon}^{1} \exp \bigg\{ -\frac{n(x-\rho)^2}{2(1-\rho)} \bigg\} dx
		\leq \frac{\exp(2-\rho/2)}{1+\rho} \cdot \frac{\exp(-n\epsilon^2/4)}{\epsilon\sqrt{n}}.
	\end{equation*}
	\normalsize
	So when $-1<\rho<1$ and $\rho <\rho + \epsilon < 1$,
	\small
	\begin{align*}
		P(r > \rho + \epsilon) & < \frac{1}{\sqrt{\pi}}\bigg(1+\frac{1}{6\delta_{\rho}}\bigg)\frac{\exp(2+\abs{\rho}/2)}{1-\abs{\rho}} \cdot \frac{\exp(-n\epsilon^2/4)}{\epsilon\sqrt{n}} \\
		& < \frac{7}{1-\abs{\rho}} \bigg(1+\frac{1}{6\delta_{\rho}}\bigg) \frac{\exp(-n\epsilon^2/4)}{\epsilon\sqrt{n}} \\
		& < \frac{10.5}{(1-\abs{\rho})^2} \frac{\exp(-n\epsilon^2/4)}{\epsilon\sqrt{n}}, \text{ for any } 0<\epsilon<1-\rho. 
	\end{align*}
	\normalsize
	\noindent For $P_n(r_0, \rho)$, we only need to consider when $r_0 > \rho$, i.e. $r_0=\rho+\epsilon$. For the case which $r_0<\rho$, i.e. $-1<r_0 = \rho-\epsilon<\rho$, we have the following equality,
	\small
	\begin{align*}
		P(r<\rho-\epsilon) & = 1 - P(r>\rho-\epsilon) \\
		& = 1 - \int_{\rho-\epsilon}^{1} f_n(-x\mid-\rho) dx \\
		& = 1 - \int_{-1}^{\epsilon-\rho} f_n(x\mid-\rho) dx \\
		& = P(r > -\rho + \epsilon) \\
		& < \frac{10.5}{(1-\abs{\rho})^2} \frac{\exp(-n\epsilon^2/4)}{\epsilon\sqrt{n}}, \quad \text{ for any } 0<\epsilon<1+\rho.
	\end{align*}
	\normalsize
	Therefore,
	\small
	\begin{equation*}
		P(\abs{r - \rho}> \epsilon) < \frac{21}{(1-\abs{\rho})^2} \frac{\exp(-n\epsilon^2/4)}{\epsilon\sqrt{n}}, \quad \text{ for any } 0<\epsilon<1-\abs{\rho}.
	\end{equation*}
	\normalsize
\end{proof}

\begin{theorem} \label{distpcorr} {\normalfont{(The CDF of sample partial correlation coefficient \cite{anderson1984introduction}).}} % Anderon (1984) p133
	If the cdf of sample correlation coefficient $\hat{\rho}_{ij}$ based on $n$ samples from a normal distribution with population correlation coefficient $\rho_{ij}$ is denoted by $F(r\mid n,\rho_{ij})$, then the cdf of the sample partial correlation coefficient $\hat{\rho}_{ij\mid s+1,\ldots, p}$, where $i,j<s+1$, based on $n$ samples from a $p$-dimensional normal distribution with population partial correlation coefficient $\rho_{ij\mid s+1,\ldots,p}$ is $F(r\mid n-p+s, \rho_{ij\mid s+1,\ldots,p}).$ 
\end{theorem}
\noindent The next corollary is an immediate result from Theorem \ref{thrate} and \ref{distpcorr}.

\begin{corollary} \label{thratep} {\normalfont{(Tail behavior of sample partial correlation coefficient).}}
	Let $\hat{\rho}_{ij\mid S}$ be the sample partial correlation coefficient between $X_i$ and $X_j$, where $i,j\not\in S$, holding $X_S$ fixed based on $n$ samples from a $p$-dimensional normal distribution and the corresponding population partial correlation coefficient is $\rho_{ij\mid S}$, where $0\leq\abs{\rho_{ij\mid S}}<1$ and $\abs{S}=d_S<p$. Then
	\small
	\begin{equation*}
		P\big(\abs{\hat{\rho}_{ij\mid S} - \rho_{ij\mid S}} > \epsilon\big) < \frac{21}{(1-\abs{\rho_{ij\mid S}})^2}\frac{\exp\big\{-(n-d_S)\epsilon^2/4\big\}}{\epsilon\sqrt{n-d_S}}, \quad 0<\epsilon<1-\abs{\rho_{ij\mid S}}.
	\end{equation*}
	\normalsize
\end{corollary}

Before introducing the next three lemmas, we first define some notations which are used by them and will be carried on using in the following proofs. Let $R_{ij\mid S}=\big\{\abs{\hat{\rho}_{ij\mid S}-\rho_{ij\mid S}}\leq\epsilon\big\}$. If $(i,j)\not\in E_t$, denote the set of all subsets (of $V$) which separate node $i$ and $j$ as $\Pi_{ij}=\big\{ S\subseteq V\backslash\{i,j\}: \rho_{ij\mid S}=0, (i,j)\not\in E_t\big\}$, $1\leq i<j\leq p$. Define
\small
\begin{align*}
	\Delta'_{\epsilon} & = \Big\{\cap_{(i,j)\in E_t} R_{ij\mid V\backslash\{i,j\}}\Big\} \bigcap 
	\Big\{\cap_{\substack{(i,j)\not\in E_t,\\ \forall S\in\Pi_{ij}}} R_{ij\mid S} \Big\}, \text{ when } p<\infty, \\
	\Delta'_{\epsilon}(n) & = \Big\{\cap_{(i,j)\in E_t} R_{ij\mid V\backslash\{i,j\}}\Big\} \bigcap 
	\Big\{\cap_{\substack{(i,j)\not\in E_t,\\ \forall S\in\Pi_{ij}}} R_{ij\mid S} \Big\}, \text{ when } p \text{ grows with } n, \\
	\Delta''_{\epsilon}(n) & = \Big\{\cap_{(i,j)\in E_t} R_{ij\mid V\backslash\{i,j\}}\Big\} \bigcap 
	\Big\{\cap_{(i,j)\not\in E_t} \big(\cap_{S\in\Pi_{ij}} R_{ij\mid S} \big)\Big\}, \text{ when } p \text{ grows with } n,
\end{align*}
\normalsize
where $\cap_{(i,j)\not\in E_t,\,\forall S\in\Pi_{ij}}$ means intersection of $R_{ij\mid S}$ over all pairs of $(i,j)\not\in E_t$ and for each pair any set of $S\in\Pi_{ij}$ can be used. The $n$ in the bracket means the number of intersections depends on $n$. (When $p$ grows with $n$, the number of edges in the true graph depends on $n$ also.)

\begin{lemma} \label{lemmafinite} {\normalfont(Sample partial correlation simultaneous bounds for pairwise Bayes factor in finite graphs).}
	When the graph dimension $p$ is finite, assume $\rho_U\neq 1$. Let $\epsilon_1(n) = \sqrt{\frac{\log (n-p)}{\tau(n-p)}}$. If $\tau>0$, then $\bbP\big(\Delta'_{\epsilon_1}\big) \rightarrow 1$ as $n \rightarrow \infty$.
\end{lemma}
\begin{proof}
	For finite $p$, $\rho_U\neq 1$ is a positive constant which does not depend on $n$. By Corollary \ref{thratep}, we have
	\small
	\begin{align*}
		\bbP\big(\Delta'_{\epsilon_1}\big)
		& \geq 1 - \bbP\Big\{\cup_{(i,j)\in E_t} R^C_{ij\mid V\backslash\{i,j\}}\Big\} - \bbP\Big\{\cup_{\substack{(i,j)\not\in E_t,\\ \forall S\in\Pi_{ij}}} R^C_{ij\mid S}\Big\} \\
		& \geq 1 - 21\bigg\{\frac{\abs{E_t}}{(1-\rho_U)^2}+p^2-\abs{E_t}\bigg\} (n-p)^{-\frac{1}{4\tau} }\Big\{\frac{1}{\tau}\log(n-p)\Big\}^{-\frac{1}{2}} \\
		& \rightarrow 1, \text{ as } n\rightarrow\infty.
	\end{align*}
	\normalsize
\end{proof}

\begin{lemma} \label{lemmainf1} {\normalfont(Sample partial correlation simultaneous bounds for posterior ratio in high-dimensional graphs).}
	Under Assumption \ref{assump-dim}, i.e. the graph dimension $p = O(n^\alpha)$ grows with sample size $n$, where $0<\alpha<1$. Let $\epsilon_2(n) = (n-p)^{-\beta}$. If $0<\beta<\frac{1}{2}$, under Assumption \ref{assump-upper}, then $\bbP\big\{\Delta'_{\epsilon_2}(n)\big\} \rightarrow 1$ as $n\rightarrow\infty$.
\end{lemma}
\begin{proof}
	By Corollary \ref{thratep}, we have
	\small
	\begin{align*}
		\bbP\big\{\Delta'_{\epsilon_2}(n)\big\}
		& \geq 1 - \bbP\Big\{\cup_{(i,j)\in E_t} R^C_{ij\mid V\backslash\{i,j\}}\Big\} - \bbP\Big\{\cup_{\substack{(i,j)\not\in E_t,\\ \forall S\in\Pi_{ij}}} R^C_{ij\mid S}\Big\} \\
		& \geq 1 - 21\bigg\{\frac{\abs{E_t}}{(1-\rho_U)^2}+p^2-\abs{E_t}\bigg\} (n-p)^{\beta-\frac{1}{2}}\exp\Big\{-\frac{1}{4}(n-p)^{1-2\beta}\Big\} \\
		& \rightarrow 1, \text{ as } n\rightarrow\infty.
	\end{align*}
	\normalsize
\end{proof}

\begin{proposition} \label{binombound} {\normalfont(Lower and upper bound of binomial coefficient).}
	\small
	\begin{equation*}
	\Big(\frac{n}{k}\Big)^k \leq \binom{n}{k} \leq \Big(\frac{en}{k} \Big)^k,
	\end{equation*}
	\normalsize
	where $k\leq n$ and $k$, $n$ are positive integers.
\end{proposition}

\begin{lemma} \label{lemmainf2} {\normalfont(Sample partial correlation simultaneous bounds for strong selection consistency in high-dimensional graphs).}
	Under Assumption \ref{assump-dim}, i.e. the graph dimension $p = O(n^\alpha)$ grows with sample size $n$, where $0<\alpha<1$. Let $\epsilon_3(n) = (n-p)^{-\beta}$, where $0<\beta<\frac{1}{2}$. If $\alpha+2\beta<1 $, under Assumption \ref{assump-upper}, then $\bbP\big\{\Delta''_{\epsilon_3}(n)\big\} \rightarrow 1$ as $n \rightarrow \infty$.
\end{lemma}
\begin{proof}
	By Corollary \ref{thratep}, we have
	\small
	\begin{align*}
		\bbP\big\{\Delta''_{\epsilon_3}(n)\big\}
		& \geq 1 - \bbP\Big\{\cup_{(i,j)\in E_t} R^C_{ij\mid V\backslash\{i,j\}}\Big\} - \bbP\Big\{\cup_{(i,j)\not\in E_t} \big(\cup_{S\in\Pi_{ij}} R^C_{ij\mid S}\big)\Big\} \\
		& \geq 1 - \sum_{(i,j)\in E_t}\bbP\big(R^C_{ij\mid V\backslash\{i,j\}}\big) - \sum_{(i,j)\not\in E_t}\sum_{\abs{S}=0}^{p-2}\binom{p-2}{\abs{S}}\bbP\big(R^C_{ij\mid S}\big) \\
		& \geq 1 - \abs{E_t}\bbP\big(R^C_{ij\mid V\backslash\{i,j\}}\big) - \sum_{(i,j)\not\in E_t}\sum_{\abs{S}=0}^{p-2}(2e)^{p/2}\bbP\big(R^C_{ij\mid S}\big) \\
		& \geq 1 - 21\bigg\{\frac{\abs{E_t}}{(1-\rho_U)^2}+p^3e^p\bigg\} (n-p)^{\beta-\frac{1}{2}}\exp\Big\{-\frac{1}{4}(n-p)^{1-2\beta}\Big\} \\
		& \rightarrow 1, \text{ as } n\rightarrow\infty.
	\end{align*}
	\normalsize	
\end{proof}

\begin{proposition} \label{betabound} {\normalfont(Sharp bounds for Beta CDF \cite{segura2016sharp}).}
	Assume $Z\sim Beta(a,b)$, then
	\small
	\begin{align*}
		P(Z\leq z) & < \frac{z^a(1-z)^b}{B(a,b)\{a-(a+b)z\}}, \quad z<\frac{a}{a+b}, \\
		P(Z > z)   & < \frac{z^a(1-z)^b}{B(a,b)\{(a+b)z-a\}}, \quad z>\frac{a}{a+b},
	\end{align*}
	\normalsize
	where $B(a,b)=\frac{\Gamma(a)\Gamma(b)}{\Gamma(a+b)}$.
\end{proposition}

\begin{theorem} \label{exactrate} {\normalfont(Exact convergence rate of sample correlation coefficient when population correlation coefficient is zero).}
	Let $\hat{\rho}_{ij}$ be the sample correlation coefficient between $X_i$ and $X_j$ with $n$ samples from a $p$-dimensional normal distribution. Assume its corresponding population correlation coefficient $\rho_{ij}$ is zero. For any $0<\epsilon<1/2$, there exist two finite constant $0<M_1(\epsilon)<1/4$ and $M_2(\epsilon)>3$, such that
	\small
	\begin{equation*}
		\bbP\Big(\hat{\rho}^2_{ij}<\frac{M_1}{n}\Big) < \epsilon, \quad \bbP\Big(\hat{\rho}^2_{ij}>\frac{M_2}{n}\Big) < \epsilon, \quad \text{ for any } n>3.
	\end{equation*}
	\normalsize
\end{theorem}
\begin{proof}
	By Theorem \ref{thcorr0}, we know $\hat{\rho}^2_{ij}\sim \mbox{Beta}\Big(\frac{1}{2},\frac{n-2}{2}\Big)$. For any given $\epsilon$, where $0<\epsilon<\frac{1}{2}$, let $M_1=\big(\frac{\epsilon}{\epsilon+1}\big)^2<\frac{1}{4}$ and $M_2=6\log\big(\frac{5}{\epsilon}\big)>3$. Thus, $\frac{M_1}{n}<\frac{1/2}{1/2+(n-2)/2}$ and $\frac{M_2}{n}>\frac{1/2}{1/2+(n-2)/2}$.
	By Proposition \ref{betabound},
	\small
	\begin{align*}
		\bbP\Big(\hat{\rho}^2_{ij}<\frac{M_1}{n}\Big) & < \frac{\big(\frac{M_1}{n}\big)^{\frac{1}{2}}\big(1-\frac{M_1}{n}\big)^{\frac{n-2}{2}}}{B\big(\frac{1}{2},\frac{n-2}{2}\big)\big(\frac{1}{2}-\frac{n-1}{2}\frac{M_1}{n}\big)} \\
		& < \frac{\Gamma\big(\frac{n-1}{2}\big)}{\Gamma\big(\frac{n-2}{2}\big)\sqrt{\pi}}\sqrt{\frac{M_1}{n}} \exp\Big(-\frac{M_1}{2}\frac{n-2}{n}\Big) \Big(\frac{1}{2}-\frac{M_1}{2}\frac{n-1}{n}\Big)^{-1} \\
		& < \sqrt{\frac{n-2}{2n}} \sqrt{\frac{M_1}{\pi}} \Big(\frac{1}{2}-\frac{M_1}{2}\Big)^{-1} \\
		& < \frac{\sqrt{M_1}}{1-\sqrt{M_1}} = \epsilon,
	\end{align*}
	\normalsize
	\small
	\begin{align*}
		\bbP\Big(\hat{\rho}^2_{ij}>\frac{M_2}{n}\Big) &< \frac{\big(\frac{M_2}{n}\big)^{\frac{1}{2}}\big(1-\frac{M_2}{n}\big)^{\frac{n-2}{2}}}{B\big(\frac{1}{2},\frac{n-2}{2}\big)\big(\frac{n-1}{2}\frac{M_2}{n}-\frac{1}{2}\big)} \\
		& < \frac{\Gamma\big(\frac{n-1}{2}\big)}{\Gamma\big(\frac{n-2}{2}\big)\sqrt{\pi}}\sqrt{\frac{M_2}{n}} \exp\Big(-\frac{M_2}{2}\frac{n-2}{n}\Big) \Big(\frac{M_2}{2}\frac{n-1}{n}-\frac{1}{2}\Big)^{-1} \\
		& < \sqrt{\frac{M_2}{2\pi}} \exp\Big(-\frac{M_2}{6}\Big)\Big(\frac{M_2}{2}\frac{1}{2}-\frac{1}{2}\Big)^{-1} \\
		& < 5\exp\Big(-\frac{M_2}{6}\Big) = \epsilon.
	\end{align*}
	\normalsize
\end{proof}
\noindent The next corollary is an immediate result from Theorem \ref{distpcorr} and \ref{exactrate}.

\begin{corollary} \label{exactratep} {\normalfont(Exact convergence rate of sample partial correlation coefficient when population partial correlation coefficient is zero).}
	Let $\hat{\rho}_{ij\mid S}$ be the sample partial correlation coefficient between $X_i$ and $X_j$, where $i,j\not\in S$, holding $X_S$ fixed based on n samples from a $p$-dimensional normal distribution. Assume its corresponding population partial correlation coefficient $\rho_{ij\mid S}$ is zero. For any $0<\epsilon<1/2$, there exist two finite constant $0<M_1(\epsilon)<1/4$ and $M_2(\epsilon)>3$, such that
	\small
	\begin{equation*}
		\bbP\Big(\hat{\rho}^2_{ij\mid S}<\frac{M_1}{n-d_S}\Big) < \epsilon, \quad \bbP\Big(\hat{\rho}^2_{ij\mid S}>\frac{M_2}{n-d_S}\Big) < \epsilon, \quad \text{ for any } n>d_S+3, \,d_S=|S|.
	\end{equation*}
	\normalsize
\end{corollary}

\begin{lemma} \label{lemmafinite1} {\normalfont(Sample partial correlation simultaneous sharp bounds when population partial correlations are zero).}
	When the graph dimension $p$ is finite, for any $0<\epsilon<1/2$, there exist two finite constant $0<M_1(\epsilon)<1/4$ and $M_2(\epsilon)>3$, define 
	\small
	\begin{equation*}
		R^0_{ij\mid S}=\bigg\{\frac{M_1}{n}<\hat{\rho}^2_{ij\mid S}<\frac{M_2}{n-p}\bigg\}, \quad \Delta^0_\epsilon=\cap_{\substack{(i,j)\not\in E_t,\\ \forall S\in\Pi_{ij}}} R^0_{ij\mid S},
	\end{equation*}
	\normalsize
	such that $\bbP\big(\Delta^0_\epsilon\big)>1-\epsilon$, when $n>p+3$.
\end{lemma}
\begin{proof}
	For any $0<\epsilon<1/2$, let
	\small
	\begin{equation*}
		M_1 = \Big(\frac{\epsilon/p^2}{\epsilon/p^2+2}\Big)^2, \quad M_2 = 6\log\Big(\frac{10p^2}{\epsilon}\Big).
	\end{equation*}
	\normalsize
	By Theorem \ref{exactrate} and Corollary \ref{exactratep},
	\small
	\begin{equation*}
		\bbP\Big(\hat{\rho}_{ij\mid S}<\frac{M_1}{n}\Big) < \frac{\epsilon}{2p^2}, \quad \bbP\Big(\hat{\rho}_{ij\mid S}>\frac{M_2}{n-p}\Big) < \frac{\epsilon}{2p^2},
	\end{equation*}
	\normalsize
	for all $\hat{\rho}_{ij\mid S}$ such that $(i,j)\not\in E_t$ and $S\in\Pi_{ij}$. Therefore,
	\small
	\begin{align*}
		\bbP\big(\Delta^0_\epsilon\big) & \geq 1 - \sum_{\substack{(i,j)\not\in E_t,\\ \forall S\in\Pi_{ij}}}\bbP\Big(\hat{\rho}^2_{ij\mid S}<\frac{M_1}{n}\Big)
		- \sum_{\substack{(i,j)\not\in E_t,\\ \forall S\in\Pi_{ij}}}\bbP\Big(\hat{\rho}^2_{ij\mid S}>\frac{M_2}{n-p}\Big) \\
		& > 1 - p^2\cdot \frac{\epsilon}{2p^2} - p^2\cdot \frac{\epsilon}{2p^2} = 1 - \epsilon.
	\end{align*}
	\normalsize
\end{proof}

\begin{corollary} \label{lemmafinite2}
	When the graph dimension $p$ grows with $n$, for any $0<\epsilon<1/2$ and any positive integer $\delta$, there exist two finite constant $0<M_1(\epsilon)<1/4$ and $M_2(\epsilon)>3$, define 
	\small
	\begin{equation*}
		R^0_{ij\mid S}=\bigg\{\frac{M_1}{n}<\hat{\rho}^2_{ij\mid S}<\frac{M_2}{n-p}\bigg\}, \quad \Delta^{0+}_\epsilon=\cap_{(i,j,S)\in\overline{E}_t} R^0_{ij\mid S},
 	\end{equation*}
	\normalsize
	where
	\small
 	\begin{equation*}
		\overline{E}_t=\big\{(i,j,S):(i,j)\not\in E_t,S\in\Pi_{ij},|\overline{E}_t|=\delta<\infty\big\}, 
 	\end{equation*}
 	\normalsize
	we have $\bbP\big(\Delta^{0+}_\epsilon\big)>1-\epsilon$, when $n>p+3$.
\end{corollary}
\begin{proof}
	Let
	\small
	\begin{equation*}
		M_1 = \Big(\frac{\epsilon/\delta}{\epsilon/\delta+2}\Big)^2, \quad M_2 = 6\log\Big(\frac{10\delta}{\epsilon}\Big).
	\end{equation*}
	\normalsize
	The rest of the proof proceeds the same as Lemma \ref{lemmafinite1}.
\end{proof}

\section{Enumerating Bayes Factors in the Deletion Case}
\begin{theorem} \label{validdelete} {\normalfont(Condition of proper deletion while maintaining decomposability \cite{frydenberg1989decomposition,lauritzen1996graphical,giudici1999decomposable,thomas2009enumerating}).} % Lauritzen (1996) p19 Lemma 2.19, Thomas&Green (2009) p1234 Section 3
	Removing an edge $(x,y)$ from a decomposable graph $G$ will result in a decomposable graph if and only if node $x$ and $y$ are contained in exactly one clique.
\end{theorem}

For the rest of this paper, we use lower-case letter $x$, $y$ alone or with subscripts to represent nodes in the graph. \ul{We use the term ``deletion'' {\it only} in the case of deleting true edges.} And true edges are the edges in the true graph $G_t$. Let $G_{+(x,y)\in E_t}$ and $G_{-(x,y)\in E_t}$ be any decomposable graph with and without the true edge $(x,y)$, respectively. The remaining edges (excepting the true edge $(x,y)$) stays the same. (Notice $G_{+(x,y)\in E_t}$ does not need to be the true graph, except just containing the true edge $(x,y)$.) Thus $G_{-(x,y)\in E_t}$ can be seen as the result of deleting the true edge $(x,y)$ from $G_{+(x,y)\in E_t}$. From Theorem \ref{validdelete}, we know node $x$ and $y$ are contained in exactly one clique of $G_{+(x,y)\in E_t}$. The following Lemma \ref{lemmadel} provides upper and lower bound for Bayes factor in favor of deleting a true edge. 
% Thus the move from $G_{+(x,y)\in E_t}$ to $G_{-(x,y)\in E_t}$ can be seen as deleting the true edge $(x,y)$.

\begin{lemma} \label{lemmadel} {\normalfont(Bayes factor of deleting one single true edge).}
	Denote $C$ to be the only clique in $G_{+(x,y)\in E_t}$ that contains node $x$ and $y$. Let $S=C\backslash\{x,y\}$. Then,
	\small
	\begin{align*}
		\bigg(1+\frac{1}{g}\bigg)\sqrt{\frac{b+d_S-\frac{1}{2}}{b+n+d_S}}\big(1-\hat{\rho}_{xy\mid S}^2\big)^{\frac{n}{2}} &< {\normalfont\mbox{BF}}\big(G_{-(x,y)\in E_t};G_{+(x,y)\in E_t}\big) \\ 
		&<\bigg(1+\frac{1}{g}\bigg)\sqrt{\frac{b+d_S}{b+n+d_S-\frac{1}{2}}}\big(1-\hat{\rho}_{xy\mid S}^2\big)^{\frac{n}{2}},
	\end{align*}
	\normalsize
where $d_S=|S|<p$. When $S=\emptyset$, $d_S=0$ and the sample partial correlation coefficient $\hat{\rho}_{xy\mid S}$ becomes the sample correlation coefficient $\hat{\rho}_{xy}$.
\end{lemma}

\begin{proof}
To proof this lemma, we enumerate all scenarios and calculate the Bayes factor above for every case. Similar enumeration also appears in \cite{green2013sampling}. \\

\noindent\textbf{CASE 1:} Node $x$ and $y$ are contained in one clique $C$ of $G_{+(x,y)\in E_t}$ which only has node $x$ and $y$. In other words, removing edge $(x,y)$ will result in adding an empty separator to the junction tree and also disconnecting clique $C_1$ and $C_2$, where $C_1$ is the clique before $C$ and $C_2$ is the clique after $C$. They remain unchanged after deleting edge $(x,y)$. This is the special scenario of CASE 2 where $S=\emptyset$. Figure \ref{delete1} illustrates the result of deleting edge $(x,y)$ from $G_{+(x,y)\in E_t}$. Only the parts which are relative to the deletion are shown, the rest of the junction tree is omitted and will remain unchanged after the deletion. We use ellipses to denote cliques and squares to denote separators in the junction tree. % Due to the definition of the junction tree, i.e. there is an unique path between any two nodes in the tree, there is only one path from $C_1$ to $C_2$, which is through clique $\{x,y\}$. Clique $C_1$ and $C_2$ are not connected without clique $\{x,y\}$.
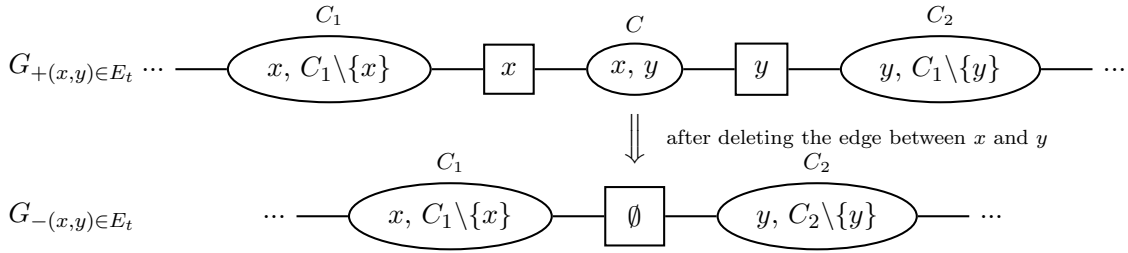
\begin{figure}[H]
	\centering
	\resizebox{\textwidth}{!}{
	\begin{tikzpicture}[thick]
		\node[draw, ellipse, label={above: $C$}] (Cxy) {$x$, $y$};
		\node[draw, regular polygon, regular polygon sides=4, left=0.7 of Cxy] (Sx) {$x$};
		\node[draw, regular polygon, regular polygon sides=4, right=0.7 of Cxy] (Sy) {$y$};
		\edge[-] {Sx} {Cxy}; 
		\edge[-] {Sy} {Cxy}; 
		\node[draw, ellipse, left=0.7 of Sx, , label={above:$C_1$}] (nC1) {$x$, $C_1\backslash \{x\}$};
		\node[draw, ellipse, right=0.7 of Sy, , label={above:$C_2$}] (nC2) {$y$, $C_1\backslash \{y\}$};
		\edge[-] {Sx} {nC1};
		\edge[-] {Sy} {nC2};
		\node[below=0.1 of Cxy, label={right: \scriptsize after deleting the edge between $x$ and $y$}] (DA) {$\Big\Downarrow$};
		\node[draw, regular polygon, regular polygon sides=4, below=0.16 of DA] (S) {$\emptyset$};
		\node[draw, ellipse, left=0.7 of S, label={above:$C_1$}] (C1) {$x$, $C_1\backslash \{x\}$};
		\node[draw, ellipse, right=0.7 of S, label={above:$C_2$}] (C2) {$y$, $C_2\backslash \{y\}$};
		\edge[-] {C1} {S};
		\edge[-] {C2} {S};
		\node[left=1.1 of nC1] (G) {$G_{+(x,y)\in E_t}$};
		\node[below=0.1 of G] (Gb) {$\phantom{\Big\Downarrow}$};
		\node[below=0.26 of Gb] {$G_{-(x,y)\in E_t}$};
		\node[left=0.7 of C1] (etc1) {...};
		\node[left=0.7 of nC1] (etc2) {...};
		\edge[-] {etc1} {C1};
		\edge[-] {etc2} {nC1};
		\node[right=0.7 of C2] (etc3) {...};
		\node[right=0.7 of nC2] (etc4) {...};
		\edge[-] {etc3} {C2};
		\edge[-] {etc4} {nC2};
	\end{tikzpicture}
	}
	\caption{Node $x$ and $y$ are in only one clique of $G_{+(x,y)\in E_t}$ that only contains themselves.} \label{delete1}
\end{figure}
\small
\begin{align*}
	& \mbox{BF}\big(G_{-(x,y)\in E_t};G_{+(x,y)\in E_t}\big) \\
	& = \frac{f(\mathrm{Y}\mid G_{-(x,y)\in E_t})}{f(\mathrm{Y}\mid G_{+(x,y)\in E_t})}
	  = \frac{1}{\frac{w(\{x,y\})}{w(\{x\})\cdot w(\{y\})}}
	  = \frac{w(\{x\})\cdot w(\{y\})}{w(\{x,y\})} \\
	& = \bigg(1+\frac{1}{g}\bigg) \frac{\Gamma_2(\frac{b+1}{2})\Gamma^2(\frac{b+n}{2})}{\Gamma^2(\frac{b}{2})\Gamma_2(\frac{b+n+1}{2})}
	\Bigg(\frac{\abs{\mathrm{Y}_{xy}^T\mathrm{Y}_{xy}}}{\abs{\mathrm{Y}_x^T\mathrm{Y}_x}\cdot\abs{\mathrm{Y}_y^T\mathrm{Y}_y}}\Bigg)^{\frac{n}{2}} \\
	& = \bigg(1+\frac{1}{g}\bigg) \frac{\Gamma(\frac{b+1}{2})\Gamma(\frac{b+n}{2})}{\Gamma(\frac{b}{2})\Gamma(\frac{b+n+1}{2})}
	\Bigg(\frac{\mathrm{Y}_x^T\mathrm{Y}_x \cdot \mathrm{X}_y^T\mathrm{Y}_y - (\mathrm{Y}_x^T\mathrm{Y}_y)^2}{\mathrm{Y}_x^T\mathrm{Y}_x \cdot \mathrm{X}_y^T\mathrm{Y}_y}\Bigg)^{\frac{n}{2}} \\
	& = \bigg(1+\frac{1}{g}\bigg) \frac{\Gamma(\frac{b+1}{2})\Gamma(\frac{b+n}{2})}{\Gamma(\frac{b}{2})\Gamma(\frac{b+n+1}{2})} \big(1-\hat{\rho}_{xy}^2\big)^{\frac{n}{2}}.
\end{align*}
\normalsize
By Proposition \ref{gammabound},
\small
\begin{equation*}
	\sqrt{\frac{b-1}{2}+\frac{1}{4}} < \frac{\Gamma\big(\frac{b+1}{2}\big)}{\Gamma\big(\frac{b}{2}\big)} < \sqrt{\frac{b}{2}}, \qquad
	\frac{1}{\sqrt{\frac{b+n}{2}}} < \frac{\Gamma\big(\frac{b+n}{2}\big)}{\Gamma\big(\frac{b+n+1}{2}\big)} < \frac{1}{\sqrt{\frac{b+n-1}{2}+\frac{1}{4}}}.
\end{equation*}
\normalsize
Thus,
\small
\begin{equation*}
	 \bigg(1+\frac{1}{g}\bigg) \sqrt{\frac{b-\frac{1}{2}}{b+n}} \big(1-\hat{\rho}_{xy}^2\big)^{\frac{n}{2}} < \mbox{BF}\big(G_{-(x,y)\in E_t};G_{+(x,y)\in E_t}\big) < \bigg(1+\frac{1}{g}\bigg) \sqrt{\frac{b}{b+n-\frac{1}{2}}} \big(1-\hat{\rho}_{xy}^2\big)^{\frac{n}{2}}.
\end{equation*}
\normalsize

\noindent\textbf{CASE 2:} Node $x$ and $y$ are contained in only one clique $C$ of $G_{+(x,y)\in E_t}$ which consists of node $x$, $y$ and a non-empty set $S$.
\begin{figure}[H]
	\centering
	\begin{tikzpicture}[thick]
		\node[draw, ellipse, label={above: $C$}] (CxyS) {$x$, $y$, $S$};
		\node[left=1.2 of CxyS] {$G_{+(x,y)\in E_t}$};	
		\node[left=0.7 of CxyS] (etc1) {...};
		\node[right=0.7 of CxyS] (etc2) {...};
		\edge[-] {etc1} {CxyS};
		\edge[-] {etc2} {CxyS};
	\end{tikzpicture}
	\caption{When $S$ is a non-empty set in $G_{+(x,y)\in E_t}$.} \label{delete2}
\end{figure}
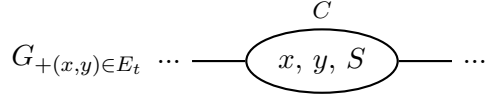

\noindent\textbf{CASE 2.1:} Both $\{x,S\}$ and $\{y,S\}$ are not separators in $G_{+(x,y)\in E_t}$. The cliques containing $\{x,S\}$ and $\{y,S\}$ are exactly $\{x,S\}$ and $\{y,S\}$ after the deletion in $G_{-(x,y)\in E_t}$, respectively \cite{green2013sampling}. Figure \ref{delete2.1} illustrates this scenario. % But $S$ or its subset can be in some other cliques. 
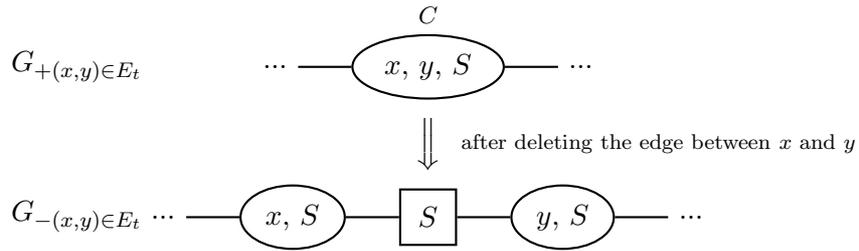
\begin{figure}[H]
	\centering
	\begin{tikzpicture}[thick]
		\node[draw, ellipse, label={above: $C$}] (CxyS) {$x$, $y$, $S$};
		\node[below=0.1 of CxyS, label={right: \scriptsize after deleting the edge between $x$ and $y$}] (DA) {$\Big\Downarrow$};
		\node[draw, regular polygon, regular polygon sides=4, below=0.1 of DA] (S) {$S$};
		\node[draw, ellipse, left=0.7 of S] (C1) {$x$, $S$};
		\node[draw, ellipse, right=0.7 of S] (C2) {$y$, $S$};
		\edge[-] {C1} {S};
		\edge[-] {C2} {S};
		\node[left=2.6 of CxyS] (G) {$G_{+(x,y)\in E_t}$};
		\node[below=0.21 of G] (Gb) {$\phantom{\Big\Downarrow}$}; 
		\node[below=0.1 of Gb] {$G_{-(x,y)\in E_t}$};
		\node[left=0.7 of CxyS] (etc1) {...};
		\node[right=0.7 of CxyS] (etc2) {...};
		\edge[-] {etc1} {CxyS};
		\edge[-] {etc2} {CxyS};
		\node[left=0.7 of C1] (etc3) {...};
		\node[right=0.7 of C2] (etc4) {...};
		\edge[-] {etc3} {C1};
		\edge[-] {etc4} {C2};	
	\end{tikzpicture}
	\caption{Both $\{x,S\}$ and $\{y,S\}$ are not in other cliques of $G_{+(x,y)\in E_t}$.} \label{delete2.1}
\end{figure}

\noindent Let 
\small
\begin{align*}
	\hat{\Sigma}_{SS} &= \mathrm{Y}_{S}^T\mathrm{Y}_{S}, \\
	\hat{H}_S &= \mathrm{Y}_S(\mathrm{Y}^T_S\mathrm{Y}_S)^{-1}\mathrm{Y}_S^T, \\
	\hat{\Sigma}_{xx\mid S} &= \mathrm{Y}^T_{x}\mathrm{Y}_{x}-\mathrm{Y}^T_{x}\hat{H}_S\mathrm{Y}_{x}, \\
	\hat{\Sigma}_{yy\mid S} &= \mathrm{Y}^T_{y}\mathrm{Y}_{y}-\mathrm{Y}^T_{y}\hat{H}_S\mathrm{Y}_{y}, \\
	\hat{\Sigma}_{xy\mid S} &= \mathrm{Y}^T_{x}\mathrm{Y}_{y}-\mathrm{Y}^T_{x}\hat{H}_S\mathrm{Y}_{y}.
\end{align*}
\normalsize
Then we have
\small
\begin{align*}
	\abs{\mathrm{Y}^T_{xyS}\mathrm{Y}_{xyS}} & = 
	\begin{vmatrix}
		\mathrm{Y}^T_{x}\mathrm{Y}_{x} & \mathrm{Y}^T_{x}\mathrm{Y}_{y} & \mathrm{Y}^T_{x}\mathrm{Y}_{S} \\
		\mathrm{Y}^T_{y}\mathrm{Y}_{x} & \mathrm{Y}^T_{y}\mathrm{Y}_{y} & \mathrm{Y}^T_{y}\mathrm{Y}_{S} \\
		\mathrm{Y}^T_{S}\mathrm{Y}_{x} & \mathrm{Y}^T_{S}\mathrm{Y}_{y} & \mathrm{Y}^T_{S}\mathrm{Y}_{S}
	\end{vmatrix}
	= \abs{\mathrm{Y}^T_{S}\mathrm{Y}_{S}}\cdot
	\begin{vmatrix}
		\mathrm{Y}^T_{x}\mathrm{Y}_{x}-\mathrm{Y}^T_{x}\hat{H}_S\mathrm{Y}_{x} & \mathrm{Y}^T_{x}\mathrm{Y}_{y}-\mathrm{Y}^T_{x}\hat{H}_S\mathrm{Y}_{y} \\
		\mathrm{Y}^T_{y}\mathrm{Y}_{x}-\mathrm{Y}^T_{y}\hat{H}_S\mathrm{Y}_{x} & \mathrm{Y}^T_{y}\mathrm{Y}_{y}-\mathrm{Y}^T_{y}\hat{H}_S\mathrm{Y}_{y} 
	\end{vmatrix} \\
	& = \abs{\hat{\Sigma}_{SS}}\cdot \big(\hat{\Sigma}_{xx\mid S}\hat{\Sigma}_{yy\mid S} - \hat{\Sigma}_{xy\mid S}^2\big), \\
	\abs{\mathrm{Y}^T_{xS}\mathrm{Y}_{xS}} & = 
	\begin{vmatrix}
		\mathrm{Y}^T_{x}\mathrm{Y}_{x} & \mathrm{Y}^T_{x}\mathrm{Y}_{S} \\
		\mathrm{Y}^T_{S}\mathrm{Y}_{x} & \mathrm{Y}^T_{S}\mathrm{Y}_{S}
	\end{vmatrix}
	 = \abs{\mathrm{Y}^T_{S}\mathrm{Y}_{S}}\cdot\abs{\mathrm{Y}^T_{x}\mathrm{Y}_{x}-\mathrm{Y}^T_{x}\hat{H}_S\mathrm{Y}_{x}}
	 = \abs{\hat{\Sigma}_{SS}}\cdot\hat{\Sigma}_{xx\mid S}, \\
	 \abs{\mathrm{Y}^T_{yS}\mathrm{Y}_{yS}} & = 
	\begin{vmatrix}
		\mathrm{Y}^T_{y}\mathrm{Y}_{y} & \mathrm{Y}^T_{y}\mathrm{Y}_{S} \\
		\mathrm{Y}^T_{S}\mathrm{Y}_{y} & \mathrm{Y}^T_{S}\mathrm{Y}_{S}
	\end{vmatrix}
	 = \abs{\mathrm{Y}^T_{S}\mathrm{Y}_{S}}\cdot\abs{\mathrm{Y}^T_{y}\mathrm{Y}_{y}-\mathrm{Y}^T_{y}\hat{H}_S\mathrm{Y}_{y}}
	 = \abs{\hat{\Sigma}_{SS}}\cdot\hat{\Sigma}_{yy\mid S}.
\end{align*}
\normalsize
\small
\begin{align*}
	& \mbox{BF}\big(G_{-(x,y)\in E_t};G_{+(x,y)\in E_t}\big) \\
	& = \frac{f(\mathrm{Y}\mid G_{-(x,y)\in E_t})}{f(\mathrm{Y}\mid G_{+(x,y)\in E_t})}
	  = \frac{\frac{w(\{x,S\})\cdot w(\{y,S\})}{w(S)}}{w(\{x,y,S\})}
	  = \frac{w(\{x,S\})\cdot w(\{y,S\})}{w(S)\cdot w(\{x,y,S\})} \\
	& = \bigg( 1 + \frac{1}{g} \bigg)
	\frac{\Gamma_{d_S}\big(\frac{b+d_S-1}{2}\big)\Gamma_{d_S+2}\big(\frac{b+d_S+1}{2}\big)\Gamma^2_{d_S+1}\big(\frac{b+n+d_S}{2}\big)}{\Gamma^2_{d_S+1}\big(\frac{b+d_S}{2}\big)\Gamma_{d_S}\big(\frac{b+n+d_S-1}{2}\big)\Gamma_{d_S+2}\big(\frac{b+n+d_S+1}{2}\big)}
 	\Bigg(\frac{\abs{\mathrm{Y}_{S}^T\mathrm{Y}_{S}}\cdot\abs{\mathrm{Y}_{xyS}^T\mathrm{Y}_{xyS}}}{\abs{\mathrm{Y}_{xS}^T\mathrm{Y}_{xS}}\cdot\abs{\mathrm{Y}_{yS}^T\mathrm{Y}_{yS}}}\Bigg)^{\frac{n}{2}} \\
 	& = \bigg( 1 + \frac{1}{g} \bigg)
 	\frac{\Gamma\big(\frac{b+d_S+1}{2}\big)\Gamma\big(\frac{b+n+d_S}{2}\big)}{\Gamma\big(\frac{b+d_S}{2}\big)\Gamma\big(\frac{b+n+d_S+1}{2}\big)}
	\Bigg( \frac{\hat{\Sigma}_{xx\mid S}\hat{\Sigma}_{yy\mid S} - \hat{\Sigma}_{xy\mid S}^2}{\hat{\Sigma}_{xx\mid S}\hat{\Sigma}_{yy\mid S}} \Bigg)^{\frac{n}{2}} \\
	& = \bigg( 1 + \frac{1}{g} \bigg)
	\frac{\Gamma\big(\frac{b+d_S+1}{2}\big)\Gamma\big(\frac{b+n+d_S}{2}\big)}{\Gamma\big(\frac{b+d_S}{2}\big)\Gamma\big(\frac{b+n+d_S+1}{2}\big)}
	\big(1- \hat{\rho}_{xy\mid S}^2\big)^{\frac{n}{2}}. 
\end{align*}
\normalsize
By Proposition \ref{gammabound},
\small
\begin{equation*}
	\sqrt{\frac{b+d_S-1}{2}+\frac{1}{4}} < \frac{\Gamma\big(\frac{b+d_S+1}{2}\big)}{\Gamma\big(\frac{b+d_S}{2}\big)} < \sqrt{\frac{b+d_S}{2}}, \quad
	\frac{1}{\sqrt{\frac{b+n+d_S}{2}}} < \frac{\Gamma\big(\frac{b+n+d_S}{2}\big)}{\Gamma\big(\frac{b+n+d_S+1}{2}\big)} < \frac{1}{\sqrt{\frac{b+n+d_S-1}{2}+\frac{1}{4}}}.
\end{equation*}
\normalsize
Thus,
\small
\begin{align*}
	\bigg(1+\frac{1}{g}\bigg)\sqrt{\frac{b+d_S-\frac{1}{2}}{b+n+d_S}}\big(1-\hat{\rho}_{xy\mid S}^2\big)^{\frac{n}{2}} &< {\normalfont\mbox{BF}}\big(G_{-(x,y)\in E_t};G_{+(x,y)\in E_t}\big) \\ 
	&<\bigg(1+\frac{1}{g}\bigg)\sqrt{\frac{b+d_S}{b+n+d_S-\frac{1}{2}}}\big(1-\hat{\rho}_{xy\mid S}^2\big)^{\frac{n}{2}}.
\end{align*}
\normalsize

\noindent\textbf{CASE 2.2:} Only one of $\{x,S\}$ and $\{y,S\}$ is a separator in $G_{+(x,y)\in E_t}$. The cliques containing $\{x,S\}$ or $\{y,S\}$ are a superset of $\{x,S\}$ or $\{y,S\}$ after the deletion in $G_{-(x,y)\in E_t}$, respectively \cite{green2013sampling}. Figure \ref{delete2.2.1} shows when $\{x,S\}$ is in other cliques (only one of those supersets is shown here which is $\{x,S,P\}$ and $P\neq\emptyset$, others are omitted for simplicity), thus $\{x,S\}$ is a separator in $G_{+(x,y)\in E_t}$. Figure \ref{delete2.2.2} shows when $\{y,S\}$ is in other cliques (which is $\{y,S,Q\}$ and $Q\neq\emptyset$), thus $\{y,S\}$ is a separator in $G_{+(x,y)\in E_t}$. % But $S$ or its subset can be in some other cliques. 
\begin{figure}[H]
	\centering
	\resizebox{0.9\textwidth}{!}{
	\begin{tikzpicture}[thick]
		\node[draw, regular polygon, regular polygon sides=4, scale=0.8] (SyS) {$x$, $S$};
		\node[draw, ellipse, left=0.7 of SyS] (CxyS) {$x$, $S$, $P$};
		\node[draw, ellipse, right=0.7 of SyS, label={above: $C$}] (nC2) {$x$, $y$, $S$};
		\edge[-] {CxyS} {SyS};
		\edge[-] {nC2} {SyS};
		\node[below=0.1 of SyS, label={right: \scriptsize after deleting the edge between $x$ and $y$}] (DA) {$\Big\Downarrow$};
		\node[draw, regular polygon, regular polygon sides=4, below=0.1 of DA] (S) {$S$};
		\node[draw, ellipse, left=0.7 of S] (C1) {$x$, $S$, $P$};
		\node[draw, ellipse, right=0.7 of S] (C2) {$y$, $S$};
		\edge[-] {C1} {S};
		\edge[-] {C2} {S};
		\node[left=3.6 of SyS] (G) {$G_{+(x,y)\in E_t}$};
		\node[below=0.1 of G] (Gb) {$\phantom{\Big\Downarrow}$};
		\node[below=0.33 of Gb] {$G_{-(x,y)\in E_t}$};
		\node[left=0.7 of CxyS] (etc1) {...};
		\node[left=0.7 of C1] (etc2) {...};
		\node[right=0.7 of nC2] (etc3) {...};
		\node[right=0.7 of C2] (etc4) {...};
		\edge[-] {etc1} {CxyS};
		\edge[-] {etc2} {C1};
		\edge[-] {etc3} {nC2};
		\edge[-] {etc4} {C2};
	\end{tikzpicture}
	}
	\caption{Only $x$ and $S$ are in a superset $\{x,S,P\}$ of $G_{+(x,y)\in E_t}$.} \label{delete2.2.1}
\end{figure}
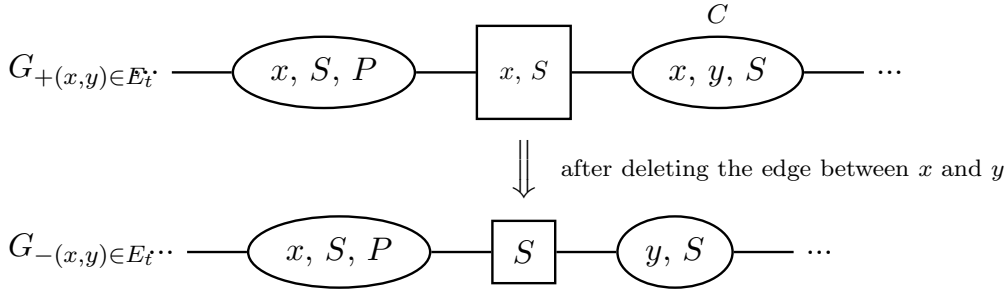
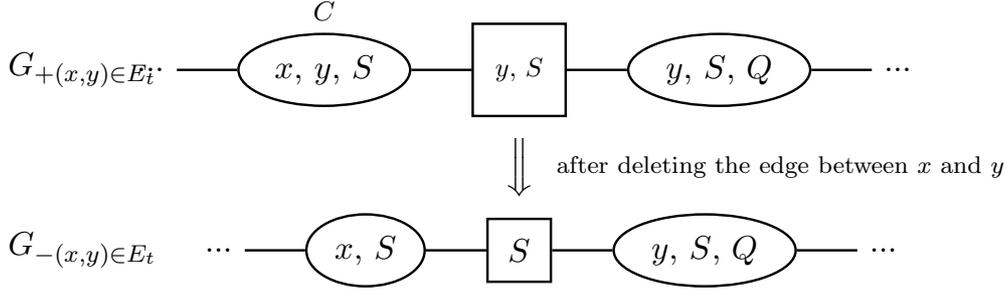
\begin{figure}[H]
	\centering
	\resizebox{0.9\textwidth}{!}{
	\begin{tikzpicture}[thick]
		\node[draw, regular polygon, regular polygon sides=4, scale=0.8] (SyS) {$y$, $S$};
		\node[draw, ellipse, left=0.7 of SyS, label={above: $C$}] (CxyS) {$x$, $y$, $S$};
		\node[draw, ellipse, right=0.7 of SyS] (nC2) {$y$, $S$, $Q$};
		\edge[-] {CxyS} {SyS};
		\edge[-] {nC2} {SyS};
		\node[below=0.1 of SyS, label={right: \scriptsize after deleting the edge between $x$ and $y$}] (DA) {$\Big\Downarrow$};
		\node[draw, regular polygon, regular polygon sides=4, below=0.1 of DA] (S) {$S$};
		\node[draw, ellipse, left=0.7 of S] (C1) {$x$, $S$};
		\node[draw, ellipse, right=0.7 of S] (C2) {$y$, $S$, $Q$};
		\edge[-] {C1} {S};
		\edge[-] {C2} {S};
		\node[left=3.5 of SyS] (G) {$G_{+(x,y)\in E_t}$};
		\node[below=0.1 of G] (Gb) {$\phantom{\Big\Downarrow}$};
		\node[below=0.33 of Gb] {$G_{-(x,y)\in E_t}$};
		\node[left=0.7 of CxyS] (etc1) {...};
		\node[left=0.7 of C1] (etc2) {...};
		\node[right=0.7 of nC2] (etc3) {...};
		\node[right=0.7 of C2] (etc4) {...};
		\edge[-] {etc1} {CxyS};
		\edge[-] {etc2} {C1};
		\edge[-] {etc3} {nC2};
		\edge[-] {etc4} {C2};
	\end{tikzpicture}
	}
	\caption{Only $y$ and $S$ are in a superset $\{y,S,Q\}$ of $G_{+(x,y)\in E_t}$.} \label{delete2.2.2}
\end{figure}
\small
\begin{equation*}
	\mbox{BF}\big(G_{-(x,y)\in E_t};G_{+(x,y)\in E_t}\big) = \frac{w(\{x,S\})\cdot w(\{y,S\})}{w(S)\cdot w(\{x,y,S\})}.
	% \frac{\frac{w(\{x,S\})}{w(S)}}{\frac{w(\{x,y,S\})}{w(\{y,S\})}} = 
\end{equation*}
\normalsize
This is the same as \textbf{CASE 2.1}. \\

\noindent\textbf{CASE 2.3:} Both $\{x,S\}$ and $\{y,S\}$ are separators in $G_{+(x,y)\in E_t}$. The cliques containing both $\{x,S\}$ and $\{y,S\}$ are supersets of them after the deletion in $G_{-(x,y)\in E_t}$ \cite{green2013sampling}. Figure \ref{delete2.3} shows $\{x,S\}$ in superset $\{x,S,P\}$ and $\{y,S\}$ in superset $\{y,S,Q\}$, where $P,Q\neq\emptyset$ and $P\cap Q=\emptyset$, thus $\{x,S\}$ and $\{y,S\}$ are separators in $G_{+(x,y)\in E_t}$. % But $S$ or its subset can be in some other cliques.
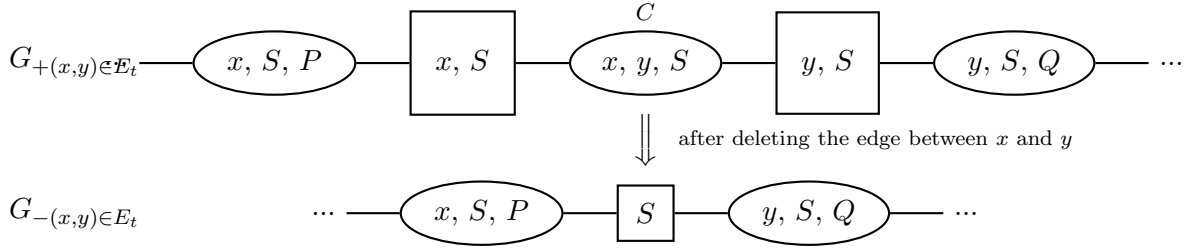
\begin{figure}[H]
	\centering
	\resizebox{1.05\textwidth}{!}{
	\begin{tikzpicture}[thick]
	\node[draw, ellipse, label={above: $C$}] (CxyS) {$x$, $y$, $S$};
	\node[draw, regular polygon, regular polygon sides=4, left=0.7 of CxyS] (SxS) {$x$, $S$};
	\node[draw, regular polygon, regular polygon sides=4, right=0.7 of CxyS] (SyS) {$y$, $S$};
	\edge[-] {SxS} {CxyS};
	\edge[-] {SyS} {CxyS};
	\node[draw, ellipse, left=0.7 of SxS] (nC1) {$x$, $S$, $P$};
	\node[draw, ellipse, right=0.7 of SyS] (nC2) {$y$, $S$, $Q$};
	\edge[-] {SxS} {nC1};
	\edge[-] {SyS} {nC2};
	\node[below=0.1 of CxyS, label={right: {\scriptsize after deleting the edge between $x$ and $y$}}] (DA) {$\Big\Downarrow$};
	\node[draw, regular polygon, regular polygon sides=4, below=0.1 of DA] (S) {$S$};
	\node[draw, ellipse, left=0.7 of S] (C1) {$x$, $S$, $P$};
	\node[draw, ellipse, right=0.7 of S] (C2) {$y$, $S$, $Q$};
	\edge[-] {C1} {S};
	\edge[-] {C2} {S};
	\node[left=5.5 of CxyS] (G) {$G_{+(x,y)\in E_t}$};
	\node[below=0.1 of G] (Gb) {$\phantom{\Big\Downarrow}$};
	\node[below=0.23 of Gb] {$G_{-(x,y)\in E_t}$};
	\node[left=0.7 of nC1] (etc1) {...};
	\node[left=0.7 of C1] (etc2) {...};
	\node[right=0.7 of nC2] (etc3) {...};
	\node[right=0.7 of C2] (etc4) {...};
	\edge[-] {etc1} {nC1};
	\edge[-] {etc2} {C1};
	\edge[-] {etc3} {nC2};
	\edge[-] {etc4} {C2};
	\end{tikzpicture}
	}
	\caption{$\{x,S\}$ and $\{y,S\}$ are in superset $\{x,S,P\}$ and $\{y,S,Q\}$ of $G_{+(x,y)\in E_t}$, respectively.} \label{delete2.3}
\end{figure}
\small
\begin{equation*}
	\mbox{BF}\big(G_{-(x,y)\in E_t};G_{+(x,y)\in E_t}\big) = \frac{w(\{x,S\})\cdot w(\{y,S\})}{w(S)\cdot w(\{x,y,S\})}.
	% \frac{\frac{1}{w(S)}}{\frac{w(\{x,y,S\})}{w(\{x,S\})w(\{y,S\})}} =
\end{equation*}
\normalsize
This is also the same as \textbf{CASE 2.1}.
\end{proof}

\section{Enumerating Bayes factors in the addition case}
\begin{theorem} \label{validadd} {\normalfont(Condition of proper addition while maintaining decomposability \cite{frydenberg1989decomposition,giudici1999decomposable,thomas2009enumerating}).}
	Adding an edge $(x,y)$ to a decomposable graph $G$ will result in a decomposable graph if and only if $x$ and $y$ are unconnected and contained in cliques that are adjacent in some junction tree of $G$.
\end{theorem}

\ul{Notice we use the term ``addition'' {\it only} in the case of adding false edges,} i.e., edges which are not in the true graph $G_t$. Let $G_{+(x,y)\not\in E_t}$ and $G_{-(x,y)\not\in E_t}$ be any decomposable graph with and without the false edge $(x,y)$, respectively. And except the false edge $(x,y)$, the rest of them are the same. ($G_{-(x,y)\not\in E_t}$ does not need to be the true graph, except not having the false edge $(x,y)$.) Therefore, $G_{+(x,y)\not\in E_t}$ can be seen as the result of adding the false edge $(x,y)$ to $G_{-(x,y)\not\in E_t}$. By Theorem \ref{validadd}, we know node $x$ and $y$ are contained in cliques that are adjacent in at least one junction tree of $G_{-(x,y)\not\in E_t}$. Thus we have the following lemma.

\begin{lemma} \label{lemmaadd} {\normalfont(Bayes factor of adding one single false edge).}
	Let $C_1$ and $C_2$ be the cliques which contain $x$ and $y$, respectively. Assume $C_1$ and $C_2$ are two adjacent nodes in at least one junction tree of $G_{-(x,y)\not\in E_t}$. Let $S=C_1\cap C_2$. Then,
	\small
	\begin{align*}
		\bigg(\frac{g}{g+1}\bigg) \sqrt{\frac{b+n+d_S-\frac{1}{2}}{b+d_S}} \big(1-\hat{\rho}^2_{xy\mid S}\big)^{-\frac{n}{2}} &< {\normalfont\mbox{BF}}\big(G_{+(x,y)\not\in E_t};G_{-(x,y)\not\in E_t}\big) \\
		&< \bigg(\frac{g}{g+1}\bigg) \sqrt{\frac{b+n+d_S}{b+d_S-\frac{1}{2}}} \big(1-\hat{\rho}^2_{xy\mid S}\big)^{-\frac{n}{2}},
	\end{align*}
	\normalsize
where $d_S=|S|<p$. When $S=\emptyset$, $d_S=0$ and the sample partial correlation coefficient $\hat{\rho}_{xy\mid S}$ becomes the sample correlation coefficient $\hat{\rho}_{xy}$.
\end{lemma}

\begin{proof}
Similar to the deletion case, we enumerate all scenarios and calculate the corresponding Bayes factors. The addition case can be partially seen as the reversion of the deletion case, only the edge added here is {\it not} a true edge. Same enumeration can be found in the appendix of \cite{giudici1999decomposable}. \\

\noindent\textbf{CASE 1:} Clique $C_1$ and $C_2$ are disconnected in $G_{-(x,y)\not\in E_t}$, i.e. node $x$ and $y$ are not adjacent and not connected. (The graph can be seen as two separate subgraphs.) In other words, adding edge $(x,y)$ will result in creating a new clique to the current junction tree of $G_{-(x,y)\not\in E_t}$, and also connecting clique $C_1$ and $C_2$. They remain unchanged after adding edge $(x,y)$. This is the special scenario of CASE 2 where $S=\emptyset$. Figure \ref{add1} illustrates the result of adding a false edge $(x,y)$ to $G_{-(x,y)\not\in E_t}$. Here $P=C_1\backslash\{x\}$ and $Q=C_2\backslash\{y\}$, thus $P\cap Q=\emptyset$ and $P,Q\neq\emptyset$.
\begin{figure}[H]
	\centering
	\resizebox{\textwidth}{!}{
	\begin{tikzpicture}[thick]
		\node[draw, regular polygon, regular polygon sides=4] (S) {$\emptyset$};
		\node[draw, ellipse, left=0.7 of S, label={above:$C_1$}] (C1) {$x$, $P$};
		\node[draw, ellipse, right=0.7 of S, label={above:$C_2$}] (C2) {$y$, $Q$};
		\edge[-] {C1} {S};
		\edge[-] {C2} {S};
		\node[left=0.7 of C1] (etc1) {...};
		\node[right=0.7 of C2] (etc2) {...};
		\edge[-] {C1} {etc1};
		\edge[-] {C2} {etc2};
		\node[below=0.1 of S, label={right: \scriptsize after adding an edge between $x$ and $y$}] (DA) {$\Big\Downarrow$};
		\node[draw, ellipse, below=0.18 of DA] (Cxy) {$x$, $y$};
		\node[draw, regular polygon, regular polygon sides=4, left=0.7 of Cxy] (Sx) {$x$};
		\node[draw, regular polygon, regular polygon sides=4, right=0.7 of Cxy] (Sy) {$y$};
		\edge[-] {Sx} {Cxy}; 
		\edge[-] {Sy} {Cxy}; 	
		\node[draw, ellipse, left=0.7 of Sx, , label={above:$C_1$}] (nC1) {$x$, $P$};
		\node[draw, ellipse, right=0.7 of Sy, , label={above:$C_2$}] (nC2) {$y$, $Q$};
		\edge[-] {Sx} {nC1};
		\edge[-] {Sy} {nC2};
		\node[left=0.7 of nC1] (etc3) {...};
		\node[right=0.7 of nC2] (etc4) {...};
		\edge[-] {nC1} {etc3};
		\edge[-] {nC2} {etc4};
		\node[left=1 of nC1] (G) {$G_{+(x,y)\not\in E_t}$};
		\node[above=0.18 of G] (Ga) {$\phantom{\Big\Uparrow}$};
		\node[above=0.18 of Ga] {$G_{-(x,y)\not\in E_t}$};
	\end{tikzpicture}
	}
	\caption{Clique $C_1$ and $C_2$ are disconnected in $G_{-(x,y)\not\in E_t}$.} \label{add1}
\end{figure}
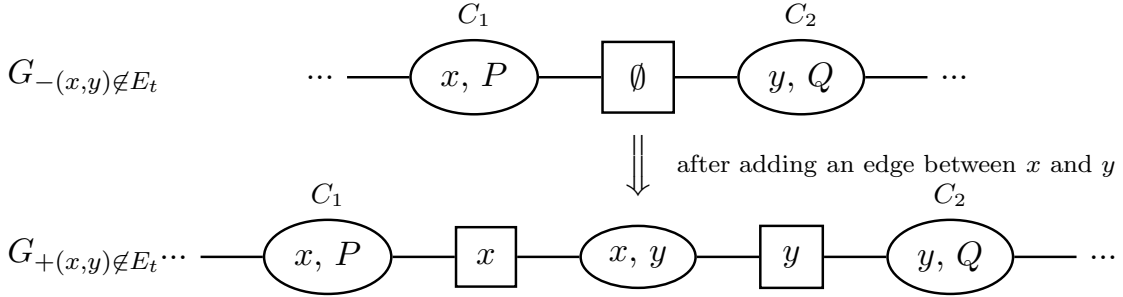
\small
\begin{align*}
	& \mbox{BF}\big(G_{+(x,y)\not\in E_t};G_{-(x,y)\not\in E_t}\big) \\
	& = \frac{f(\mathrm{Y}\mid G_{+(x,y)\not\in E_t})}{f(\mathrm{Y}\mid G_{-(x,y)\not\in E_t})} = \frac{w(\{x,y\})}{w(\{x\})\cdot w(\{y\})} \\
	& = \bigg(\frac{g}{g+1}\bigg) \frac{\Gamma^2(\frac{b}{2})\Gamma_2(\frac{b+n+1}{2})}{\Gamma_2(\frac{b+1}{2})\Gamma^2(\frac{b+n}{2})}
	\Bigg(\frac{\abs{\mathrm{Y}^T_{xy}\mathrm{Y}_{xy}}}{\abs{\mathrm{Y}^T_{x}\mathrm{Y}_{x}}\cdot\abs{\mathrm{Y}^T_{y}\mathrm{Y}_{y}}}\Bigg)^{-\frac{n}{2}} \\
	& = \bigg(\frac{g}{g+1}\bigg) \frac{\Gamma(\frac{b}{2})\Gamma(\frac{b+n+1}{2})}{\Gamma(\frac{b+1}{2})\Gamma(\frac{b+n}{2})}
	\Bigg(\frac{\mathrm{Y}_x^T\mathrm{Y}_x \cdot \mathrm{X}_y^T\mathrm{Y}_y - (\mathrm{Y}_x^T\mathrm{Y}_y)^2}{\mathrm{Y}_x^T\mathrm{Y}_x \cdot \mathrm{X}_y^T\mathrm{Y}_y}\Bigg)^{-\frac{n}{2}} \\
	& = \bigg(\frac{g}{g+1}\bigg) \frac{\Gamma(\frac{b}{2})\Gamma(\frac{b+n+1}{2})}{\Gamma(\frac{b+1}{2})\Gamma(\frac{b+n}{2})} \big(1 - \hat{\rho}_{xy}^2\big)^{-\frac{n}{2}}.
\end{align*}
\normalsize
By Proposition \ref{gammabound},
\small
\begin{equation*}
	\frac{1}{\sqrt{\frac{b}{2}}} < \frac{\Gamma\big(\frac{b}{2}\big)}{\Gamma\big(\frac{b+1}{2}\big)}<\frac{1}{\sqrt{\frac{b-1}{2}+\frac{1}{4}}}, \quad
	\sqrt{\frac{b+n-1}{2}+\frac{1}{4}} < \frac{\Gamma\big(\frac{b+n+1}{2}\big)}{\Gamma\big(\frac{b+n}{2}\big)} < \sqrt{\frac{b+n}{2}}. 
\end{equation*}
\normalsize
\noindent Thus,
\small
\begin{equation*}
	\bigg(\frac{g}{g+1}\bigg)\sqrt{\frac{b+n-\frac{1}{2}}{b}} \big(1 - \hat{\rho}_{xy}^2\big)^{-\frac{n}{2}} < \mbox{BF}\big(G_{+(x,y)\not\in E_t};G_{-(x,y)\not\in E_t}\big) 
	< \bigg(\frac{g}{g+1}\bigg)\sqrt{\frac{b+n}{b-\frac{1}{2}}} \big(1 - \hat{\rho}_{xy}^2\big)^{-\frac{n}{2}}.
\end{equation*}
\normalsize

\noindent\textbf{CASE 2:} Clique $C_1$ and $C_2$ are connected by a non-empty separator $S$ in $G_{-(x,y)\not\in E_t}$ and $P\cap Q=\emptyset$.
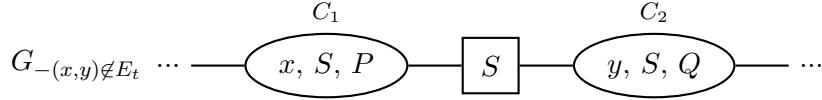
\begin{figure}[H]
	\centering
	\begin{tikzpicture}[thick]
	\node[draw, regular polygon, regular polygon sides=4] (S) {$S$};
	\node[draw, ellipse, left=0.7 of S, label={above:$C_1$}] (C1) {$x$, $S$, $P$};
	\node[draw, ellipse, right=0.7 of S, label={above:$C_2$}] (C2) {$y$, $S$, $Q$};
	\edge[-] {C1} {S};
	\edge[-] {C2} {S};
	\node[left=0.7 of C1] (etc1) {...};
	\node[right=0.7 of C2] (etc2) {...};
	\edge[-] {C1} {etc1};
	\edge[-] {C2} {etc2};
	\node[left=1.2 of C1] {$G_{-(x,y)\not\in E_t}$};
	\end{tikzpicture}
	\caption{When $S$ is a non-empty separator in $G_{-(x,y)\not\in E_t}$.} \label{add2}
\end{figure}

\noindent\textbf{CASE 2.1:} When $P$, $Q$ are both empty sets, i.e. clique $C_1$ contains only $\{x,S\}$ and clique $C_2$ contains only $\{y,S\}$ in $G_{-(x,y)\not\in E_t}$. In this case, adding an edge between $x$ and $y$ will consolidate $C_1$ and $C_2$ to create a single clique which consists of $x$, $y$ and $S$. Figure \ref{add2.1} shows this scenario.
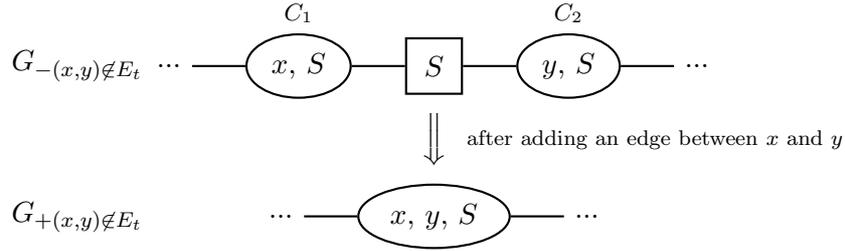
\begin{figure}[H]
	\centering
	\begin{tikzpicture}[thick]
	\node[draw, regular polygon, regular polygon sides=4] (S) {$S$};
	\node[draw, ellipse, left=0.7 of S, label={above:$C_1$}] (C1) {$x$, $S$};
	\node[draw, ellipse, right=0.7 of S, label={above:$C_2$}] (C2) {$y$, $S$};
	\edge[-] {C1} {S};
	\edge[-] {C2} {S};
	\node[left=0.7 of C1] (etc1) {...};
	\node[right=0.7 of C2] (etc2) {...};
	\edge[-] {C1} {etc1};
	\edge[-] {C2} {etc2};
	\node[below=0.1 of S, label={right: \scriptsize after adding an edge between $x$ and $y$}] (DA) {$\Big\Downarrow$};
	\node[draw, ellipse, below=0.1 of DA] (CxyS) {$x$, $y$, $S$};
	\node[left=0.7 of CxyS] (etc3) {...};
	\node[right=0.7 of CxyS] (etc4) {...};
	\edge[-] {CxyS} {etc3};
	\edge[-] {CxyS} {etc4};
	\node[left=1.2 of C1] (G) {$G_{-(x,y)\not\in E_t}$};
	\node[below=0.1 of G] (Gb) {$\phantom{\Big\Downarrow}$};
	\node[below=0.23 of Gb] {$G_{+(x,y)\not\in E_t}$};
	\end{tikzpicture}
	\caption{When $P,Q=\emptyset$, i.e. $C_1=\{x,S\}$ and $C_2=\{y,S\}$ in $G_{-(x,y)\not\in E_t}$.} \label{add2.1}
\end{figure}
\small
\begin{align*}
	& \mbox{BF}\big(G_{+(x,y)\not\in E_t};G_{-(x,y)\not\in E_t}\big) \\
	& = \frac{f(\mathrm{Y}\mid G_{+(x,y)\not\in E_t})}{f(\mathrm{Y}\mid G_{-(x,y)\not\in E_t})} = \frac{w(\{x,y,S\})}{\frac{w(\{x,S\})\cdot w(\{y,S\})}{w(S)}}
	= \frac{w(\{x,y,S\})\cdot w(S)}{w(\{x,S\})\cdot w(\{y,S\})} \\
	& = \bigg(\frac{g}{g+1}\bigg)
	\frac{\Gamma_{d_S+1}^2\big(\frac{b+d_S}{2}\big)\Gamma_{d_S}\big(\frac{b+n+d_S-1}{2}\big)\Gamma_{d_S+2}\big(\frac{b+n+d_S+1}{2}\big)}{\Gamma_{d_S}\big(\frac{b+d_S-1}{2}\big)\Gamma_{d_S+2}\big(\frac{b+d_S+1}{2}\big)\Gamma_{d_S+1}^2\big(\frac{b+n+d_S}{2}\big)}
	\Bigg(\frac{\abs{\mathrm{Y}_S^T\mathrm{Y}_S}\cdot\abs{\mathrm{Y}_{xyS}^T\mathrm{Y}_{xyS}}}{\abs{\mathrm{Y}_{xS}^T\mathrm{Y}_{xS}}\cdot\abs{\mathrm{Y}_{yS}^T\mathrm{Y}_{yS}}}\Bigg)^{-\frac{n}{2}} \\
	& =	\bigg(\frac{g}{g+1}\bigg)
	\frac{\Gamma\big(\frac{b+d_S}{2}\big)\Gamma\big(\frac{b+n+d_S+1}{2}\big)}{\Gamma\big(\frac{b+d_S+1}{2}\big)\Gamma\big(\frac{b+n+d_S}{2}\big)}
	\Bigg( \frac{\hat{\Sigma}_{xx\mid S}\hat{\Sigma}_{yy\mid S} - \hat{\Sigma}_{xy\mid S}^2}{\hat{\Sigma}_{xx\mid S}\hat{\Sigma}_{yy\mid S}} \Bigg)^{-\frac{n}{2}} \\
	& = \bigg(\frac{g}{g+1}\bigg)
	\frac{\Gamma\big(\frac{b+d_S}{2}\big)\Gamma\big(\frac{b+n+d_S+1}{2}\big)}{\Gamma\big(\frac{b+d_S+1}{2}\big)\Gamma\big(\frac{b+n+d_S}{2}\big)}\big(1-\hat{\rho}^2_{xy\mid S}\big)^{-\frac{n}{2}}.
\end{align*}
\normalsize
By Proposition \ref{gammabound},
\small
\begin{equation*}
	\frac{1}{\sqrt{\frac{b+d_S}{2}}} < \frac{\Gamma\big(\frac{b+d_S}{2}\big)}{\Gamma\big(\frac{b+d_S+1}{2}\big)} < \frac{1}{\sqrt{\frac{b+d_S-1}{2}+\frac{1}{4}}}
\end{equation*}
\normalsize
and
\small
\begin{equation*}
		\sqrt{\frac{b+n+d_S-1}{2}+\frac{1}{4}} < \frac{\Gamma\big(\frac{b+n+d_S+1}{2}\big)}{\Gamma\big(\frac{b+n+d_S}{2}\big)} < \sqrt{\frac{b+n+d_S}{2}}.
\end{equation*}
\normalsize
\noindent Thus,
\small
\begin{align*}
	\bigg(\frac{g}{g+1}\bigg) \sqrt{\frac{b+n+d_S-\frac{1}{2}}{b+d_S}} \big(1-\hat{\rho}^2_{xy\mid S}\big)^{-\frac{n}{2}} &< {\normalfont\mbox{BF}}\big(G_{+(x,y)\not\in E_t};G_{-(x,y)\not\in E_t}\big) \\
	&< \bigg(\frac{g}{g+1}\bigg) \sqrt{\frac{b+n+d_S}{b+d_S-\frac{1}{2}}} \big(1-\hat{\rho}^2_{xy\mid S}\big)^{-\frac{n}{2}}.
\end{align*}
\normalsize

\noindent\textbf{CASE 2.2:} One of $P$, $Q$ is an empty set, i.e. clique $C_1$ contains only $\{x,S\}$ or clique $C_2$ contains only $\{y,S\}$ in $G_{-(x,y)\not\in E_t}$. In this case, adding an edge between node $x$ and $y$ will not create a new clique, but extending the original separator $S$ by node $x$ or $y$. Figure \ref{add2.2.1} shows when $P\neq\emptyset$ and $Q=\emptyset$, where $C_1=\{x,S,P\}$, $C_2=\{y,S\}$ in $G_{-(x,y)\not\in E_t}$. Figure \ref{add2.2.2} shows when $Q\neq\emptyset$ and $P=\emptyset$, where $C_1=\{x,S\}$, $C_2=\{y,S,Q\}$ in $G_{-(x,y)\not\in E_t}$.
\begin{figure}[H]
	\centering
	\resizebox{0.9\textwidth}{!}{
	\begin{tikzpicture}[thick]
	\node[draw, regular polygon, regular polygon sides=4] (S) {$S$};
	\node[draw, ellipse, left=0.7 of S, label={above:$C_1$}] (C1) {$x$, $S$, $P$};
	\node[draw, ellipse, right=0.7 of S, label={above:$C_2$}] (C2) {$y$, $S$};
	\edge[-] {C1} {S};
	\edge[-] {C2} {S};
	\node[left=0.7 of C1] (etc1) {...};
	\node[right=0.7 of C2] (etc2) {...};
	\edge[-] {C1} {etc1};
	\edge[-] {C2} {etc2};
	\node[below=0.1 of S, label={right: \scriptsize after adding an edge between $x$ and $y$}] (DA) {$\Big\Downarrow$};
	\node[draw, regular polygon, regular polygon sides=4, below=0.1 of DA] (SyS) {$x$, $S$};
	\node[draw, ellipse, left=0.7 of SyS] (CxyS) {$x$, $S$, $P$};
	\node[draw, ellipse, right=0.7 of SyS] (nC2) {$x$, $y$, $S$};
	\edge[-] {CxyS} {SyS};
	\edge[-] {nC2} {SyS};
	\node[left=0.7 of CxyS] (etc3) {...};
	\node[right=0.7 of nC2] (etc4) {...};
	\edge[-] {CxyS} {etc3};
	\edge[-] {nC2} {etc4};		
	\node[left=1.9 of C1] (G) {$G_{-(x,y)\not\in E_t}$};
	\node[below=0.1 of G] (Gb) {$\phantom{\Big\Downarrow}$};
	\node[below=0.42 of Gb] {$G_{+(x,y)\not\in E_t}$};
	\end{tikzpicture}
	}
	\caption{$P\neq\emptyset$ and $Q=\emptyset$, where $C_1=\{x,S,P\}$, $C_2=\{y,S\}$ in $G_{-(x,y)\not\in E_t}$.} \label{add2.2.1}
\end{figure}
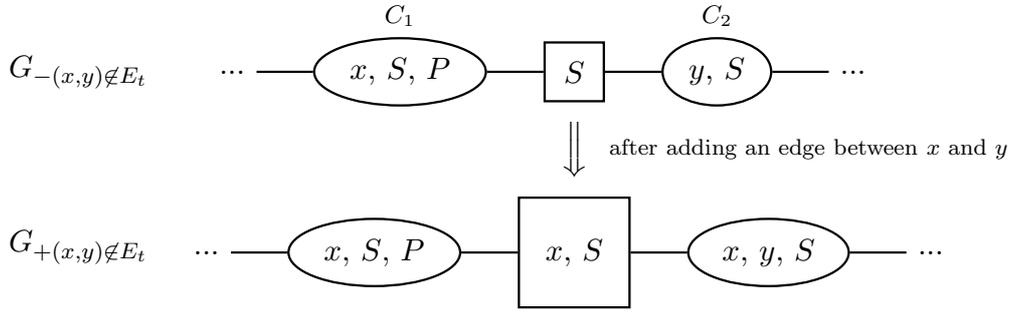
\begin{figure}[H]
	\centering
	\resizebox{0.9\textwidth}{!}{
	\begin{tikzpicture}[thick]
	\node[draw, regular polygon, regular polygon sides=4] (S) {$S$};
	\node[draw, ellipse, left=0.7 of S, label={above:$C_1$}] (C1) {$x$, $S$};
	\node[draw, ellipse, right=0.7 of S, label={above:$C_2$}] (C2) {$y$, $S$, $Q$};
	\edge[-] {C1} {S};
	\edge[-] {C2} {S};
	\node[left=0.7 of C1] (etc1) {...};
	\node[right=0.7 of C2] (etc2) {...};
	\edge[-] {C1} {etc1};
	\edge[-] {C2} {etc2};
	\node[below=0.1 of S, label={right: \scriptsize after adding an edge between $x$ and $y$}] (DA) {$\Big\Downarrow$};
	\node[draw, regular polygon, regular polygon sides=4, below=0.1 of DA] (SyS) {$y$, $S$};
	\node[draw, ellipse, left=0.7 of SyS] (CxyS) {$x$, $y$, $S$};
	\node[draw, ellipse, right=0.7 of SyS] (nC2) {$y$, $S$, $Q$};
	\edge[-] {CxyS} {SyS};
	\edge[-] {nC2} {SyS};
	\node[left=0.7 of CxyS] (etc3) {...};
	\node[right=0.7 of nC2] (etc4) {...};
	\edge[-] {CxyS} {etc3};
	\edge[-] {nC2} {etc4};		
	\node[left=1.9 of C1] (G) {$G_{-(x,y)\not\in E_t}$};
	\node[below=0.1 of G] (Gb) {$\phantom{\Big\Downarrow}$};
	\node[below=0.42 of Gb] {$G_{+(x,y)\not\in E_t}$};
	\end{tikzpicture}
	}
	\caption{$Q\neq\emptyset$ and $P=\emptyset$, where $C_1=\{x,S\}$, $C_2=\{y,S,Q\}$ in $G_{-(x,y)\not\in E_t}$.} \label{add2.2.2}
\end{figure}
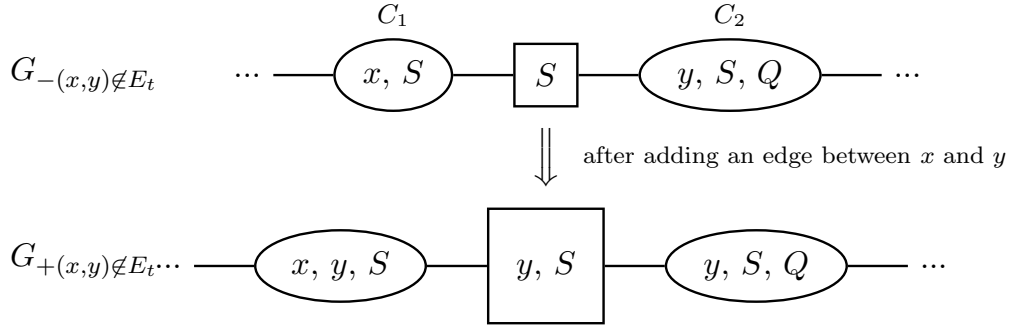
\small
\begin{equation*}
	\mbox{BF}\big(G_{+(x,y)\not\in E_t};G_{-(x,y)\not\in E_t}\big) = \frac{w(\{x,y,S\})\cdot w(S)}{w(\{x,S\})\cdot w(\{y,S\})}.
\end{equation*}
\normalsize
This is the same as \textbf{CASE 2.1}.\\

\noindent\textbf{CASE 2.3:} When $P$, $Q$ are both non-empty sets and $P\cap Q=\emptyset$, i.e. $C_1=\{x,S,P\}$, $C_2=\{y,S,Q\}$ in $G_{-(x,y)\not\in E_t}$. In this case, adding an edge between $x$ and $y$ will create a new clique $\{x,y,S\}$ and two new separators $\{x,S\}$ and $\{y,S\}$. Figure \ref{add2.3} illustrates this case.
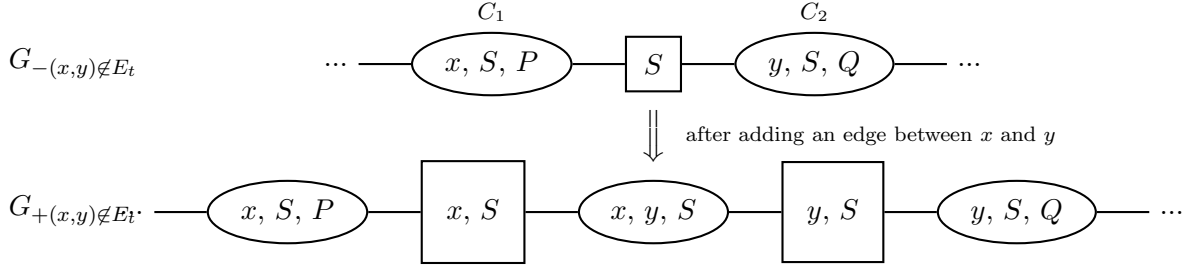
\begin{figure}[H]
	\centering
	\resizebox{1.05\textwidth}{!}{
	\begin{tikzpicture}[thick]
	\node[draw, regular polygon, regular polygon sides=4] (S) {$S$};
	\node[draw, ellipse, left=0.7 of S, label={above:$C_1$}] (C1) {$x$, $S$, $P$};
	\node[draw, ellipse, right=0.7 of S, label={above:$C_2$}] (C2) {$y$, $S$, $Q$};
	\edge[-] {C1} {S};
	\edge[-] {C2} {S};
	\node[left=0.7 of C1] (etc1) {...};
	\node[right=0.7 of C2] (etc2) {...};
	\edge[-] {C1} {etc1};
	\edge[-] {C2} {etc2};
	\node[below=0.1 of S, label={right: \scriptsize after adding an edge between $x$ and $y$}] (DA) {$\Big\Downarrow$};
	\node[draw, ellipse, below=0.1 of DA] (CxyS) {$x$, $y$, $S$};
	\node[draw, regular polygon, regular polygon sides=4, left=0.7 of CxyS] (SxS) {$x$, $S$};
	\node[draw, regular polygon, regular polygon sides=4, right=0.7 of CxyS] (SyS) {$y$, $S$};
	\edge[-] {SxS} {CxyS};
	\edge[-] {SyS} {CxyS};
	\node[draw, ellipse, left=0.7 of SxS] (nC1) {$x$, $S$, $P$};
	\node[draw, ellipse, right=0.7 of SyS] (nC2) {$y$, $S$, $Q$};
	\edge[-] {SxS} {nC1};
	\edge[-] {SyS} {nC2};
	\node[left=0.7 of nC1] (etc3) {...};
	\node[right=0.7 of nC2] (etc4) {...};
	\edge[-] {nC1} {etc3};
	\edge[-] {nC2} {etc4};	
	\node[left=3.5 of C1] (G) {$G_{-(x,y)\not\in E_t}$};
	\node[below=0.1 of G] (Gb) {$\phantom{\Big\Downarrow}$};
	\node[below=0.2 of Gb] {$G_{+(x,y)\not\in E_t}$};
	\end{tikzpicture}
	}
	\caption{$P,Q\neq\emptyset$ and $P\cap Q=\emptyset$, where $C_1=\{x,S,P\}$, $C_2=\{y,S,Q\}$ in $G_{-(x,y)\not\in E_t}$.} \label{add2.3}
\end{figure}
\small
\begin{equation*}
	\mbox{BF}\big(G_{+(x,y)\not\in E_t};G_{-(x,y)\not\in E_t}\big) = \frac{w(\{x,y,S\})\cdot w(S)}{w(\{x,S\})\cdot w(\{y,S\})}.
\end{equation*}
\normalsize
This is also the same as \textbf{CASE 2.1}.
\end{proof}

\section{Pairwise Bayes Factor Consistency and Posterior Ratio Consistency -- any graph $G_a$ versus the true graph $G_t$}
\begin{lemma} \label{chainrule} {\normalfont{(Decomposable graph chain rule \cite{lauritzen1996graphical}).}} % Lauritzen (1996) Lemma 2.21 p20
	Let $G=(V,E)$ be a decomposable graph and let $G'=(V,E')$ be a subgraph of $G$ that also is decomposable with $|E\backslash E'|=k$. Then there is an increasing sequence  $G'=G_0 \subset G_1 \cdots \subset G_{k-1} \subset G_k = G$ of decomposable graphs that differ by exactly one edge.
\end{lemma}

% Let $G_t$ be the true graph which is decomposable and denote $|E_t|$ as the number of edges in $G_t$. We use $G_a$ to denote any decomposable graph other than $G_t$. Let $|E_a|$ and $|E_a^1|$ be the number of edges and the number of true edges in $G_a$, respectively. $|E_a^1|\leq|E_t|$. Let $G_c=(V,E_c)$ be the complete graph, where $E_c=\{(i,j): 1\leq i<j \leq p \}$. By definition $G_c$ is decomposable.

Assume $G_t\not\subset G_a$, then $|E_t|>|E_a^1|$. By Lemma \ref{chainrule}, there exists a decreasing sequence of decomposable graphs from $G_c$ to $G_a$ that differ by exactly one edge, say $\big\{\overline{G}_i^{\,c\rightarrow a}\big\}_{i=0}^{|E_c|-|E_a|}$, where $G_c=\overline{G}_0^{\,c\rightarrow a}\supsetneq\overline{G}_1^{\,c\rightarrow a}\supsetneq\cdots\supsetneq\overline{G}_{|E_c|-|E_a|-1}^{\,c\rightarrow a}\supsetneq\overline{G}_{|E_c|-|E_a|}^{\,c\rightarrow a}=G_a$. There are $|E_c|-|E_a|$ steps for moving from $G_c$ to $G_a$. Let $\big\{\rho_{\overline{x}_i\overline{y}_i\mid\overline{S}_i}\big\}_{i=1}^{|E_c|-|E_a|}$ be the corresponding population partial correlation (or correlation, when $\overline{S}_i=\emptyset$) sequence and $\big\{\mbox{BF}(\overline{G}_i^{\,c\rightarrow a};\overline{G}_{i-1}^{\,c\rightarrow a})\big\}_{i=1}^{|E_c|-|E_a|}$ be the corresponding Bayes factor sequence for each step. By that, we mean in the $i$th step, edge $(\overline{x}_i,\overline{y}_i)$ is removed; $\rho_{\overline{x}_i\overline{y}_i\mid\overline{S}_i}$ and $\mbox{BF}(\overline{G}_i^{\,c\rightarrow a};\overline{G}_{i-1}^{\,c\rightarrow a})$ are the population partial correlation and the Bayes factor accordingly, $i=1,2,\ldots,|E_c|-|E_a|$. $\overline{S}_i$ is the specific separator corresponding to the $i$th step. Among them $|E_t|-|E_a^1|$ steps are removal of true edges that are deletion cases; $|E_c|-|E_a|-|E_t|+|E_a^1|$ steps are removal of false edges that can be seen as the reciprocal of addition cases.

\begin{lemma} \label{Gc2Ga} {\normalfont(Origin of the exponential rate in the deletion case).}
	Assume $G_t\not\subset G_a$. In $\big\{\rho_{\overline{x}_i\overline{y}_i\mid\overline{S}_i}\big\}_{i=1}^{|E_c|-|E_a|}$, among all population partial correlations that are corresponding to the removal of true edges, at least one is non-zero and it is not a population correlation {\normalfont($\overline{S}_i\neq\emptyset$)}.
\end{lemma}
\begin{proof}
	There are many sequences of $\{(\overline{x}_i,\overline{y}_i)\}_{i=1}^{|E_c|-|E_a|}$ (in different orders) that can achieve moving from $G_c$ to $G_a$ and still maintaining decomposability along the way. Let $(\overline{x}_*,\overline{y}_*)\in E_t\backslash E_a^1$. Thus $(\overline{x}_*,\overline{y}_*)\in\{(\overline{x}_i,\overline{y}_i)\}_{i=1}^{|E_c|-|E_a|}$. Choose $(\overline{x}_1,\overline{y}_1)=(\overline{x}_*,\overline{y}_*)$. This means the first step is the removal of a true edge in $E_t\backslash E_a^1$ from $G_c$. Let $\overline{S}_*$ be the corresponding separator. Thus we know $\overline{S}_*=V\backslash\{\overline{x}_*,\overline{y}_*\}\neq\emptyset$, since $(\overline{x}_*,\overline{y}_*)$ is removed from $G_c$. In fact, the removal of any edge from a complete graph still maintains decomposability, i.e. $\overline{G}_{1}^{\,c\rightarrow a}$ is a decomposable graph. Since $(\overline{x}_*,\overline{y}_*)\in E_t$, by the pairwise Markov property, $\rho_{\overline{x}_*,\overline{y}_*\mid V\backslash\{\overline{x}_*,\overline{y}_*\}}\neq 0$. And $\rho_L\leq\abs{\rho_{\overline{x}_*,\overline{y}_*\mid V\backslash\{\overline{x}_*,\overline{y}_*\}}}\leq\rho_U$. Therefore, we complete the proof of this lemma.
\end{proof}

\begin{lemma} \label{inherit} {\normalfont(The inheritance of separators).}
	Let $G=(V,E)$ and $G'=(V,E')$ be two undirected graphs (not necessary to be decomposable). Assume $E\subseteq E'$. If $S\subsetneq V$ separates node $x\in V$ from node $y\in V$ in $G'$, where $(x,y)\not\in E'$, then $S$ also separates them in $G$.
\end{lemma}
\begin{proof}
	Assume $S$ does not separate $x$ from $y$ in $G$. By the definition of separators, there exists a path from $x$ to $y$ in $G$, say $x=v_0, v_1, \ldots, v_{l-1}, v_l=y$ and $v_i\not\in S$, for all $i=0,1,\ldots,l$. Since $E\subseteq E'$, the path from $x$ to $y$, $\{v_i\}_{i=1}^{l-1}$, is still a path from $x$ to $y$ in $G'$. By the definition of separators again, we know that $S$ does not separate $x$ from $y$ in $G'$. But this contradicts with the assumption in the lemma. Therefore, $S$ separates $x$ from $y$ in $G$.
\end{proof}

Assume $G_t\subsetneq G_a$, thus $|E_t|=|E_a^1|$. By Lemma \ref{chainrule}, there exists an increasing sequence of decomposable graphs from $G_t$ to $G_a$ that differ by exactly one edge, say $\big\{\widetilde{G}_i^{\,t\rightarrow a}\big\}_{i=0}^{|E_a|-|E_t|}$, where $G_t=\widetilde{G}_0^{\,t\rightarrow a}\subsetneq\widetilde{G}_1^{\,t\rightarrow a}\subsetneq\ldots\subsetneq\widetilde{G}_{|E_a|-|E_t|-1}^{\,t\rightarrow a}\subsetneq\widetilde{G}_{|E_a|-|E_t|}^{\,t\rightarrow a}=G_a$. There are $|E_a|-|E_t|$ steps for moving from $G_t$ to $G_a$. All of them are addition of false edges that are addition cases. Let $\big\{\rho_{\widetilde{x}_i\widetilde{y}_i\mid\widetilde{S}_i}\big\}_{i=1}^{|E_a|-|E_t|}$ be the corresponding population partial correlation (or correlation, when $\widetilde{S}_i=\emptyset$) sequence and $\big\{\mbox{BF}(\widetilde{G}_i^{\,t\rightarrow a};\widetilde{G}_{i-1}^{\,t\rightarrow a})\big\}_{i=1}^{|E_a|-|E_t|}$ be the corresponding Bayes factor sequence for each step. By that, we mean in the $i$th step, edge $(\widetilde{x}_i,\widetilde{y}_i)\not\in E_t$ is added; $\rho_{\widetilde{x}_i,\widetilde{y}_i\mid\widetilde{S}_i}$ and $\mbox{BF}(\widetilde{G}_i^{\,t\rightarrow a};\widetilde{G}_{i-1}^{\,t\rightarrow a})$ are the population partial correlation and the Bayes factor accordingly, $i=1,2,\ldots,|E_a|-|E_t|$. $\widetilde{S}_i$ is the specific separator corresponding to the $i$th step.

\begin{lemma} \label{Gt2Gc} {\normalfont(Origin of the polynomial rate in the addition case).}
	Assume $G_t\subsetneq G_a$. For any edge sequence $\{(\widetilde{x}_i,\widetilde{y}_i)\}_{i=1}^{|E_a|-|E_t|}$ from $G_t$ to $G_a$ described above, all population partial correlations in $\big\{\rho_{\widetilde{x}_i\widetilde{y}_i\mid\widetilde{S}_i}\big\}_{i=1}^{|E_a|-|E_t|}$ are zero. {\normalfont(}or correlation, when $\widetilde{S}_i=\emptyset${\normalfont)}
\end{lemma}
\begin{proof}
	Assume in the $i$th step, we add edge $(\widetilde{x}_{i},\widetilde{y}_{i})\not\in E_t$ to graph $\widetilde{G}_{i-1}^{\,t\rightarrow a}$ and $\widetilde{S}_i$ is the corresponding separator, where $1\leq i\leq |E_a|-|E_t|$.
	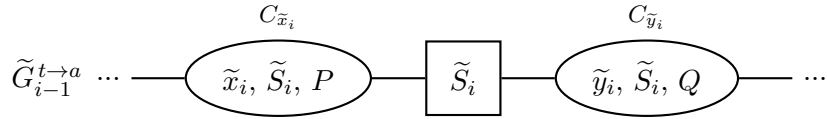
\begin{figure}[H]
		\centering
		\begin{tikzpicture}[thick]
			\node[draw, regular polygon, regular polygon sides=4] (S) {$\widetilde{S}_i$};
			\node[draw, ellipse, left=0.7 of S, label={above:$C_{\widetilde{x}_i}$}] (C1) {$\widetilde{x}_i$, $\widetilde{S}_i$, $P$};
			\node[draw, ellipse, right=0.7 of S, label={above:$C_{\widetilde{y}_i}$}] (C2) {$\widetilde{y}_i$, $\widetilde{S}_i$, $Q$};
			\edge[-] {C1} {S};
			\edge[-] {C2} {S};
			\node[left=0.7 of C1] (etc1) {...};
			\node[right=0.7 of C2] (etc2) {...};
			\edge[-] {C1} {etc1};
			\edge[-] {C2} {etc2};
			\node[left=1.2 of C1] {$\widetilde{G}_{i-1}^{\,t\rightarrow a}$};
		\end{tikzpicture}
		\caption{$\widetilde{G}_{i-1}^{\,t\rightarrow a}$ before adding edge $(\widetilde{x}_{i},\widetilde{y}_{i})\not\in E_t$ where $\widetilde{S}_i\neq\emptyset$.} \label{ithstep}
	\end{figure}
	First, when $\widetilde{S}_i\neq\emptyset$. Since edge $(\widetilde{x}_i,\widetilde{y}_i)\not\in E_t$ is added in the $i$th step, by Lemma \ref{validadd}, $C_{\widetilde{x}_i}$ and $C_{\widetilde{y}_i}$ are adjacent in some junction tree of $\widetilde{G}_{i-1}^{\,t\rightarrow a}$ where $C_{\widetilde{x}_i}$ and $C_{\widetilde{y}_i}$ are the cliques that contain $\widetilde{x}_i$ and $\widetilde{y}_i$, respectively. And $\widetilde{S}_i$ is the separator between them, i.e. $\widetilde{S}_i=C_{\widetilde{x}_i}\cap C_{\widetilde{y}_i}$. By the property of junction trees, we know $\widetilde{S}_i$ separates $\widetilde{x}_i$ from $\widetilde{y}_i$ in $\widetilde{G}_{i-1}^{\,t\rightarrow a}$. Since $\big\{\widetilde{G}_i^{\,t\rightarrow a}\big\}_{i=0}^{|E_a|-|E_t|}$ is an increasing sequence by edge, by Lemma \ref{inherit}, we know $\widetilde{S}_i$ also separates $\widetilde{x}_i$ from $\widetilde{y}_i$ in $\widetilde{G}_0^{\,t\rightarrow a}=G_t$. By the global Markov property, $\rho_{\widetilde{x}_i\widetilde{y}_i\mid\widetilde{S}_i}=0$.
	\begin{figure}[H]
		\centering
		\begin{tikzpicture}[thick]
			\node[draw, regular polygon, regular polygon sides=4] (S) {$\emptyset$};
			\node[draw, ellipse, left=0.7 of S, label={above:$C_{\widetilde{x}_i}$}] (C1) {$\widetilde{x}_i$, $P$};
			\node[draw, ellipse, right=0.7 of S, label={above:$C_{\widetilde{y}_i}$}] (C2) {$\widetilde{y}_i$, $Q$};
			\edge[-] {C1} {S};
			\edge[-] {C2} {S};
			\node[left=0.7 of C1] (etc1) {...};
			\node[right=0.7 of C2] (etc2) {...};
			\edge[-] {C1} {etc1};
			\edge[-] {C2} {etc2};
			\node[left=1.2 of C1] {$\widetilde{G}_{i-1}^{\,t\rightarrow a}$};
		\end{tikzpicture}
		\caption{$\widetilde{G}_{i-1}^{\,t\rightarrow a}$ before adding edge $(\widetilde{x}_{i},\widetilde{y}_{i})\not\in E_t$ where $\widetilde{S}_i=\emptyset$.} \label{ithstep}
	\end{figure}
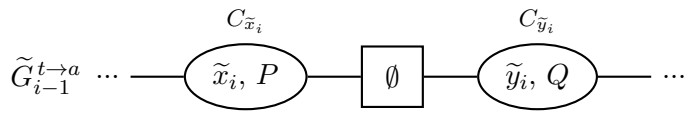
	Next, when $\widetilde{S}_i=\emptyset$, we show $\rho_{\widetilde{x}_i\widetilde{y}_i}=0$. By the property of junction trees, we know node $\widetilde{x}_i$ and $\widetilde{y}_i$ are disconnected. Furthermore, in the current graph $\widetilde{G}_{i-1}^{\,t\rightarrow a}$, nodes before clique $C_{\widetilde{x}_i}$ (including nodes in $C_{\widetilde{x}_i}$) and nodes after clique $C_{\widetilde{y}_i}$ (including nodes in $C_{\widetilde{y}_i}$) are disconnected. Since $G_t\subsetneq \widetilde{G}_{i-1}^{\,t\rightarrow a}$, then this is also true in $G_t$. Thus, nodes before clique $C_{\widetilde{x}_i}$ (including nodes in $C_{\widetilde{x}_i}$) and nodes after clique $C_{\widetilde{y}_i}$ (including nodes in $C_{\widetilde{y}_i}$) are disconnected in $G_t$. We can rearrange the precision matrix of $G_t$ into a block matrix such that the block which $\widetilde{x}_i$ is in and the block which $\widetilde{y}_i$ is in are independent. Therefore, node $\widetilde{x}_i$ and $\widetilde{y}_i$ are marginally independent in $G_t$, $\rho_{\widetilde{x}_i\widetilde{y}_i}=0$. Notice when $G_a=G_c$ this lemma still holds. 
\end{proof}

For the rest of proofs, when $G_t\not\subset G_a$, moving from $G_c$ to $G_a$ is restricted to the order of deleting edges in Lemma \ref{Gc2Ga} (deleting a true edge at the beginning); when $G_t\subsetneq G_a$, moving from $G_t$ to $G_a$ (or $G_c$) can be any order of adding edges (as long as decomposability is satisfied) according to Lemma \ref{Gt2Gc}. Following the notations in Lemma \ref{Gc2Ga} and \ref{Gt2Gc}, we have the decomposition of Bayes factor in favor of $G_a$ as follows. \\

\noindent When $G_t\not\subset G_a$, 
\small
\begin{align*}
	\mbox{BF}(G_a;G_t) &= \frac{f(\mathrm{Y}\mid G_a)}{f(\mathrm{Y}\mid G_t)}
	  = \frac{f(\mathrm{Y}\mid G_a)}{f(\mathrm{Y}\mid G_c)} \cdot \frac{f(\mathrm{Y}\mid G_c)}{f(\mathrm{Y}\mid G_t)} \\
	& = \frac{p(\mathrm{Y}\mid G_a)}{p(\mathrm{Y}\mid \overline{G}_{|E_c|-|E_a|-1}^{\,c\rightarrow a})}\frac{p(\mathrm{Y}\mid \overline{G}_{|E_c|-|E_a|-1}^{\,c\rightarrow a})}{p(\mathrm{Y}\mid \overline{G}_{|E_c|-|E_a|-2}^{\,c\rightarrow a})} \ldots \frac{p(\mathrm{Y}\mid \overline{G}_2^{\,c\rightarrow a})}{p(\mathrm{Y}\mid \overline{G}_1^{\,c\rightarrow a})} \frac{p(\mathrm{Y}\mid \overline{G}_1^{\,c\rightarrow a})}{p(\mathrm{Y}\mid G_c)} \\
	& \times \frac{p(\mathrm{Y}\mid G_c)}{p(\mathrm{Y}\mid \widetilde{G}_{|E_c|-|E_t|-1}^{\,t\rightarrow c})} \frac{p(\mathrm{Y}\mid \widetilde{G}_{|E_c|-|E_t|-1}^{\,t\rightarrow c})}{p(\mathrm{Y}\mid \widetilde{G}_{|E_c|-|E_t|-2}^{\,t\rightarrow c})} \ldots \frac{p(\mathrm{Y}\mid \widetilde{G}_2^{\,t\rightarrow c})}{p(\mathrm{Y}\mid \widetilde{G}_1^{\,t\rightarrow c})} \frac{p(\mathrm{Y}\mid \widetilde{G}_1^{\,t\rightarrow c})}{p(\mathrm{Y}\mid G_t)} \\
	& = \prod_{i=1}^{|E_c|-|E_a|} \mbox{BF}(\overline{G}_i^{\,c\rightarrow a};\overline{G}_{i-1}^{\,c\rightarrow a}) \cdot \prod_{i=1}^{|E_c|-|E_t|} \mbox{BF}(\widetilde{G}_i^{\,t\rightarrow c};\widetilde{G}_{i-1}^{\,t\rightarrow c}) \\
	& = \mbox{BF}_{c\rightarrow a} \cdot \mbox{BF}_{t\rightarrow c}.
\end{align*}
\normalsize
Therefore,
\small
\begin{align*}
	\mbox{PR}(G_a;G_t) & = \frac{p(G_a\mid\mathrm{Y})}{p(G_t\mid\mathrm{Y})} = \frac{f(\mathrm{Y}\mid G_a)\pi(G_a)}{f(\mathrm{Y}\mid G_t)\pi(G_t)} = \mbox{BF}(G_a;G_t)\frac{\pi(G_a)}{\pi(G_t)}\\
	& = \mbox{BF}_{c\rightarrow a} \cdot \mbox{BF}_{t\rightarrow c} \cdot \bigg(\frac{q}{1-q}\bigg)^{|E_a|-|E_t|}.
\end{align*}
\normalsize
$\mbox{BF}_{c\rightarrow a}$ contains $|E_c|-|E_a|$ terms, in which $|E_t|-|E_a^1|$ terms are deletion cases and $|E_c|-|E_a|-|E_t|+|E_a^1|$ terms are the reciprocal of addition cases. $\mbox{BF}_{t\rightarrow c}$ has $|E_c|-|E_t|$ terms that are all addition cases. \\

\noindent When $G_t\subsetneq G_a$,
\small
\begin{align*}
	\mbox{BF}(G_a;G_t) &= \prod_{i=1}^{|E_a|-|E_t|} \mbox{BF}(\widetilde{G}_i^{\,t\rightarrow a};\widetilde{G}_{i-1}^{\,t\rightarrow a}) = \mbox{BF}_{t\rightarrow a}, \\
	\mbox{PR}(G_a;G_t) &= \mbox{BF}_{t\rightarrow a} \cdot \bigg(\frac{q}{1-q}\bigg)^{|E_a|-|E_t|}.
\end{align*}
\normalsize

\subsection{Proof of Theorem \ref{bfupper0}}
First, for any $\tau^*>2$, let $\epsilon_{1,n}=\sqrt{\frac{\log(n-p)}{\tau^*(n-p)}}$. Then define
\small
\begin{equation*}
	R'_{ij\mid S}=\big\{|\hat{\rho}_{ij\mid S}-\rho_{ij\mid S}|<\epsilon_{1,n}\big\}.
\end{equation*}
\normalsize
Given any decomposable graph $G_a\neq G_t$, when $G_t\not\subset G_a$, by Lemma \ref{Gc2Ga}, we have the edge sequence $\{(\overline{x}_i,\overline{y}_i)\}_{i=1}^{|E_c|-|E_a|}$ for moving from $G_c$ to $G_a$ and let $(\overline{x}_1,\overline{y}_1) = (\overline{x}_*,\overline{y}_*)$ be the first in the sequence where a true edge is deleted from $G_c$. Let $\{(\widetilde{x}_i,\widetilde{y}_i)\}_{i=1}^{|E_c|-|E_t|}$ and $\{\widetilde{S}_i\}_{i=1}^{|E_c|-|E_t|}$ be the edge sequence and the corresponding separator sequence for moving from $G_t$ to $G_c$ according to Lemma \ref{Gt2Gc}. Let
\small
\begin{equation*}
	\Delta_{t\not\subset a,\epsilon_1}=\Big(R'_{\overline{x}_*\overline{y}_*\mid V\backslash\{\overline{x}_*,\overline{y}_*\}}\Big)\bigcap
	\Big(\cap_{i=1}^{|E_c|-|E_t|} R'_{\widetilde{x}_i\widetilde{y}_i\mid \widetilde{S}_i}\Big).
\end{equation*}
\normalsize
Since $\rho_U\neq 1$, by the proof of Lemma \ref{lemmafinite}, we have
\small
\begin{equation*}
	\bbP(\Delta_{t\not\subset a,\epsilon_1}) \geq \bbP(\Delta'_{\epsilon_1}) \geq 1 - \frac{42p^2}{(1-\rho_U)^2} (n-p)^{-\frac{1}{4\tau^*} }\Big\{\frac{1}{\tau^*}\log(n-p)\Big\}^{-\frac{1}{2}}.
\end{equation*}
\normalsize
When $G_t\subsetneq G_a$, let $\{(\widetilde{x}_i,\widetilde{y}_i)\}_{i=1}^{|E_a|-|E_t|}$ and $\{\widetilde{S}_i\}_{i=1}^{|E_a|-|E_t|}$ be the edge sequence and the corresponding separator sequence for moving from $G_t$ to $G_a$ according to Lemma \ref{Gt2Gc}. (Notice here we use the same edge and separator notations as in $G_t$ to $G_c$ for consistency reason and $G_t$ to $G_a$ can be seen as a part of $G_t$ to $G_c$.) Let
\small
\begin{equation*}
	\Delta_{t\subsetneq a,\epsilon_1}=	\bigcap_{i=1}^{|E_a|-|E_t|} R'_{\widetilde{x}_i\widetilde{y}_i\mid \widetilde{S}_i}.
\end{equation*}
\normalsize
Since $\rho_U\neq 1$, by the proof of Lemma \ref{lemmafinite}, we also have
\small
\begin{equation*}
	\bbP(\Delta_{t\subsetneq a,\epsilon_1}) \geq \bbP(\Delta'_{\epsilon_1}) \geq 1 - \frac{42p^2}{(1-\rho_U)^2} (n-p)^{-\frac{1}{4\tau^*} }\Big\{\frac{1}{\tau^*}\log(n-p)\Big\}^{-\frac{1}{2}}.
\end{equation*}
\normalsize
Thus, $\Delta_{a,\epsilon_1} = \Delta_{t\not\subset a,\epsilon_1}$ when $G_t\not\subset G_a$ and $\Delta_{a,\epsilon_1} = \Delta_{t\subsetneq a,\epsilon_1}$ when $G_t\subsetneq G_a$. For the following proof, we restrict it to the event $\Delta_{a,\epsilon_1}$. Next, we consider two scenarios for Bayes factor consistency, i.e. $G_t\not\subset G_a$ and $G_t\subsetneq G_a$. \\

First, when $G_t\not\subset G_a$ and $G_t\neq G_c$, we have $|E_t|>|E_a^1|$ and $|E_c|>|E_t|$. We begin by simplifying the upper bound of $\mbox{BF}_{t\rightarrow c}$. (for $G_t=G_c$, $\mbox{BF}_{t\rightarrow c}=1$) By Lemma \ref{lemmaadd} and \ref{Gt2Gc},
\small 
\begin{align*}
	\mbox{BF}_{t\rightarrow c} & = \prod_{i=1}^{|E_c|-|E_t|} \mbox{BF}(\widetilde{G}_i^{\,t\rightarrow c};\widetilde{G}_{i-1}^{\,t\rightarrow c}) \\
	& < \prod_{i=1}^{|E_c|-|E_t|} \Big(\frac{g}{g+1}\Big) \sqrt{\frac{b+n+d_{\widetilde{S}_i}}{b+d_{\widetilde{S}_i}-\frac{1}{2}}} (1-\hat{\rho}^2_{\widetilde{x}_i\widetilde{y}_i\mid\widetilde{S}_i})^{-\frac{n}{2}} \\
	% & < \Big( \frac{2}{n} \Big)^{\frac{|E_c|-|E_t|}{2}} \prod_{i=1}^{|E_c|-|E_t|} (1-\hat{\rho}^2_{\widetilde{x}_i\widetilde{y}_i\mid\widetilde{S}_i})^{-\frac{n}{2}}, \quad \text{ when }n>b+p,\\
	& < \Big( \frac{2}{n} \Big)^{\frac{|E_c|-|E_t|}{2}} \bigg\{1 - \frac{\log(n-p)}{\tau^*(n-p)} \bigg\}^{-(|E_c|-|E_t|)\frac{n}{2}}, \text{ when } n>b+p\\
	% & < \Big( \frac{2}{n} \Big)^{\frac{|E_c|-|E_t|}{2}} \bigg\{1 + \frac{\log n}{\tau^*(n-p)-\log n} \bigg\}^{(|E_c|-|E_t|)\frac{n}{2}} \\
	& < \Big( \frac{2}{n} \Big)^{\frac{|E_c|-|E_t|}{2}} \exp\bigg(\frac{n}{n-p-1/\tau^*\log n} \cdot \frac{|E_c|-|E_t|}{2\tau^*} \cdot \log n \bigg) \\
	& < \Big( \frac{2}{n} \Big)^{\frac{|E_c|-|E_t|}{2}} \exp\Big(\frac{|E_c|-|E_t|}{\tau^*}\cdot\log n \Big), \text{ when } n>4p \\
	& < \exp\bigg\{ p^2 - \Big(\frac{1}{2}-\frac{1}{\tau^*}\Big)(|E_c|-|E_t|)\log n \bigg\}.
\end{align*}
\normalsize
Next, we examine $\mbox{BF}_{c\rightarrow a}$. Based on Lemma \ref{Gc2Ga} and its proof, we divide it into two parts, i.e deletion cases and the reciprocal of addition cases. For deletion cases, we use $\{(\overline{x}^{\,d}_i,\overline{y}^{\,d}_i)\}_{i=1}^{|E_t|-|E_a^1|}$ to denote the sequence of true edges and $\{\overline{S}^{\,d}_i\}_{i=1}^{|E_t|-|E_a^1|}$ are the corresponding separator sequence. For addition cases, we use $\{(\overline{x}^{\,a}_i,\overline{y}^{\,a}_i)\}_{i=1}^{|E_c|-|E_a|-|E_t|+|E_a^1|}$ and $\{\overline{S}^{\,a}_i\}_{i=1}^{|E_c|-|E_a|-|E_t|+|E_a^1|}$. Since $p$ is finite, by the definition of $\rho_L$, then $\rho_L$ is a positive finite constant.
\small
\begin{align*}
	\mbox{BF}_{c\rightarrow a} & = \prod_{i=1}^{|E_c|-|E_a|} \mbox{BF}(\overline{G}_i^{\,c\rightarrow a};\overline{G}_{i-1}^{\,c\rightarrow a}) \\
	& < \prod_{i=1}^{|E_t|-|E_a^1|} \Big(1+\frac{1}{g}\Big) \sqrt{\frac{b+d_{\overline{S}^{\,d}_i}}{b+n+d_{\overline{S}^{\,d}_i}-\frac{1}{2}}}(1-\hat{\rho}^2_{\overline{x}^{\,d}_i\overline{y}^{\,d}_i\mid\overline{S}^{\,d}_i})^{\frac{n}{2}} \\
	& \phantom{111111111111} \times \prod_{i=1}^{|E_c|-|E_a|-|E_t|+|E_a^1|} \Big(1+\frac{1}{g}\Big) \sqrt{\frac{b+d_{\overline{S}^{\,a}_i}}{b+n+d_{\overline{S}^{\,a}_i}-\frac{1}{2}}}(1-\hat{\rho}^2_{\overline{x}^{\,a}_i\overline{y}^{\,a}_i\mid\overline{S}^{\,a}_i})^{\frac{n}{2}} \\
	& < \big\{2p(n+1)\big\}^{\frac{|E_c|-|E_a|}{2}} (1-\hat{\rho}^2_{\overline{x}_*\overline{y}_*\mid V\backslash\{\overline{x}_*,\overline{y}_*\}})^{\frac{n}{2}}, \text{ wlog assume } p>b \\
	& < \big\{2p(n+1)\big\}^{\frac{|E_c|-|E_a|}{2}} \Big\{1-\big(\epsilon_1-\abs{\rho_{\overline{x}_*\overline{y}_*\mid V\backslash\{\overline{x}_*,\overline{y}_*\}}}\big)^2\Big\}^{\frac{n}{2}} \\
	& < \big\{2p(n+1)\big\}^{\frac{|E_c|-|E_a|}{2}}  
	\exp\Big( -\frac{n\rho^2_L}{2} + n\epsilon_1 -\frac{n\epsilon_1^2}{2}\Big) \\
	& < \big\{2p(n+1)\big\}^{\frac{|E_c|-|E_a|}{2}} 
	\exp\Big\{ -\frac{n\rho^2_L}{2} + \sqrt{n\log n} -\frac{1}{2\tau^*}\log (n-p) \Big\}, \text{ when } n>2p \\
	& < \exp\Big\{ -\frac{n\rho^2_L}{2} + p^2\log n + \sqrt{n\log n} -\frac{1}{2\tau^*}\log (n-p) + 2p^2\log p \Big\}, \text{ when } n>1.
\end{align*}
\normalsize
Let $\delta(n) = p^2\log n + \sqrt{n\log n} + 3p^2\log p$ and $\delta(n)/n\rightarrow 0$ as $n\rightarrow\infty$. Hence,
\small
\begin{equation*}
	\mbox{BF}(G_a;G_t \mid G_t\not\subset G_a) = \mbox{BF}_{c\rightarrow a} \cdot \mbox{BF}_{t\rightarrow c} < \exp \Big\{-\frac{n\rho^2_L}{2} + \delta(n) \Big\}.
\end{equation*}
\normalsize
When $G_t\subsetneq G_a$, by Lemma \ref{lemmaadd} and \ref{Gt2Gc} we have
\small
\begin{align*}
	\mbox{BF}(G_a;G_t\mid G_t\subsetneq G_a) 
	& = \prod_{i=1}^{|E_a|-|E_t|} \mbox{BF}(\widetilde{G}_i^{t\rightarrow a};\widetilde{G}_{i-1}^{t\rightarrow a}) \\
	& < \exp\bigg\{ p^2 - \Big(\frac{1}{2}-\frac{1}{\tau^*}\Big)(|E_a|-|E_t|)\log n \bigg\}.
\end{align*}
\normalsize
% Therefore, we have shown the pairwise Bayes factor consistency for any $G_a\neq G_t$ when $p$ is finite.
%\end{proof}

\subsection{Proof of Theorem \ref{thpr}}
From $\gamma>1-4\alpha$, we have $\frac{1-\gamma}{2}<2\alpha$; from $\lambda<\frac{1}{2}-\alpha$, we have $\alpha+\lambda<\frac{1}{2}$; from $\lambda<\alpha$, we have $\alpha+\lambda<2\alpha$. For any $\beta^*$ that satisfies
\small
\begin{equation*}
	\max\Big\{\alpha+\lambda, \frac{1-\gamma}{2}\Big\}<\beta^*<\min\Big\{\frac{1}{2},2\alpha\Big\},
\end{equation*}
\normalsize
let $\epsilon_{2,n}=(n-p)^{-\beta^*}$. Then define
\small
\begin{equation*}
	R''_{ij|S}=\big\{|\hat{\rho}_{ij|S}-\rho_{ij|S}|<\epsilon_{2,n}\big\}.
\end{equation*}
\normalsize
Given any decomposable graph $G_a\neq G_t$, when $G_t\not\subset G_a$, by Lemma \ref{Gc2Ga}, we have the edge sequence $\{(\overline{x}_i,\overline{y}_i)\}_{i=1}^{|E_c|-|E_a|}$ for moving from $G_c$ to $G_a$ and let $(\overline{x}_1,\overline{y}_1) = (\overline{x}_*,\overline{y}_*)$ be the first in the sequence where a true edge is deleted from $G_c$. Let $\{(\widetilde{x}_i,\widetilde{y}_i)\}_{i=1}^{|E_c|-|E_t|}$ and $\{\widetilde{S}_i\}_{i=1}^{|E_c|-|E_t|}$ be the edge sequence and the corresponding separator sequence for moving from $G_t$ to $G_c$ according to Lemma \ref{Gt2Gc}. Let
\small
\begin{equation*}
	\Delta_{t\not\subset a,\epsilon_2}(n)=\Big(R''_{\overline{x}_*\overline{y}_*\mid V\backslash\{\overline{x}_*,\overline{y}_*\}}\Big)\bigcap
	\Big(\cap_{i=1}^{|E_c|-|E_t|} R''_{\widetilde{x}_i\widetilde{y}_i\mid \widetilde{S}_i}\Big).
\end{equation*}
\normalsize
Since $0<\beta^*<\frac{1}{2}$ and Assumption \ref{assump-upper}, by Lemma \ref{lemmainf1}, when $n\rightarrow\infty$,
\small
\begin{equation*}
	\bbP\big\{\Delta_{t\not\subset a,\epsilon_2}(n)\big\} \geq \bbP\big\{\Delta'_{\epsilon_2}(n)\big\} \geq 1 - \frac{42p^2}{(1-\rho_U)^2}(n-p)^{\beta^*-\frac{1}{2}}\exp\Big\{-\frac{1}{4}(n-p)^{1-2\beta}\Big\} \rightarrow 1.
\end{equation*}
\normalsize
When $G_t\subsetneq G_a$, let $\{(\widetilde{x}_i,\widetilde{y}_i)\}_{i=1}^{|E_a|-|E_t|}$ and $\{\widetilde{S}_i\}_{i=1}^{|E_a|-|E_t|}$ be the edge sequence and the corresponding separator sequence for moving from $G_t$ to $G_a$ according to Lemma \ref{Gt2Gc}. (Notice here we use the same edge and separator notations as in $G_t$ to $G_c$ for consistency reason and $G_t$ to $G_a$ can be seen as a part of $G_t$ to $G_c$.) Let
\small
\begin{equation*}
	\Delta_{t\subsetneq a,\epsilon_2}(n) =	\bigcap_{i=1}^{|E_a|-|E_t|} R''_{\widetilde{x}_i\widetilde{y}_i\mid \widetilde{S}_i}.
\end{equation*}
\normalsize
Since $0<\beta^*<\frac{1}{2}$ and Assumption \ref{assump-upper}, by Lemma \ref{lemmainf1}, when $n\rightarrow\infty$,
\small
\begin{equation*}
	\bbP\big\{\Delta_{t\subsetneq a,\epsilon_2}(n)\big\} \geq \bbP\big\{\Delta'_{\epsilon_2}(n)\big\} \geq 1 - \frac{42p^2}{(1-\rho_U)^2}(n-p)^{\beta^*-\frac{1}{2}}\exp\Big\{-\frac{1}{4}(n-p)^{1-2\beta}\Big\} \rightarrow 1.
\end{equation*}
\normalsize
Thus, $\Delta_{a,\epsilon_2}(n) = \Delta_{t\not\subset a,\epsilon_2}(n)$ when $G_t\not\subset G_a$ and $\Delta_{a,\epsilon_2}(n) = \Delta_{t\subseteq a,\epsilon_2}(n)$ when $G_t\subsetneq G_a$. For the following proof, we restrict it to the event $\Delta_{a,\epsilon_2}(n)$. Similar to the proof of Theorem \ref{bfupper0}, we consider two scenarios here for posterior ratio consistency, i.e. $G_t\not\subset G_a$ and $G_t\subsetneq G_a$. \\

First, when $G_t\not\subset G_a$ and $G_t\neq G_c$, we have $|E_t|>|E_a^1|$ and $|E_c|>|E_t|$. (for $G_t=G_c$, $\mathrm{BF}_{t\rightarrow c}=1$) By Lemma \ref{lemmaadd} and \ref{Gt2Gc},
\small
\begin{align*}
	\mathrm{BF}_{t\rightarrow c} & = \prod_{i=1}^{|E_c|-|E_t|} \mbox{BF}(\widetilde{G}_i^{\,t\rightarrow c};\widetilde{G}_{i-1}^{\,t\rightarrow c}) \\
	& < \Big( \frac{2}{n} \Big)^{\frac{|E_c|-|E_t|}{2}} \Big\{1 - (n-p)^{-2\beta^*} \Big\}^{-(|E_c|-|E_t|)\frac{n}{2}}, \text{ when } n>b+p\\
	% & < \Big( \frac{2}{n} \Big)^{\frac{|E_c|-|E_t|}{2}} \Big\{1+\frac{1}{(n-p)^{2\beta^*}-1} \Big\}^{(|E_c|-|E_t|)\frac{n}{2}} \\
	& < \Big( \frac{2}{n} \Big)^{\frac{|E_c|-|E_t|}{2}} \bigg\{1+\frac{2}{(n-p)^{2\beta^*}} \bigg\}^{(|E_c|-|E_t|)\frac{n}{2}}, \text{ when } n>\max\{2p,2^{1/(2\beta^*)+1}\} \\
	& < \exp\bigg\{\frac{np^2}{(n-p)^{2\beta^*}}-\frac{|E_c|-|E_t|}{4}\log n \bigg\}, \text{ when } n>4.
\end{align*}
\normalsize
Similar to the proof of Theorem \ref{bfupper0}, we have
\small
\begin{align*}
	\mbox{BF}_{c\rightarrow a} & = \prod_{i=1}^{|E_c|-|E_a|} \mbox{BF}(\overline{G}_i^{\,c\rightarrow a};\overline{G}_{i-1}^{\,c\rightarrow a}) \\
	& < \big\{2p(n+1)\big\}^{\frac{|E_c|-|E_a|}{2}} \big(1-\hat{\rho}^2_{\overline{x}_*\overline{y}_*\mid V\backslash\{\overline{x}_*,\overline{y}_*\}}\big)^{\frac{n}{2}}, \text{ when } p>b \\
	& < \big\{2p(n+1)\big\}^{\frac{|E_c|-|E_a|}{2}} \exp\Big( -\frac{n\rho^2_L}{2} + n\epsilon_2 -\frac{n\epsilon^2_2}{2} \Big) \\
	& < \big\{2p(n+1)\big\}^{\frac{|E_c|-|E_a|}{2}} \exp\bigg\{ -\frac{n\rho^2_L}{2} + \frac{n}{(n-p)^{\beta^*}}-\frac{1}{2}n^{1-2\beta^*} \bigg\} \\
	& < \exp\bigg\{ -\frac{n\rho^2_L}{2} + \frac{n}{(n-p)^{\beta^*}}-\frac{1}{2}n^{1-2\beta^*} + 3p^2\log n \bigg\}.
\end{align*}
\normalsize
When $n>3\exp\{(1-2\beta^*)^{-2}\}$, we have $n(n-p)^{-2\beta^*}>3\log n$. Hence,
\small
\begin{equation*}
	\mbox{BF}(G_a;G_t \mid G_t\not\subset G_a)
	% < \exp\Big\{ -\frac{n\rho^2_L}{2} + \frac{n}{(n-p)^{\beta^*}}-\frac{1}{2}n^{1-2\beta^*}+\frac{np^2}{(n-p)^{2\beta}}+3p^2\log n \Big\} \\
	< \exp\bigg\{ -\frac{n\rho^2_L}{2} + \frac{n}{(n-p)^{\beta^*}}-\frac{1}{2}n^{1-2\beta^*}+\frac{2np^2}{(n-p)^{2\beta^*}} \bigg\}.
\end{equation*}
\normalsize
Therefore, when $G_t\not\subset G_a$, for $n>(\log 2/C_q)^{1/\gamma}$,
\small
\begin{equation*}
	\mbox{PR}(G_a;G_t\mid G_t\not\subset G_a) < \exp\bigg\{-\frac{n\rho^2_L}{2} + \frac{n}{(n-p)^{\beta^*}}-\frac{1}{2}n^{1-2\beta^*}+\frac{2np^2}{(n-p)^{2\beta^*}}+\big(|E_a|-|E_t|\big)\log(2q)\bigg\}.
\end{equation*}
\normalsize
By the construction of $\beta^*$, we have
\small
\begin{equation*}
	1-2\lambda>1+2\alpha-2\beta^*>\max\{2\alpha,1-2\beta^*,1-\beta^*\},
\end{equation*}
\normalsize
and $1-2\lambda>\sigma+\gamma$. Therefore, $-n\rho^2_L/2$ is the leading term in the upper bound of $\mbox{PR}(G_a;G_t\mid G_t\not\subset G_a)$. Thus, $\mbox{PR}(G_a;G_t)\rightarrow 0$, as $n\rightarrow\infty$ when $G_t\not\subset G_a$. \\

\noindent When $G_t\subsetneq G_a$, by Lemma \ref{lemmaadd} and \ref{Gt2Gc} we have
\small
\begin{equation*}
	\mbox{BF}(G_a;G_t \mid G_t\subsetneq G_a) < \exp\bigg\{\frac{\big(|E_a|-|E_t|\big)n}{(n-p)^{2\beta^*}}\bigg\}.
\end{equation*}
\normalsize
So
\small
\begin{equation*}
	\mbox{PR}(G_a;G_t \mid G_t\subsetneq G_a) < \exp\bigg\{\frac{\big(|E_a|-|E_t|\big)n}{(n-p)^{2\beta^*}} + \big(|E_a|-|E_t|\big)\log(2q)\bigg\}.
\end{equation*}
\normalsize
Since $\beta^*>\frac{1-\gamma}{2}$, then $\big(|E_a|-|E_t|\big)\log(2q)$ is the leading term above and $|E_a|-|E_t|>0$. Therefore, $\mbox{PR}(G_a;G_t | G_t\subsetneq G_a)\rightarrow 0$, as $n\rightarrow\infty$.

\section{Proof of Theorem \ref{thstrong}}
From $\gamma>\alpha$, we have $\frac{1-\gamma}{2}<\frac{1-\alpha}{2}$; from $\gamma>1-4\alpha$, we have $\frac{1-\gamma}{2}<2\alpha$; from $\lambda<\frac{1}{2}(1-3\alpha)$, we have $\alpha+\lambda<\frac{1-\alpha}{2}$; from $\lambda<\alpha$, we have $\alpha+\lambda<2\alpha$. For any $\beta^\#$ satisfies
\small
\begin{equation*}
	\max\Big\{\alpha+\lambda, \frac{1-\gamma}{2}\Big\}<\beta^\#<\min\Big\{\frac{1-\alpha}{2},2\alpha\Big\},
\end{equation*}
\normalsize
let $\epsilon_{3,n}=(n-p)^{-\beta^\#}$. Then define
\small
\begin{equation*}
	R'''_{ij|S}=\big\{|\hat{\rho}_{ij|S}-\rho_{ij|S}|<\epsilon_{3,n}\big\}.
\end{equation*}
\normalsize
Denote
\small
\begin{equation*}
	\Delta''_{\epsilon_3}(n)=\Big\{\cap_{(i,j)\in E_t} R'''_{ij\mid V\backslash\{i,j\}}\Big\} \bigcap 
	\Big\{\cap_{(i,j)\not\in E_t}\big(\cap_{S\in\Pi_{ij}} R'''_{ij\mid S}\big)\Big\}. 
\end{equation*}
\normalsize
Since $0<\alpha<\frac{1}{3}$, thus $0<\beta^\#<\frac{1-\alpha}{2}<\frac{1}{2}$. By Assumption \ref{assump-upper} and Lemma \ref{lemmainf2}, 
\begin{equation*}
	\bbP\big\{\Delta''_{\epsilon_3}(n)\big\} \rightarrow 1, \text{ as } n\rightarrow\infty.
\end{equation*}
For any decomposable graph $G_a$, there exists a set $\Delta_{a,\epsilon_3}(n)$ defined in Theorem \ref{thpr}, such that $\Delta''_{\epsilon_3}(n)\subset \Delta_{a,\epsilon_3}(n)$. For the following proof, we restrict it to the event $\Delta''_{\epsilon_3}(n)$. Thus, the upper bound of Bayes factors derived under $\Delta''_{\epsilon_3}(n)$ is a uniform upper bound for all decomposable graphs that are not $G_t$. Following the proof of Theorem \ref{thpr}, when $G_t\not\subset G_a$,
\small
\begin{equation*}
	\mbox{PR}(G_a;G_t\mid G_t\not\subset G_a) < \exp\Big\{-\frac{n\rho^2_L}{2} + \frac{n}{(n-p)^{\beta^\#}}-\frac{1}{2}n^{1-2\beta^\#}+\frac{2np^2}{(n-p)^{2\beta^\#}}+\big(|E_a|-|E_t|\big)\log(2q)\Big\}.
\end{equation*}
\normalsize
By the construction of $\beta^\#$, we have
\begin{equation*}
	1-2\lambda>1+2\alpha-2\beta^\#>\max\{2\alpha,1-2\beta^\#,1-\beta^\#\},
\end{equation*}
and $1-2\lambda>\gamma+\sigma$. Therefore, $-n\rho^2_L/2$ is the leading term in the upper bound of $\mbox{PR}(G_a;G_t\mid G_t\not\subset G_a)$. For simplicity, only the leading term is used in the following calculation. \\

\noindent When $G_t\subsetneq G_a$,
\small
\begin{equation*}
	\mbox{PR}(G_a;G_t \mid G_t\subsetneq G_a) < \exp\bigg\{\frac{\big(|E_a|-|E_t|\big)n}{(n-p)^{2\beta^\#}} + \big(|E_a|-|E_t|\big)\log(2q)\bigg\}.
\end{equation*}
\normalsize
Since $\beta^\#>\frac{1-\gamma}{2}$, then $\big(|E_a|-|E_t|\big)\log(2q)$ is the leading term above and $|E_a|-|E_t|>0$. Thus, when $n$ is sufficiently large, for any decomposable graph $G_a\neq G_t$, we have 
\small
\begin{align*}
	\mbox{PR}(G_a;G_t \mid G_t\not\subset G_a) &< \exp\Big(-D_1n\rho^2_L \Big), \\
	\mbox{PR}(G_a;G_t \mid G_t\subsetneq G_a) &< \exp\Big\{-D_2n^{\gamma}\big(|E_a|-|E_t|\big) \Big\},
\end{align*}
\normalsize
where $D_1$ and $D_2$ are two positive finite constants.
\small
\begin{align*}
	\sum_{G_t\not\subset G_a} \mbox{PR}(G_a;G_t) & = \sum_{|E_a^1|=0}^{|E_t|-1} \binom{|E_t|}{|E_a^1|} \sum_{|E_a|-|E_a^1|=0}^{|E_c|-|E_t|} \binom{|E_c|-|E_t|}{|E_a|-|E_a^1|}\mbox{PR}(G_a;G_t \mid G_t\not\subseteq G_a) \\
	& < \exp(p^2\log 2)\exp(-D_1n\rho^2_L) \rightarrow 0, \text{ as } n\rightarrow\infty. 
\end{align*}
\normalsize
\small
\begin{align*}
	\sum_{G_t\subsetneq G_a} \mbox{PR}(G_a;G_t) & = \sum_{|E_a|=|E_t|+1}^{|E_c|} \binom{|E_c|-|E_t|}{|E_a|-|E_t|} \mbox{PR}(G_a;G_t \mid G_t\subsetneq G_a) \\
	& < \sum_{i=1}^{|E_c|-|E_t|} \binom{|E_c|-|E_t|}{i} \big(e^{-D_2n^\gamma}\big)^i \\
	& = (1 + e^{-D_2n^\gamma})^{|E_c|-|E_t|} - 1 \\
	& < \exp\big\{\big(|E_c|-|E_t|\big)e^{-D_2n^\gamma}\big\} - 1 \rightarrow 0, \text{ as } n\rightarrow\infty.
\end{align*}
\normalsize
\noindent (i) When $G_t=G_0$, where $G_0$ is the null graph with no edges.
\small
\begin{equation*}
	\sum_{G_a\neq G_0} \mbox{PR}(G_a;G_0) = \sum_{G_0\subsetneq G_a} \mbox{PR}(G_a;G_0) \rightarrow 0, \text{ as } n\rightarrow\infty;
\end{equation*}
\normalsize
\noindent (ii) When $G_t\neq G_0$ and $G_t\neq G_c$,
\small
\begin{equation*}
	\sum_{G_a\neq G_t} \mbox{PR}(G_a;G_t) = \sum_{G_t\not\subset G_a} \mbox{PR}(G_a;G_t) + \sum_{G_t\subsetneq G_a} \mbox{PR}(G_a;G_t) \rightarrow 0, \text{ as } n\rightarrow\infty;
\end{equation*}
\normalsize
\noindent (iii) When $G_t=G_c$,
\small
\begin{equation*}
	\sum_{G_a\neq G_c} \mbox{PR}(G_a;G_c) = \sum_{G_c\not\subset G_a} \mbox{PR}(G_a;G_c) \rightarrow 0, \text{ as } n\rightarrow\infty.
\end{equation*}
\normalsize
Therefore,
\small
\begin{equation*}
	\pi(G_t\mid \mathrm{Y}) = \frac{1}{1+\sum_{G_a\neq G_t} \mbox{PR}(G_a;G_t)}\rightarrow 1, \text{ as } n\rightarrow\infty.
\end{equation*}
\normalsize

\subsection{Proof of Corollary \ref{decomp-mode}}

	According to the proof of Theorem \ref{thstrong}, in the set $\Delta''_{\epsilon_3}(n)$, all Bayes factors in favor of $G_a$ converge to zero uniformly. Thus, we have

	\begin{equation*}
		\bbP \Big\{\max_{G_a\neq G_t} \pi(G_a\mid \mathrm{Y}) < \pi(G_t\mid \mathrm{Y}) \Big\}  \rightarrow 1, \quad \text{ as } n\rightarrow \infty.
	\end{equation*}

	Therefore,

	\begin{equation*}
		\bbP \big( \hat{G} = G_t \big) \rightarrow 1, \quad \text{ as } n\rightarrow \infty.
	\end{equation*}

\section{Equivalence Of Minimal Triangulations When $G_t$ Is Not Decomposable}
Let $G_m=(V,E_m)$ be any minimal triangulation of $G_t$, where $E_m=E_t\cup F$, $F\neq\emptyset$. In here $G_a$ denotes any decomposable graph other than minimal triangulations of $G_t$. Since $G_m$ is a minimal triangulation, then $E_a\neq E_t\cup F'$, where $F'\subseteq F$. Different from when $G_t$ is decomposable, there are three cases here: (1) $|E_a^1|<|E_m^1|=|E_t|$, thus $G_m\not\subset G_a$; (2) $|E_a^1|=|E_m^1|=|E_t|$ and $G_m\subsetneq G_a$; (3) $|E_a^1|=|E_m^1|=|E_t|$ and $G_m\not\subset G_a$. But in case (3) there exists at least one minimal triangulation of $G_t$ which is a subset of $G_a$. And in both (2) and (3), we have $|E_m|<|E_a|$.

For case (1), when $|E_a^1|<|E_m^1|=|E_t|$, i.e. one of the two cases where $G_m\not\subset G_a$, we inherit all notations from Lemma \ref{Gc2Ga}, $\{\overline{x}_i,\overline{y}_i\}_{i=1}^{|E_c|-|E_a|}$ is the edge sequence from $G_c$ to $G_a$ and $\{\rho_{\overline{x}_i\overline{y}_i\mid\overline{S}_i}\}_{i=1}^{|E_c|-|E_a|}$ is the corresponding population partial correlation sequence. And Lemma \ref{Gc2Ga} still holds here, i.e. at least one population partial correlation in $\{\rho_{\overline{x}_i\overline{y}_i\mid\overline{S}_i}\}_{i=1}^{|E_c|-|E_a|}$ corresponding to the removal of a true edge is non-zero and it is not a correlation. The proof carries out the same as in Lemma \ref{Gc2Ga}, just let the first step of moving from $G_c$ to $G_a$ be the deletion of one true edge which is missing in $G_a$. For case (3), where $|E_a^1|=|E_m^1|=|E_t|$ but $G_m\not\subset G_a$, when moving from $G_c$ to $G_a$, all steps are the reciprocal of addition cases. There is no deletion case here since $G_a$ has all the true edges in $G_t$.

For case (2), when $G_m\subsetneq G_a$ and $|E_a^1|=|E_m^1|=|E_t|$, we still use $\{(\widetilde{x}_i,\widetilde{y}_y)\}_{i=1}^{|E_a|-|E_m|}$ to denote the sequence of edges which are added in each steps from $G_m$ to $G_a$ and $\{\rho_{\widetilde{x}_i\widetilde{y}_i\mid\widetilde{S}_i}\}_{i=1}^{|E_a|-|E_m|}$ is the corresponding population partial correlation sequence. A similar version of Lemma \ref{Gt2Gc} still holds here.

\begin{lemma} \label{Gm2Ga}
	For any edge sequence $\{(\widetilde{x}_i,\widetilde{y}_i)\}_{i=1}^{|E_a|-|E_m|}$ from $G_m$ to $G_a$ describe above, all population partial correlations in $\{\rho_{\widetilde{x}_i\widetilde{y}_i\mid\widetilde{S}_i}\}_{i=1}^{|E_a|-|E_m|}$ are zero. {\normalfont(}or correlation, when $\widetilde{S}_i=\emptyset${\normalfont)}
\end{lemma}
\begin{proof}
	This proof follows similarly to the proof of Lemma \ref{Gt2Gc}. Assume in the $i$th step we add edge $(\widetilde{x}_i,\widetilde{y}_i)\not\in E_t$ to graph $\widetilde{G}^{\,m\rightarrow a}_{i-1}$ and $\widetilde{S}_i$ is the corresponding separator.

	When $\widetilde{S}_i\neq\emptyset$. Since adding edge $(\widetilde{x}_i,\widetilde{y}_i)\not\in E_t$ to graph $\widetilde{G}^{\,m\rightarrow a}_{i-1}$ maintains the decomposability of graph $\widetilde{G}^{\,m\rightarrow a}_i$. By Lemma \ref{validadd}, $\widetilde{x}_i$ and $\widetilde{y}_i$ are in two cliques which are adjacent in the current junction tree of $\widetilde{G}^{\,m\rightarrow a}_{i-1}$. Thus by the property of junction trees, we know $\widetilde{S}_i$ separates $\widetilde{x}_i$ from $\widetilde{y}_i$ in $\widetilde{G}^{\,m\rightarrow a}_{i-1}$. Since this is an increasing sequence in terms of edges from $G_m$ to $G_a$, thus $G_m\subsetneq\widetilde{G}^{\,m\rightarrow a}_{i-1}$. And due to the minimal triangulation, $G_t\subsetneq G_m\subsetneq\widetilde{G}^{\,m\rightarrow a}_{i-1}$.	By Lemma \ref{inherit}, $\widetilde{S}_i$ separates node $\widetilde{x}_i$ from $\widetilde{y}_i$ in $G_t$, $\rho_{\widetilde{x}_i\widetilde{y}_i\mid\widetilde{S}_i}=0$. 

	When $\widetilde{S}_i=\emptyset$, $\widetilde{x}_i$ and $\widetilde{y}_i$ are disconnected in the current graph $\widetilde{G}^{\,m\rightarrow a}_{i-1}$. Then they are also disconnected in $G_t$. Thus, they are marginally independent in $G_t$, $\rho_{\widetilde{x}_i\widetilde{y}_i}=0$.
\end{proof}

\begin{remark}
	For $|E_a^1|=|E_t|$ and $|E_a|-|E_a^1|=0,\ldots,|F|-1$, no decomposable $G_a$ exists; for $|E_a^1|=|E_t|$ and $|E_a|-|E_a^1|>|F|$, at least one decomposable $G_a$ exists; but for $|E_a^1|<|E_t|$ and $|E_a|-|E_a^1|\geq 0$, a decomposable $G_a$ may not exist. The Bayes factor $\mbox{BF}(G_a;G_m)$ under $|E_a^1|<|E_t|$ and $|E_a|-|E_a^1|\geq 0$ is only valid when a decomposable $G_a$ exists, otherwise it is defined to be zero.
\end{remark}

\subsection{Proof of Theorem \ref{finitetri}}
	{\bf Part 1.} For any given decomposable graph $G_a$ that is not a minimal triangulation of $G_t$, let
	\small
	\begin{equation*}
		\tau^*>\max\bigg\{2, \frac{2(|E_c|-|E_m|)}{|E_a|-|E_m|}\bigg\}.
	\end{equation*}
	\normalsize
	The construction of $\Delta_{a,\epsilon_1}$ is the same as in the proof of Theorem \ref{bfupper0}. After that, we restrict the following proof to the set $\Delta_{a,\epsilon_1}$. For case (1), when $|E_a^1|<|E_m^1|=|E_t|$, we have
	\small
	\begin{align*}
		\mbox{BF}_{m\rightarrow c} &< \exp\Big\{ p^2-\Big(\frac{1}{2}-\frac{1}{\tau^*}\Big)(|E_c|-|E_m|)\log n\Big\}\rightarrow 0, \\
		\mbox{BF}_{c\rightarrow a} &< \exp\Big\{ -\frac{n\rho^2_L}{2} + p^2\log n + \sqrt{n\log n} -\frac{1}{2\tau^*}\log(n-p) + 2p^2\log p \Big\}\rightarrow 0.
	\end{align*}
	\normalsize
	Hence,
	\small
	\begin{equation*}
		\mbox{BF}(G_a;G_m\mid G_m\not\subset G_a,|E_a^1|<|E_m^1|) = \mbox{BF}_{c\rightarrow a} \cdot \mbox{BF}_{m\rightarrow c} \rightarrow 0.
	\end{equation*}
	\normalsize
	For case (2), when $G_m\subsetneq G_a$, i.e. $|E_a^1|=|E_m^1|=|E_t|$ and $|E_a|>|E_m|$, we have
	\small
	\begin{equation*}
		\mbox{BF}(G_a;G_m\mid G_m\subsetneq G_a) < \exp\Big\{ p^2-\Big(\frac{1}{2}-\frac{1}{\tau^*}\Big)(|E_a|-|E_m|)\log n\Big\}\rightarrow 0.
	\end{equation*}
	\normalsize
	For case (3), when $|E_a^1|=|E_m^1|=|E_t|$ and $G_m\not\subset G_a$, also $|E_a|>|E_m|$, we have
	\small
	\begin{align*}
		\mbox{BF}_{m\rightarrow c} &< 2^{p^2} n^{-\frac{|E_c|-|E_a|}{2}} \exp\bigg[-\bigg\{\frac{|E_a|-|E_m|}{2(|E_c|-|E_m|)}-\frac{1}{\tau^*}\bigg\}(|E_c|-|E_m|)\log n\bigg], \\
		\mbox{BF}_{c\rightarrow a} &< (4p)^{p^2} n^{\frac{|E_c|-|E_a|}{2}}, \text{ when } n>1.
	\end{align*}
	\normalsize
	Hence,
	\small
	\begin{align*}
		& \mbox{BF}(G_a;G_m\mid G_m\not\subset G_a,|E_a^1|=|E_m^1|) \\
		& \phantom{1111111111} < (8p)^{p^2} \exp\bigg[-\bigg\{\frac{|E_a|-|E_m|}{2(|E_c|-|E_m|)}-\frac{1}{\tau^*}\bigg\}(|E_c|-|E_m|)\log n\bigg]\rightarrow 0.
	\end{align*}
	\normalsize
	Therefore, $\mbox{BF}(G_a;G_m)\rightarrow 0$, as $n\rightarrow\infty$. \\

	\noindent{\bf Part 2.} Let $\{\hat{\rho}_{m_1,i}\}_{i=1}^{|E_c|-|E_{m_1}|}$ and $\{\rho_{m_1,i}\}_{i=1}^{|E_c|-|E_{m_1}|}$ be the sample and population partial correlation sequence corresponding to each step from $G_{m_1}$ to $G_c$. By Lemma \ref{Gm2Ga}, $\rho_{m_1,i}=0$, $i=1,2,\ldots,|E_c|-|E_{m_1}|$. By Lemma \ref{lemmafinite1}, for any $0<\epsilon<1$, there exist $0<M_1(\epsilon)<1/4$ and $M_2(\epsilon)>3$ (the choice of $M_1$ and $M_2$ is the same as in the proof of Lemma \ref{lemmafinite1}), we have $\bbP(\Delta^0_{\epsilon})>1-\epsilon/2$, for $n>p+3$. Let
	\small
	\begin{equation*}
		R_{m_1,i} = \bigg\{ \frac{M_1}{n}< \hat{\rho}^2_{m_1,i} <\frac{M_2}{n-p} \bigg\},
	\end{equation*}
	\normalsize
	and denote
	\small
	\begin{equation*}
		\Delta_{m_1} = \bigcap_{i=1}^{|E_c|-|E_{m_1}|} R_{m_1,i}.
	\end{equation*}
	\normalsize
	Then 
	\small
	\begin{equation*}
		\bbP(\Delta_{m_1}) \geq \bbP(\Delta^0_\epsilon) \geq 1-\epsilon/2.
	\end{equation*}
	\normalsize
	By Lemma \ref{validadd}, when $n>b+p$, we have
	\small
	\begin{align*}
		\Big(\frac{1}{2n}\Big)^{\frac{|E_c|-|E_{m_1}|}{2}}\prod_{i=1}^{|E_c|-|E_{m_1}|} & (1-\hat{\rho}^2_{m_1,i})^{-\frac{n(|E_c|-|E_{m_1}|)}{2}} < \mbox{BF}(G_c;G_{m_1}) \\
		& < \Big(\frac{2}{n}\Big)^{\frac{|E_c|-|E_{m_1}|}{2}}\prod_{i=1}^{|E_c|-|E_{m_1}|}(1-\hat{\rho}^2_{m_1,i})^{-\frac{n(|E_c|-|E_{m_1}|)}{2}}.
	\end{align*}
	\normalsize
	Under the event $\Delta_{m_1}$, when $n>p+M_2$,
	\small
	\begin{equation*}
		\Big(\frac{e^{M_1}}{2n}\Big)^{\frac{|E_c|-|E_{m_1}|}{2}} < \mbox{BF}(G_c;G_{m_1})
		 < \Big(\frac{2e^{2M_2}}{n}\Big)^{\frac{|E_c|-|E_{m_1}|}{2}}.
	\end{equation*}
	\normalsize
	Thus we have
	\small
	\begin{equation*}
		\bbP\Bigg\{ \Big(\frac{e^{M_1}}{2n}\Big)^{\frac{|E_c|-|E_{m_1}|}{2}} < \mbox{BF}(G_c;G_{m_1})
		 < \Big(\frac{2e^{2M_2}}{n}\Big)^{\frac{|E_c|-|E_{m_1}|}{2}} \Bigg\} > 1-\frac{\epsilon}{2}.
	\end{equation*}
	\normalsize
	Similarly,
	\small
	\begin{equation*}
		\bbP\Bigg\{  \Big(\frac{2e^{2M_2}}{n}\Big)^{-\frac{|E_c|-|E_{m_1}|}{2}} < \mbox{BF}(G_{m_2};G_c)
		 < \Big(\frac{e^{M_1}}{2n}\Big)^{-\frac{|E_c|-|E_{m_1}|}{2}} \Bigg\} > 1-\frac{\epsilon}{2}.
	\end{equation*}
	\normalsize
	Therefore, let $A_1=\frac{1}{4}e^{-M_2}$ and $A_2=4e^{2M_2p^2}$,
	\small
	\begin{equation*}
		\bbP\big\{A_1<\mbox{BF}(G_{m_1};G_{m_2})<A_2\big\} > 1-\epsilon.
	\end{equation*}
	\normalsize

	\noindent{\bf Part 3.} Let $G_{m_1},G_{m_2},\ldots,G_{m_l}$ be all the minimal triangulations of $G_t$, where $l$ is a positive finite integer, since the graph dimension is finite. By Part 1, on the set $\Delta_{a,\epsilon_1}$,
	\small
	\begin{equation*}
		\mbox{BF}(G_{m_i};G_a) \rightarrow \infty, \quad i=1,2,\ldots,l,
	\end{equation*}
	\normalsize
	where $G_a\not\in\mathcal{M}_t$. Therefore,
	\small
	\begin{align*}
		\sum_{G_m\in \mathcal{M}_t} \pi(G_m\mid \mathrm{Y}) &= \frac{\sum_{i=1}^l p(\mathrm{Y}\mid G_{m_i})}{\sum_{i=1}^l p(\mathrm{Y}\mid G_{m_i}) + \sum_{G_a\not\in\mathcal{M}_t}p(\mathrm{Y}\mid G_a)} \\
		& = \frac{1}{1+ \sum_{G_a\not\in\mathcal{M}_t}\frac{p(\mathrm{Y}\mid G_a)}{\sum_{i=1}^l p(\mathrm{Y}\mid G_{m_i})} } \\
		& = \frac{1}{1+ \sum_{G_a\not\in\mathcal{M}_t}\frac{1}{\sum_{i=1}^l \mbox{BF}(G_{m_i};G_a)} } \rightarrow 1, \quad \text{ as } n\rightarrow\infty.
	\end{align*}
	\normalsize

\subsection{Proof of Theorem \ref{hightri}}
	{\bf Part 1.} From $\gamma>1-2\alpha$, we have $\frac{1-\gamma+2\alpha}{2}<2\alpha$; from $\lambda<\frac{1}{2}-\alpha$, we have $\alpha+\lambda<\frac{1}{2}$; from $\lambda<\alpha$, we have $\alpha+\lambda<2\alpha$; from $\gamma>2\alpha$, we have $\frac{1-\gamma+2\alpha}{2}<\frac{1}{2}$. Let $\beta^*$ satisfy
	\small
	\begin{equation*}
		\max\Big\{\alpha+\lambda, \frac{1-\gamma+2\alpha}{2}\Big\}<\beta^*<\min\Big\{\frac{1}{2},2\alpha\Big\},
	\end{equation*}
	\normalsize
	then follow the construction of $\Delta_{a,\epsilon_2}(n)$ in the proof of Theorem \ref{thpr} using $\beta^*$ specified above. After that, we restrict the following proof to the set $\Delta_{a,\epsilon_2}(n)$. For case (1), when $|E_a^1|<|E_m^1|=|E_t|$, by the construction of $\beta^*$, we have
	\small
	\begin{equation*}
		1-2\lambda>1+2\alpha-2\beta^*>\max\{2\alpha,1-2\beta^*,1-\beta^*\},
	\end{equation*}
	\normalsize
	and $1-2\lambda>\sigma+\gamma$. Thus,
	\small
	\begin{align*}
		& \mbox{PR}(G_a;G_m\mid G_m\not\subset G_a, |E_a^1|<|E_m^1|) \\
		& < \exp\Big\{ -\frac{n\rho^2_L}{2}+ \frac{n}{(n-p)^{\beta^*}} -\frac{1}{2}n^{1-2\beta^*} + \frac{2np^2}{(n-p)^{2\beta^*}}+\big(|E_a|-|E_m|\big)\log(2q) \Big\} \rightarrow 0.
	\end{align*}
	\normalsize
	For case (2), when $G_m\subsetneq G_a$, i.e. $|E_a^1|=|E_m^1|=|E_t|$ and $|E_a|>|E_m|$, since $\beta^*>\frac{1-\gamma}{2}$, we have
	\small
	\begin{equation*}
		\mbox{PR}(G_a;G_m\mid G_m\subsetneq G_a) < \exp\bigg\{ \frac{(|E_a|-|E_m|)n}{(n-p)^{2\beta^*}} + \big(|E_a|-|E_m|\big)\log(2q) \bigg\} \rightarrow 0.
	\end{equation*}
	\normalsize
	For case (3), when $|E_a^1|=|E_m^1|=|E_t|$ and $G_m\not\subset G_a$, also $|E_a|>|E_m|$, since $\beta^*>\frac{1-\gamma+2\alpha}{2}$, we have 
	\small
	\begin{align*}
		& \mbox{PR}(G_a;G_m\mid G_m\not\subset G_a, |E_a^1|=|E_m^1|) \\
		& \phantom{11111} < \{2p(n+1)\}^{\frac{|E_c|-|E_a|}{2}} \exp\bigg\{ \frac{(|E_c|-|E_m|)n}{(n-p)^{2\beta^*}} + \big(|E_a|-|E_m|\big)\log(2q) \bigg\} \rightarrow 0.
	\end{align*}
	\normalsize
	Therefore, $\mbox{PR}(G_a;G_m)\rightarrow 0$, as $n\rightarrow\infty$. \\

	\noindent{\bf Part 2.} Since the number of fill-in edges is finite, then the number of cycles length greater than 3 without a chord in $G_t$ is finite and the length of the longest cycle without a chord is also finite. Thus instead of adding one chord for each of those cycles that are length greater than 3 in $G_t$, we can complete the subgraphs induced by those cycles with finite number of edges. Let $G_{m_c}$ be the graph after completing all subgraphs induced by those cycles. Then $G_{m_c}$ is decomposable and $|E_{m_c}|-|E_t|$ is finite. We also know $G_{m_1},G_{m_2}\subsetneq G_{m_c}$. Let $\delta_c=|E_{m_c}|-|E_{m_1}|=|E_{m_c}|-|E_{m_2}|$. 

	Let $\{\hat{\rho}_{m_1,i}\}_{i=1}^{|E_{m_c}|-|E_{m_1}|}$ and $\{\rho_{m_1,i}\}_{i=1}^{|E_{m_c}|-|E_{m_1}|}$ be the sample and population partial correlation sequence corresponding to each step from $G_{m_1}$ to $G_{m_c}$. By Lemma \ref{Gm2Ga}, $\rho_{m_1,i}=0$, $i=1,2,\ldots,|E_{m_c}|-|E_{m_1}|$. By Corollary \ref{lemmafinite2}, for any $0<\epsilon<1$, there exist $0<M_1(\epsilon)<1/4$ and $M_2(\epsilon)>3$ (the choice of $M_1$ and $M_2$ is the same as in the proof of Corollary \ref{lemmafinite2}), we have $P(\Delta^{0+}_{\epsilon})>1-\epsilon/2$, for $n>p+3$. Let
	\small
	\begin{equation*}
		R'_{m_1,i} = \bigg\{ \frac{M_1}{n}< \hat{\rho}^2_{m_1,i} <\frac{M_2}{n-p} \bigg\},
	\end{equation*}
	\normalsize
	and denote
	\small
	\begin{equation*}
		\Delta'_{m_1} = \bigcap_{i=1}^{|E_{m_c}|-|E_{m_1}|} R'_{m_1,i}.
	\end{equation*}
	\normalsize
	Then 
	\small
	\begin{equation*}
		\bbP(\Delta'_{m_1}) \geq \bbP(\Delta^{0+}_\epsilon) \geq 1-\epsilon/2.
	\end{equation*}
	\normalsize
	By Lemma \ref{validadd}, when $n>b+p$, we have
	\small
	\begin{equation*}
		\Big(\frac{1}{2n}\Big)^{\frac{\delta_c}{2}}\prod_{i=1}^{\delta_c} (1-\hat{\rho}^2_{m_1,i})^{-\frac{n\delta_c}{2}} < \mbox{BF}(G_{m_c};G_{m_1})
		< \Big(\frac{2}{n}\Big)^{\frac{\delta_c}{2}}\prod_{i=1}^{\delta_c}(1-\hat{\rho}^2_{m_1,i})^{-\frac{n\delta_c}{2}}.
	\end{equation*}
	\normalsize
	Under the event $\Delta'_{m_1}$, when $n>p+M_2$,
	\small
	\begin{equation*}
		\Big(\frac{e^{M_1}}{2n}\Big)^{\frac{\delta_c}{2}} < \mbox{BF}(G_{m_c};G_{m_1})
		 < \Big(\frac{2e^{2M_2}}{n}\Big)^{\frac{\delta_c}{2}}.
	\end{equation*}
	\normalsize
	Thus we have
	\small
	\begin{equation*}
		\bbP\Bigg\{ \Big(\frac{e^{M_1}}{2n}\Big)^{\frac{\delta_c}{2}} < \mbox{BF}(G_{m_c};G_{m_1})
		 < \Big(\frac{2e^{2M_2}}{n}\Big)^{\frac{\delta_c}{2}} \Bigg\} > 1-\frac{\epsilon}{2}.
	\end{equation*}
	\normalsize
	Similarly,
	\small
	\begin{equation*}
		\bbP\Bigg\{  \Big(\frac{2e^{2M_2}}{n}\Big)^{-\frac{\delta_c}{2}} < \mbox{BF}(G_{m_2};G_{m_c})
		 < \Big(\frac{e^{M_1}}{2n}\Big)^{-\frac{\delta_c}{2}} \Bigg\} > 1-\frac{\epsilon}{2}.
	\end{equation*}
	\normalsize
	Therefore, let $A_1=\frac{1}{4}e^{-M_2\delta_c}$ and $A_2=4e^{M_2\delta_c}$,
	\small
	\begin{equation*}
		\bbP\big\{A_1<\mbox{BF}(G_{m_1};G_{m_2})<A_2\big\} > 1-\epsilon.
	\end{equation*}
	\normalsize

	\noindent{\bf Part 3.} From $\gamma>1-2\alpha$, we have $\frac{1-\gamma+2\alpha}{2}<2\alpha$; from $\lambda<\frac{1-3\alpha}{2}$, we have $\alpha+\lambda<\frac{1-\alpha}{2}$; from $\lambda<\alpha$, we have $\alpha+\lambda<2\alpha$; from $\gamma>3\alpha$, we have $\frac{1-\gamma+2\alpha}{2}<\frac{1-\alpha}{2}$. Let $\beta^*$ satisfy
	\small
	\begin{equation*}
		\max\Big\{\alpha+\lambda, \frac{1-\gamma+2\alpha}{2}\Big\}<\beta^*<\min\Big\{\frac{1-\alpha}{2},2\alpha\Big\},
	\end{equation*}
	\normalsize
	then follow the construction of $\Delta''_{\epsilon_3}(n)$ in the proof of Theorem \ref{thstrong} using $\beta^*$ specified above. After that, we restrict the following proof to the set $\Delta''_{\epsilon_3}(n)$. Let $G_{m_1},G_{m_2},\ldots,G_{m_h}$ be all the minimal triangulations of $G_t$, where $h$ is a positive integer that depends on $n$. By Part 1, we have
	\small
	\begin{align*}
		\mbox{PR}(G_a;G_m\mid G_m\not\subset G_a, |E_a^1|<|E_m^1|) &< \exp\Big( -D_1n\rho^2_L \Big), \\
		\mbox{PR}(G_a;G_m\mid G_m\subsetneq G_a) &< \exp\Big\{ -D_2n^\gamma\big(|E_a|-|E_m|\big)\Big\}, \\
		\mbox{PR}(G_a;G_m\mid G_m\not\subset G_a, |E_a^1|=|E_m^1|) &< \exp\Big\{ -D_3n^\gamma\big(|E_a|-|E_m|\big)\Big\},
	\end{align*}
	\normalsize
	where $D_1$, $D_2$, $D_3$ are three positive finite constants. And
	\small
	\begin{align*}
		\sum_{\substack{G_a\not\in\mathcal{M}_t, \\ G_{m_1}\not\subset G_a, \\ |E_a^1|<|E_{m_1}^1|}} \mbox{PR}(G_a;G_{m_1}) &< \exp(p^2)\exp\big( -D_1n\rho^2_L \big)\rightarrow 0, \\
		\sum_{\substack{G_a\not\in\mathcal{M}_t, \\ G_{m_1} \subsetneq G_a}} \mbox{PR}(G_a;G_{m_1}) &< \sum_{i=1}^{|E_c|-|E_{m_1}|}\binom{|E_c|-|E_{m_1}|}{i}\big(e^{-D_2n^\gamma}\big)^i \\
		& < \exp\big\{ (|E_c|-|E_{m_1}|)e^{-D_2n^\gamma} \big\} -1 \rightarrow 0, \\
		&\phantom{111} \\
		\sum_{\substack{G_a\not\in\mathcal{M}_t, \\ G_{m_1}\not\subset G_a, \\ |E_a^1|=|E_{m_1}^1|}} \mbox{PR}(G_a;G_{m_1}) &< \sum_{i=1}^{|E_c|-|E_{m_1}|}\binom{|E_c|-|E_{m_1}^1|}{|E_{m_1}|-|E_{m_1}^1|+i}\big(e^{-D_3n^\gamma}\big)^i \\
		& < \exp(p^2) \exp\big(-D_3n^\gamma\big) \rightarrow 0.
		% & < \sum_{j=|E_{m_1}|-|E_{m_1}^1|+1}^{|E_c|-|E_{m_1}^1|}\binom{|E_c|-|E_{m_1}^1|}{j}\big(e^{-D_3n^\gamma}\big)^{j-(|E_{m_1}|-|E_{m_1}^1|)} \\
		% & < exp\big\{ (|E_c|-|E_{m_1}^1|)e^{-D_3n^\gamma} \big\} -1 \rightarrow 0.
	\end{align*}
	\normalsize
	Thus
	\small
	\begin{align*}
		&\sum_{G_a\not\in\mathcal{M}_t}\frac{1}{\sum_{i=1}^h \mbox{PR}(G_{m_i};G_a)} < \sum_{G_a\not\in\mathcal{M}_t}\frac{1}{\mbox{PR}(G_{m_1};G_a)} \\
		& = \sum_{\substack{G_a\not\in\mathcal{M}_t, \\ G_{m_1}\subsetneq G_a}} \mbox{PR}(G_a;G_{m_1}) + \sum_{\substack{G_a\not\in\mathcal{M}_t, \\ G_{m_1}\not\subset G_a}} \mbox{PR}(G_a;G_{m_1}) + \sum_{\substack{G_a\not\in\mathcal{M}_t, \\ G_{m_1}\not\subset G_a, \\ |E_a^1|=|E_{m_1}^1|}} \mbox{PR}(G_a;G_{m_1}) \rightarrow 0.
	\end{align*}
	\normalsize
	Therefore,
	\small
	\begin{equation*}
		\sum_{G_m\in \mathcal{M}_t} \pi(G_m\mid \mathrm{Y}) 
		= \frac{1}{1+ \sum_{G_a\not\in\mathcal{M}_t}\frac{1}{\sum_{i=1}^h \mbox{PR}(G_{m_i};G_a)} }\rightarrow 1, \quad \text{ as } n\rightarrow\infty.
	\end{equation*}
	\normalsize

\subsection{Proof of Corollary \ref{nondecomp-mode}}
	Under the event $\Delta''_{\epsilon_3}(n)$ in the proof of Theorem \ref{hightri}, given any $G_m\in\mathcal{M}_t$, all Bayes factors in favor of $G_a$ converge to zero uniformly. Thus, we have
	\begin{equation*}
		\bbP \Big\{\max_{G_a\not\in\mathcal{M}_t} \pi(G_a\mid \mathrm{Y}) < \min_{G_m\in\mathcal{M}_t} \pi(G_m\mid \mathrm{Y}) \Big\}  \rightarrow 1, \quad \text{ as } n\rightarrow \infty.
	\end{equation*}
	Therefore,
	\small
	\begin{equation*}
		\bbP \big( \hat{G}\in\mathcal{M}_t \big) \rightarrow 1, \quad \text{ as } n\rightarrow \infty.
	\end{equation*}

\bibliographystyle{biometrika}
\bibliography{GraphConsist}

\begin{thebibliography}{46}
\expandafter\ifx\csname natexlab\endcsname\relax\def\natexlab#1{#1}\fi

\bibitem[{Anderson(1984)}]{anderson1984introduction}
\textsc{Anderson, T.~W.} (1984).
\newblock \textit{An introduction to multivariate statistical analysis}.
\newblock Wiley.

\bibitem[{Atay-Kayis \& Massam(2005)}]{atay2005monte}
\textsc{Atay-Kayis, A.} \& \textsc{Massam, H.} (2005).
\newblock A monte carlo method for computing the marginal likelihood in
  nondecomposable gaussian graphical models.
\newblock \textit{Biometrika} \textbf{92}, 317--335.

\bibitem[{Banerjee \& Ghosal(2015)}]{banerjee2015bayesian}
\textsc{Banerjee, S.} \& \textsc{Ghosal, S.} (2015).
\newblock Bayesian structure learning in graphical models.
\newblock \textit{Journal of Multivariate Analysis} \textbf{136}, 147--162.

\bibitem[{Banerjee et~al.(2014)Banerjee, Ghosal et~al.}]{banerjee2014posterior}
\textsc{Banerjee, S.}, \textsc{Ghosal, S.} et~al. (2014).
\newblock Posterior convergence rates for estimating large precision matrices
  using graphical models.
\newblock \textit{Electronic Journal of Statistics} \textbf{8}, 2111--2137.

\bibitem[{Ben-David et~al.(2011)Ben-David, Li, Massam \&
  Rajaratnam}]{ben2011high}
\textsc{Ben-David, E.}, \textsc{Li, T.}, \textsc{Massam, H.} \&
  \textsc{Rajaratnam, B.} (2011).
\newblock High dimensional bayesian inference for gaussian directed acyclic
  graph models.
\newblock \textit{arXiv preprint arXiv:1109.4371} .

\bibitem[{Bickel \& Levina(2008)}]{bickel2008regularized}
\textsc{Bickel, P.} \& \textsc{Levina, E.} (2008).
\newblock Regularized estimation of large covariance matrices.
\newblock \textit{The Annals of Statistics} \textbf{36}, 199--227.

\bibitem[{Cai \& Liu(2011)}]{cai2011adaptive}
\textsc{Cai, T.} \& \textsc{Liu, W.} (2011).
\newblock Adaptive thresholding for sparse covariance matrix estimation.
\newblock \textit{Journal of the American Statistical Association}
  \textbf{106}, 672--684.

\bibitem[{Cao et~al.(2016)Cao, Khare \& Ghosh}]{cao2016posterior}
\textsc{Cao, X.}, \textsc{Khare, K.} \& \textsc{Ghosh, M.} (2016).
\newblock Posterior graph selection and estimation consistency for
  high-dimensional bayesian dag models.
\newblock \textit{arXiv preprint arXiv:1611.01205} .

\bibitem[{Carvalho et~al.(2007)Carvalho, Massam \&
  West}]{carvalho2007simulation}
\textsc{Carvalho, C.~M.}, \textsc{Massam, H.} \& \textsc{West, M.} (2007).
\newblock Simulation of hyper-inverse wishart distributions in graphical
  models.
\newblock \textit{Biometrika} \textbf{94}, 647--659.

\bibitem[{Carvalho \& Scott(2009)}]{carvalho2009objective}
\textsc{Carvalho, C.~M.} \& \textsc{Scott, J.~G.} (2009).
\newblock Objective bayesian model selection in gaussian graphical models.
\newblock \textit{Biometrika} \textbf{96}, 497--512.

\bibitem[{Dawid \& Lauritzen(1993)}]{dawid1993hyper}
\textsc{Dawid, A.~P.} \& \textsc{Lauritzen, S.~L.} (1993).
\newblock Hyper markov laws in the statistical analysis of decomposable
  graphical models.
\newblock \textit{The Annals of Statistics} , 1272--1317.

\bibitem[{Dellaportas et~al.(2003)Dellaportas, Giudici \&
  Roberts}]{dellaportas2003bayesian}
\textsc{Dellaportas, P.}, \textsc{Giudici, P.} \& \textsc{Roberts, G.} (2003).
\newblock Bayesian inference for nondecomposable graphical gaussian models.
\newblock \textit{Sankhy{\=a}: The Indian Journal of Statistics} , 43--55.

\bibitem[{Diaconis \& Ylvisaker(1979)}]{diaconis1979conjugate}
\textsc{Diaconis, P.} \& \textsc{Ylvisaker, D.} (1979).
\newblock Conjugate priors for exponential families.
\newblock \textit{The Annals of statistics} , 269--281.

\bibitem[{Dobra et~al.(2004)Dobra, Hans, Jones, Nevins, Yao \&
  West}]{dobra2004sparse}
\textsc{Dobra, A.}, \textsc{Hans, C.}, \textsc{Jones, B.}, \textsc{Nevins,
  J.~R.}, \textsc{Yao, G.} \& \textsc{West, M.} (2004).
\newblock Sparse graphical models for exploring gene expression data.
\newblock \textit{Journal of Multivariate Analysis} \textbf{90}, 196--212.

\bibitem[{Donnet \& Marin(2012)}]{donnet2012empirical}
\textsc{Donnet, S.} \& \textsc{Marin, J.-M.} (2012).
\newblock An empirical bayes procedure for the selection of gaussian graphical
  models.
\newblock \textit{Statistics and Computing} \textbf{22}, 1113--1123.

\bibitem[{Drton et~al.(2007)Drton, Perlman et~al.}]{drton2007multiple}
\textsc{Drton, M.}, \textsc{Perlman, M.~D.} et~al. (2007).
\newblock Multiple testing and error control in gaussian graphical model
  selection.
\newblock \textit{Statistical Science} \textbf{22}, 430--449.

\bibitem[{El~Karoui(2008)}]{el2008operator}
\textsc{El~Karoui, N.} (2008).
\newblock Operator norm consistent estimation of large-dimensional sparse
  covariance matrices.
\newblock \textit{The Annals of Statistics} \textbf{36}, 2717--2756.

\bibitem[{Fitch et~al.(2014)Fitch, Jones, Massam et~al.}]{fitch2014performance}
\textsc{Fitch, A.~M.}, \textsc{Jones, M.~B.}, \textsc{Massam, H.} et~al.
  (2014).
\newblock The performance of covariance selection methods that consider
  decomposable models only.
\newblock \textit{Bayesian Analysis} \textbf{9}, 659--684.

\bibitem[{Frydenberg \& STEFFEN(1989)}]{frydenberg1989decomposition}
\textsc{Frydenberg, M.} \& \textsc{STEFFEN, L.~L.} (1989).
\newblock Decomposition of maximum likelihood in mixed graphical interaction
  models.
\newblock \textit{Biometrika} \textbf{76}, 539--555.

\bibitem[{Giudici(1996)}]{giudici1996learning}
\textsc{Giudici, P.} (1996).
\newblock Learning in graphical gaussian models.
\newblock \textit{Bayesian Statistics} \textbf{5}, 621--628.

\bibitem[{Giudici \& Green(1999)}]{giudici1999decomposable}
\textsc{Giudici, P.} \& \textsc{Green, P.} (1999).
\newblock Decomposable graphical gaussian model determination.
\newblock \textit{Biometrika} \textbf{86}, 785--801.

\bibitem[{Green \& Thomas(2013)}]{green2013sampling}
\textsc{Green, P.~J.} \& \textsc{Thomas, A.} (2013).
\newblock Sampling decomposable graphs using a markov chain on junction trees.
\newblock \textit{Biometrika} \textbf{100}, 91--110.

\bibitem[{Heggernes(2006)}]{heggernes2006minimal}
\textsc{Heggernes, P.} (2006).
\newblock Minimal triangulations of graphs: A survey.
\newblock \textit{Discrete Mathematics} \textbf{306}, 297--317.

\bibitem[{Hotelling(1953)}]{hotelling1953new}
\textsc{Hotelling, H.} (1953).
\newblock New light on the correlation coefficient and its transforms.
\newblock \textit{Journal of the Royal Statistical Society. Series B
  (Methodological)} \textbf{15}, 193--232.

\bibitem[{Johnstone(2010)}]{johnstone2010high}
\textsc{Johnstone, I.~M.} (2010).
\newblock High dimensional bernstein-von mises: simple examples.
\newblock \textit{Institute of Mathematical Statistics collections} \textbf{6},
  87.

\bibitem[{Jones et~al.(2005)Jones, Carvalho, Dobra, Hans, Carter \&
  West}]{jones2005experiments}
\textsc{Jones, B.}, \textsc{Carvalho, C.}, \textsc{Dobra, A.}, \textsc{Hans,
  C.}, \textsc{Carter, C.} \& \textsc{West, M.} (2005).
\newblock Experiments in stochastic computation for high-dimensional graphical
  models.
\newblock \textit{Statistical Science} , 388--400.

\bibitem[{Khare et~al.(2018)Khare, Rajaratnam \& Saha}]{khare2018bayesian}
\textsc{Khare, K.}, \textsc{Rajaratnam, B.} \& \textsc{Saha, A.} (2018).
\newblock Bayesian inference for gaussian graphical models beyond decomposable
  graphs.
\newblock \textit{Journal of the Royal Statistical Society: Series B
  (Statistical Methodology)} \textbf{80}, 727--747.

\bibitem[{Lam \& Fan(2009)}]{lam2009sparsistency}
\textsc{Lam, C.} \& \textsc{Fan, J.} (2009).
\newblock Sparsistency and rates of convergence in large covariance matrix
  estimation.
\newblock \textit{Annals of statistics} \textbf{37}, 4254.

\bibitem[{Lauritzen(1996)}]{lauritzen1996graphical}
\textsc{Lauritzen, S.~L.} (1996).
\newblock \textit{Graphical models}, vol.~17.
\newblock Clarendon Press.

\bibitem[{Letac et~al.(2007)Letac, Massam et~al.}]{letac2007wishart}
\textsc{Letac, G.}, \textsc{Massam, H.} et~al. (2007).
\newblock Wishart distributions for decomposable graphs.
\newblock \textit{The Annals of Statistics} \textbf{35}, 1278--1323.

\bibitem[{Meinshausen et~al.(2006)Meinshausen, B{\"u}hlmann
  et~al.}]{meinshausen2006high}
\textsc{Meinshausen, N.}, \textsc{B{\"u}hlmann, P.} et~al. (2006).
\newblock High-dimensional graphs and variable selection with the lasso.
\newblock \textit{The annals of statistics} \textbf{34}, 1436--1462.

\bibitem[{Moghaddam et~al.(2009)Moghaddam, Khan, Murphy \&
  Marlin}]{moghaddam2009accelerating}
\textsc{Moghaddam, B.}, \textsc{Khan, E.}, \textsc{Murphy, K.~P.} \&
  \textsc{Marlin, B.~M.} (2009).
\newblock Accelerating bayesian structural inference for non-decomposable
  gaussian graphical models.
\newblock In \textit{Advances in Neural Information Processing Systems}.

\bibitem[{Mortici(2010)}]{mortici2010new}
\textsc{Mortici, C.} (2010).
\newblock New approximation formulas for evaluating the ratio of gamma
  functions.
\newblock \textit{Mathematical and Computer Modelling} \textbf{52}, 425--433.

\bibitem[{Rajaratnam et~al.(2008)Rajaratnam, Massam, Carvalho
  et~al.}]{rajaratnam2008flexible}
\textsc{Rajaratnam, B.}, \textsc{Massam, H.}, \textsc{Carvalho, C.~M.} et~al.
  (2008).
\newblock Flexible covariance estimation in graphical gaussian models.
\newblock \textit{The Annals of Statistics} \textbf{36}, 2818--2849.

\bibitem[{Raskutti et~al.(2009)Raskutti, Yu, Wainwright \& Ravikumar}]{rask}
\textsc{Raskutti, G.}, \textsc{Yu, B.}, \textsc{Wainwright, M.~J.} \&
  \textsc{Ravikumar, P.~K.} (2009).
\newblock Model selection in gaussian graphical models: High-dimensional
  consistency of lregularized mle.
\newblock In \textit{Advances in Neural Information Processing Systems}.

\bibitem[{Rose et~al.(1976)Rose, Tarjan \& Lueker}]{rose1976algorithmic}
\textsc{Rose, D.~J.}, \textsc{Tarjan, R.~E.} \& \textsc{Lueker, G.~S.} (1976).
\newblock Algorithmic aspects of vertex elimination on graphs.
\newblock \textit{SIAM Journal on computing} \textbf{5}, 266--283.

\bibitem[{Roverato(2000)}]{roverato2000cholesky}
\textsc{Roverato, A.} (2000).
\newblock Cholesky decomposition of a hyper inverse wishart matrix.
\newblock \textit{Biometrika} \textbf{87}, 99--112.

\bibitem[{Roverato(2002)}]{roverato2002hyper}
\textsc{Roverato, A.} (2002).
\newblock Hyper inverse wishart distribution for non-decomposable graphs and
  its application to bayesian inference for gaussian graphical models.
\newblock \textit{Scandinavian Journal of Statistics} \textbf{29}, 391--411.

\bibitem[{Scott \& Carvalho(2008)}]{scott2008feature}
\textsc{Scott, J.~G.} \& \textsc{Carvalho, C.~M.} (2008).
\newblock Feature-inclusion stochastic search for gaussian graphical models.
\newblock \textit{Journal of Computational and Graphical Statistics}
  \textbf{17}, 790--808.

\bibitem[{Segura(2016)}]{segura2016sharp}
\textsc{Segura, J.} (2016).
\newblock Sharp bounds for cumulative distribution functions.
\newblock \textit{Journal of Mathematical Analysis and Applications}
  \textbf{436}, 748--763.

\bibitem[{Spokoiny(2013)}]{spokoiny2013bernstein}
\textsc{Spokoiny, V.} (2013).
\newblock Bernstein-von mises theorem for growing parameter dimension.
\newblock \textit{arXiv preprint arXiv:1302.3430} .

\bibitem[{Thomas \& Green(2009)}]{thomas2009enumerating}
\textsc{Thomas, A.} \& \textsc{Green, P.~J.} (2009).
\newblock Enumerating the decomposable neighbors of a decomposable graph under
  a simple perturbation scheme.
\newblock \textit{Computational statistics \& data analysis} \textbf{53},
  1232--1238.

\bibitem[{Uhler et~al.(2018)Uhler, Lenkoski, Richards et~al.}]{uhler2018exact}
\textsc{Uhler, C.}, \textsc{Lenkoski, A.}, \textsc{Richards, D.} et~al. (2018).
\newblock Exact formulas for the normalizing constants of wishart distributions
  for graphical models.
\newblock \textit{The Annals of Statistics} \textbf{46}, 90--118.

\bibitem[{Wang et~al.(2010)Wang, Carvalho et~al.}]{wang2010simulation}
\textsc{Wang, H.}, \textsc{Carvalho, C.~M.} et~al. (2010).
\newblock Simulation of hyper-inverse wishart distributions for
  non-decomposable graphs.
\newblock \textit{Electronic Journal of Statistics} \textbf{4}, 1470--1475.

\bibitem[{Watson(1959)}]{watson1959note}
\textsc{Watson, G.} (1959).
\newblock A note on gamma functions.
\newblock \textit{Edinburgh Mathematical Notes} \textbf{42}, 7--9.

\bibitem[{Yuan \& Lin(2007)}]{yuan2007model}
\textsc{Yuan, M.} \& \textsc{Lin, Y.} (2007).
\newblock Model selection and estimation in the gaussian graphical model.
\newblock \textit{Biometrika} \textbf{94}, 19--35.

\end{thebibliography}

\end{document}